\documentclass{amsart}

\usepackage{algo}
\usepackage[latin1]{inputenc}
\usepackage{graphicx}
\usepackage{amssymb,amsmath, amsthm}
\usepackage{color}
\usepackage{verbatim}

\newtheorem{theorem}{Theorem}
\newtheorem{proposition}[theorem]{Proposition}
\newtheorem{lemma}[theorem]{Lemma}
\newtheorem{claim}[theorem]{Claim}

\newtheorem{conjecture}[theorem]{Conjecture}
\newtheorem{corollary}[theorem]{Corollary}

\newcommand{\cacher}[1]{}

\theoremstyle{definition}\newtheorem{definition}[theorem]{Definition}
\theoremstyle{definition}\newtheorem{remark}[theorem]{Remark}

\newcounter{exampleno}
\newenvironment{example}{
\smallbreak\noindent%
{\sc Example.}}%
{\dotfill\hbox{~$\square$}\smallbreak}

\def\alphaLU{\alpha_{L/U}}
\def\alphaUL{\alpha_{U/L}}
\def\pacc{p_{\mathrm{acc}}}
\def\Rb{R_{\bullet}}
\def\Rw{R_{\circ}}

\def\Claim#1.{\medbreak\ni{\bf Claim~#1.}\ }
\def\qedclaim{\hfill$\triangle$\smallskip}

\newcommand\LtoU{\textsc{Lderived$\rightarrow$Uderived}\ }
\newcommand\UtoL{\textsc{Uderived$\rightarrow$Lderived}\ }

\def\ds{\displaystyle}

\DeclareMathOperator{\Pois}{Pois}
\DeclareMathOperator{\Bern}{Bern}
\DeclareMathOperator{\Set}{\textsc{Set}}

\def\cC{\mathcal{C}}
\def\cB{\mathcal{B}}
\def\cA{\mathcal{A}}
\def\cG{\mathcal{G}}
\def\cGc{\mathcal{G}_1}
\def\cGcp{\mathcal{G}_1\ \!\!\!'}
\def\cGb{\mathcal{G}_2}
\def\cGt{\mathcal{G}_3}
\def\cGbp{\mathcal{G}_2\ \!\!\!'}

\def\cZ{\mathcal{Z}}
\def\cD{\mathcal{D}}
\def\cE{\mathcal{E}}
\def\cH{\mathcal{H}}
\def\cK{\mathcal{K}}
\def\cI{\mathcal{I}}
\def\cR{\mathcal{R}}
\def\cP{\mathcal{P}}
\def\cQ{\mathcal{Q}}
\def\cS{\mathcal{S}}

\def\cK{\mathcal{K}}

\def\cF{\mathcal{F}}
\def\cM{\mathcal{M}}
\def\cMt{\mathcal{M}_3}

\def\ol{\overline}
\def\cZU{\mathcal{Z}_U}
\def\cZL{\mathcal{Z}_L}

\def\cMtr{\vec{\cM_3}}
\def\ni{\noindent}

\def\cGtr{\vec{\cG_3}}

\def\cGbr{\vec{\cG_2}}
\def\cJ{\mathcal{J}}
\def\cJa{\mathcal{J}_{\mathrm{a}}}

\def\Ja{J_{\mathrm{a}}}

\def\vec{\widevec}
\def\widevec{\overrightarrow}
\def\ul{\underline}

\def\cRb{\mathcal{R}_{\bullet}}
\def\cRw{\mathcal{R}_{\circ}}
\def\cRbas{\mathcal{R}_{\bullet}^{\mathrm{(as)}}}
\def\cRwas{\mathcal{R}_{\circ}^{\mathrm{(as)}}}
\def\cRbh{\widehat{\mathcal{R}}_{\bullet}}
\def\cRwh{\widehat{\mathcal{R}}_{\circ}}

\def\barpi{\ol{\pi}}

\author{\'Eric Fusy}
\title[Uniform random sampling of planar graphs in linear time]{Uniform random sampling of planar graphs\\ in linear time}
\address{Algorithms project, INRIA Rocquencourt 78153 Le Chesnay Cedex, France}
\email{eric.fusy@inria.fr}

\begin{document}

\begin{abstract}

  {\small This article introduces new algorithms for the uniform 
random generation of
labelled planar graphs. Its principles rely on Boltzmann samplers, as 
recently developed by Duchon, Flajolet, Louchard,
and Schaeffer. It combines the Boltzmann framework, a suitable use of
rejection, a new combinatorial bijection
found by Fusy, Poulalhon and Schaeffer, as well as a precise analytic 
description of the generating functions counting planar graphs, which was recently obtained by 
Gim\'enez and Noy. This gives rise to
an extremely efficient algorithm for the random generation
of planar graphs. There is a preprocessing step of some fixed small cost; and
 the expected time complexity of generation is quadratic for exact-size uniform
sampling and linear for approximate-size sampling. This greatly improves on
the best previously known time complexity for exact-size uniform
sampling of planar graphs with $n$ vertices, which was a little over
$O(n^7)$.

\emph{This is the extended and revised journal version of a conference
paper with the title ``Quadratic exact-size and 
linear approximate-size
random generation of
planar graphs'', which appeared in the Proceedings of the International Conference on  Analysis of Algorithms (AofA'05), 6-10 June 2005, Barcelona.}}
\end{abstract}

\maketitle

\section{Introduction}
\label{sec:intro}
A graph is said to be planar if it can be embedded in the plane so
that no two edges cross each other. In this article, we consider
 planar graphs that are \emph{labelled}, i.e., the $n$ 
vertices
bear distinct labels in $[1..n]$, and \emph{simple}, i.e., with no loop nor multiple edges. 
Statistical properties of planar graphs have been intensively
studied~\cite{BGHPS04,Ge,gimeneznoy}. Very recently,
Gim\'enez and Noy~\cite{gimeneznoy} have solved \emph{exactly} the
difficult problem of the asymptotic enumeration of labelled planar graphs. They
also provide exact analytic expressions for the asymptotic
probability distribution of parameters such as the number of edges
and the number of connected components. However many other statistics on
random planar graphs remain analytically and combinatorially
intractable. Thus, it is an important issue to design efficient random samplers in order to observe the  (asymptotic) behaviour of such parameters on random planar graphs.
Moreover, random generation is useful to test the correctness and efficiency of 
algorithms on planar graphs,
such as planarity testing, embedding algorithms, procedures
for finding geometric cuts, and so on.


Denise,
Vasconcellos, and Welsh have proposed a first algorithm for the random generation of 
planar graphs~\cite{alain96random}, by defining a Markov chain on
the set $\mathcal{G}_n$ of labelled planar graphs with $n$ vertices. 
At each step, two different vertices $v$ and $v'$ are chosen
at random. If they are adjacent, the edge $(v,v')$ is deleted. If they
are not adjacent and if the operation of adding $(v,v')$ does not
break planarity, then the edge $(v,v')$ is added. By symmetry of the transition matrix of the Markov chain, the
probability distribution converges to the uniform distribution on
$\mathcal{G}_n$. This algorithm is
very easy to describe but more difficult to implement, as there
exists
no simple linear-time planarity testing algorithm. More importantly, the rate
of convergence to the uniform distribution is unknown.

A second approach for uniform random generation is 
the \emph{recursive method} introduced by Nijenhuis and
Wilf~\cite{NiWi79}
and formalised by Flajolet, Van Cutsem and
Zimmermann~\cite{FlZiVa94}. The recursive method is a general
framework for the random generation of combinatorial classes admitting a
recursive decomposition. For such classes, producing an object of the class
uniformly at
random boils down to producing the \emph{decomposition tree}
corresponding to its recursive decomposition. Then, the branching
probabilities that produce the decomposition tree with suitable
(uniform) probability are computed
using the \emph{coefficients} counting the objects involved in the
decomposition. As a consequence, this method requires a preprocessing
step where large tables of large coefficients are calculated using the recursive
relations they satisfy.

\begin{figure}
{\small
\begin{tabular}{lllll}
&Aux. mem.&Preproc. time&\multicolumn{2}{c}{Time per generation}\\
\hline
\hline
Markov chains & $O(\log n)$$\phantom{n^{2^2}}$&$O(\log n)$ &\emph{unknown}
&\{exact size\} \\[1mm]
\hline
Recursive method &$O(n^5\log n)$&$O^{*}\!\!\left( n^7  \right) $&$O(n^3)$$\phantom{n^{2^2}}$ &\{exact size\} \\[1mm]
\hline
\raisebox{-0.7ex}{Boltzmann sampler} &\raisebox{-0.7ex}{$O((\log n)^k)$}&\raisebox{-0.7ex}{$O((\log
n)^k)$}& $O(n^2)$ $\phantom{n^{2^2}}$ &\{exact size\}\\[-1mm]
&&&\raisebox{+0.0ex}{$O(n/\epsilon)$} &\{approx. size\}
\end{tabular}
}
\caption{Complexities of the
random samplers of planar graphs ($O^{*}$
stands for a big $O$ taken up to logarithmic factors).}
\label{table:compar}
\end{figure}

Bodirsky \emph{et al.} have described in~\cite{bodirsky} the first polynomial-time random sampler
for planar graphs. Their idea is to apply the recursive method of sampling to 
a well known combinatorial decomposition of planar graphs according to successive
levels of connectivity, which has been formalised by Tutte~\cite{Tut}. 
Precisely, the decomposition yields 
some recurrences satisfied by the
coefficients counting planar graphs as well as subfamilies (connected, 2-connected, 3-connected), which  in turn yield an explicit recursive
way to generate planar graphs uniformly at random. As the recurrences are rather involved,
the complexity of the preprocessing step is large. Precisely, in order to draw planar graphs with $n$ vertices
(and possibly also a fixed number $m$ of edges),  
the random generator described in~\cite{bodirsky} requires a
preprocessing time of order
$O\left( n^7 (\log n)^2(\log \log n )   \right) $ and
an auxiliary memory of size
$O( n^5 \log n)$. Once the tables have been computed, the complexity
 of each generation is $O(n^3)$. A more recent optimisation of the
 recursive method by Denise and 
 Zimmermann~\cite{denise99uniform} ---based on controlled real arithmetics--- should be applicable;
it would improve the time complexity somewhat, but the storage complexity  would still be large.

In this article, we introduce a new random generator for 
labelled planar graphs, which relies on the same  
decomposition of planar graphs as the algorithm of Bodirsky \emph{et al}. The main 
difference is that we translate this 
decomposition into a random generator using the framework of Boltzmann samplers,
instead of the recursive method. 
Boltzmann samplers have been recently
developed by Duchon, Flajolet, Louchard, and Schaeffer
in~\cite{DuFlLoSc04} as a powerful framework for the random generation
of decomposable combinatorial structures. The idea of Boltzmann sampling is to gain efficiency by
 relaxing the
constraint of exact-size sampling. As we will see, the gain is particularly significant in 
the case of planar graphs, where the decomposition is more involved than for classical classes, 
such as trees. 
Given a combinatorial class, a \emph{Boltzmann sampler} draws an object of size $n$
with probability proportional to $x^n$ (or proportional to
$x^n/n!$ for labelled objects), where $x$ is a certain
\emph{real} parameter that can be appropriately tuned. Accordingly, 
the probability distribution is spread over all the objects
of the class, with the property that objects of the same size have the same
probability of occurring. In particular, the probability distribution is uniform
when restricted to a fixed size. Like the recursive method, Boltzmann
samplers can be designed for any combinatorial class admitting a
recursive decomposition, as there are explicit sampling rules associated with
each classical construction (Sum, Product, Set, Substitution). The branching probabilities used
to produce 
the decomposition tree of a random object are not based on the
\emph{coefficients} as in the
recursive method, but on the \emph{values} at $x$ of the generating
functions of 
the classes intervening in the decomposition.

In this article, we translate the decomposition of planar graphs into Boltzmann samplers and obtain
 very efficient random generators that  produce
 planar graphs with a fixed number of vertices or with fixed numbers
of vertices and edges uniformly at
random. Furthermore, our samplers have an approximate-size version where a small tolerance, 
say a few percents, is allowed for the size of the output. For
practical 
purpose, approximate-size random sampling
often suffices. The approximate-size samplers we propose are very
efficient as they have \emph{linear time complexity}. 

\begin{theorem}[Samplers with respect to number of vertices]
\label{theo:planarsamp1}
Let $n\in \mathbf{N}$ be a target size.
An \emph{exact-size} 
sampler $\frak{A}_n$ can be designed so as to generate 
labelled planar graphs with $n$ vertices uniformly at
random. For any tolerance ratio $\epsilon>0$, an
\emph{approximate-size} sampler  $\frak{A}_{n,\epsilon}$ can be designed so as to generate planar
graphs with their number of vertices in $[n(1-\epsilon),n(1+\epsilon)]$, and following the 
 uniform distribution for each size $k\in
[n(1-\epsilon),n(1+\epsilon)]$.  

Under a real-arithmetics complexity model, 
Algorithm $\frak{A}_n$ is of expected complexity $O(n^2)$, and  
Algorithm $\frak{A}_{n,\epsilon}$ is of expected complexity $O(n/\epsilon)$.

\end{theorem}

\begin{theorem}[Samplers with respect to the numbers of vertices and edges]
\label{theo:planarsamp2}
Let $n\in \mathbf{N}$ be a target size and $\mu\in(1,3)$ be a parameter describing
the ratio edges-vertices.
An \emph{exact-size}
sampler $\ol{\frak{A}}_{n,\mu}$ can be designed so as to generate planar graphs with $n$ vertices and $\lfloor \mu n\rfloor$ edges uniformly at
random. For any tolerance-ratio $\epsilon>0$, an
\emph{approximate-size} sampler  $\ol{\frak{A}}_{n,\mu,\epsilon}$ can be designed so as to generate  planar
graphs with their number of vertices in $[n(1-\epsilon),n(1+\epsilon)]$ and their ratio edges/vertices in 
$[\mu (1-\epsilon),\mu (1+\epsilon)]$, and following the 
 uniform distribution for each fixed pair (number of vertices,
number of edges).

Under a real-arithmetics complexity model, 
for a fixed $\mu\in(1,3)$, Algorithm $\ol{\frak{A}}_{n,\mu}$ is of expected complexity $O_{\mu}(n^{5/2})$. For fixed constants $\mu\in(1,3)$ and $\epsilon>0$, 
Algorithm $\ol{\frak{A}}_{n,\mu,\epsilon}$ is of expected complexity $O_{\mu}(n/\epsilon)$
(the bounding constants depend on $\mu$).

\end{theorem}
\noindent The samplers are completely described in Section~\ref{sec:sample_vertices} 
and Section~\ref{sec:sample_edges}. The expected complexities will be proved in Section~\ref{sec:complexity}. For the sake of simplicity, we give big $O$ bounds that might depend on $\mu$ and
we do not care about quantifying the constant in the big $O$ in a precise way. 
However we strongly believe that a more careful analysis would allow us to have 
a uniform bounding constant (over $\mu\in(1,3)$) of reasonable magnitude.
This means that not only the theoretical complexity is good but also the practical one.
(As we review in Section~\ref{sec:implement}, 
we have implemented the algorithm, which easily draws  
graphs of sizes in the range of $10^5$.)

\vspace{0.2cm}

\emph{Complexity model.} Let us comment on the model we adopt to state the complexities of the random samplers.
We assume here that we are given an \emph{oracle}, which provides at unit cost the exact evaluations of the generating
functions intervening in the decomposition of planar graphs. (For planar graphs, these 
generating functions are those of families of planar graphs of different  
connectivity degrees and pointed in 
different ways.)
This assumption, called the ``oracle assumption", is by now 
classical to analyse the complexity of Boltzmann samplers, see~\cite{DuFlLoSc04} for a more detailed discussion;
it allows us to separate the \emph{combinatorial complexity} of the samplers from the 
complexity of \emph{evaluating} the generating functions, which resorts to computer algebra and is a 
research project on its own. 
Once the oracle assumption is done, the scenario of generation of a Boltzmann sampler is typically similar to a branching
process;  the generation follows a sequence of \emph{random choices} ---typically coin flips biased by 
some generating function values--- that determine the 
shape of the object to be drawn. According to these choices, the object (in this article, a planar graph) is built effectively 
by a  sequence of primitive operations such as vertex creation, edge creation, merging two graphs at a common vertex... 
The \emph{combinatorial complexity} is precisely defined as the sum of the number of coin flips and the number
of primitive operations performed to build the object.
The (combinatorial) complexity of our algorithm is compared to the complexities of the
two preceding random samplers in Figure~\ref{table:compar}.

Let us now comment on the preprocessing complexity.  
The implementation of
$\frak{A}_{n,\epsilon}$ and $\frak{A}_n$, as well as $\ol{\frak{A}}_{n,\mu,\epsilon}$ and
$\ol{\frak{A}}_{n,\mu}$, requires the
storage of a fixed number of real constants, which are special values
of generating functions. The generating functions to  be evaluated
are those of several families of planar graphs (connected,
2-connected, 3-connected). A crucial result, recently established
by Gim\'enez and Noy~\cite{gimeneznoy}, is that there exist
exact analytic equations satisfied by these generating functions. Hence,
their numerical evaluation can be performed efficiently with the help of a 
computer algebra system; the complexity we have observed in practice
(doing the computations with Maple) is  of low
polynomial degree $k$ in the number of digits that need to be computed.
(However, there is not yet a complete rigorous proof of the fact, as the Boltzmann parameter
has to approach the singularity in order to draw planar graphs of large size.)
To draw objects of size $n$, the precision needed to 
make the probability of failure small 
is typically of order $\log(n)$ digits\footnote{Notice that it is possible to 
 achieve perfect uniformity by calling 
adaptive precision routines in case of failure, 
see Denise and Zimmermann~\cite{denise99uniform} for a detailed discussion on similar problems.}.  Thus the preprocessing step to
evaluate the generating functions with a precision of $\log(n)$ digits
 has a complexity
of order $\log(n)^k$ 
(again, this is yet to be proved rigorously).
The following informal statement summarizes the discussion; making a theorem of it is the subject
of ongoing research (see the recent article~\cite{PiSaSo07}):  
\smallskip

\noindent{\bf Fact.}
\emph{With high probability, the auxiliary memory necessary to generate planar graphs of size  $n$
is of order $O(\log(n))$ and  the preprocessing time complexity is of order $O(\log(n)^k)$
for some low integer $k$.}

\smallskip

\emph{Implementation and experimental results.} We have completely implemented the 
random samplers stated in Theorem~\ref{theo:planarsamp1} and Theorem~\ref{theo:planarsamp2}.
 Details are given in Section~\ref{sec:implement}, as well as experimental results. 
Precisely, 
the evaluations of the generating functions of planar graphs have been
carried out with the computer algebra system Maple, based on the
analytic expressions given by Gim\'enez and
Noy~\cite{gimeneznoy}. Then, the random generator has been implemented in
Java, with a precision of 64 bits for the values of generating
functions (``double'' type). Using the approximate-size sampler, planar
graphs with size of order 100,000 are generated in a few seconds with
a machine clocked at 1GHz. In contrast, the recursive method of Bodirsky \emph{et
al} is currently limited to sizes of about 100.

Having the random generator implemented, we have performed some simulations
in order to observe typical properties of random planar graphs.
In particular we have observed a sharp concentration for 
the proportion of vertices of a given degree $k$ in a random planar graph of large size.


\section{Overview}

The algorithm we describe relies mainly on two ingredients. 
The first one is  a recent correspondence, called the 
closure-mapping, 
between binary trees and (edge-rooted) 3-connected planar graphs~\cite{FuPoSc05}, which makes
it possible to obtain a Boltzmann sampler for 3-connected planar graphs. 
The second one is a decomposition  formalised by Tutte~\cite{Tut}, which ensures that 
any planar graph can be decomposed into 3-connected
components, via connected and 2-connected components.
Taking advantage of Tutte's decomposition, we explain in Section~\ref{sec:decomp} how to specify
a Boltzmann sampler for planar graphs, denoted $\Gamma\cG(x,y)$, 
from the Boltzmann sampler for 3-connected
planar graphs. To do this, we 
have to extend the collection of constructions for Boltzmann samplers,
as detailed in~\cite{DuFlLoSc04}, 
and develop new rejection techniques so as to suitably handle the 
rooting/unrooting operations that appear alongside Tutte's decomposition.

Even if the Boltzmann
sampler  $\Gamma\cG(x,y)$ already yields a polynomial-time
uniform random sampler for planar graphs, the expected time complexity
to generate a graph of size $n$ ($n$ vertices) is not good, due to the fact that the size distribution of $\Gamma \cG(x,y)$ is too concentrated 
on objects of small
size. To improve the size distribution, we \emph{point} the objects,
in a way inspired by~\cite{DuFlLoSc04}, which corresponds
to a \emph{derivation} (differentiation) of the associated generating function. 
The precise singularity
analysis of the generating functions of planar graphs, which has been recently done 
in~\cite{gimeneznoy}, indicates that we have to take the second derivative of planar
graphs in order to get a good size distribution. 
In Section~\ref{sec:efficient}, we
explain how the derivation operator can be injected in the
decomposition of planar graphs. This yields a Boltzmann sampler $\Gamma
\cG''(x,y)$ for ``bi-derived'' planar
graphs. Our random generators for planar graphs are finally obtained as \emph{targetted samplers}, 
which call $\Gamma \cG''(x,y)$ (with suitably tuned
values of $x$ and $y$) until the generated graph has the desired size.  
The time complexity of the targetted samplers 
is analysed in Section~\ref{sec:complexity}. This eventually 
yields the complexity results 
stated in Theorems~\ref{theo:planarsamp1} and ~\ref{theo:planarsamp2}. The general scheme
of the planar graph generator is shown in Figure~\ref{fig:relations}.

\begin{figure}
  \begin{center}
  \includegraphics{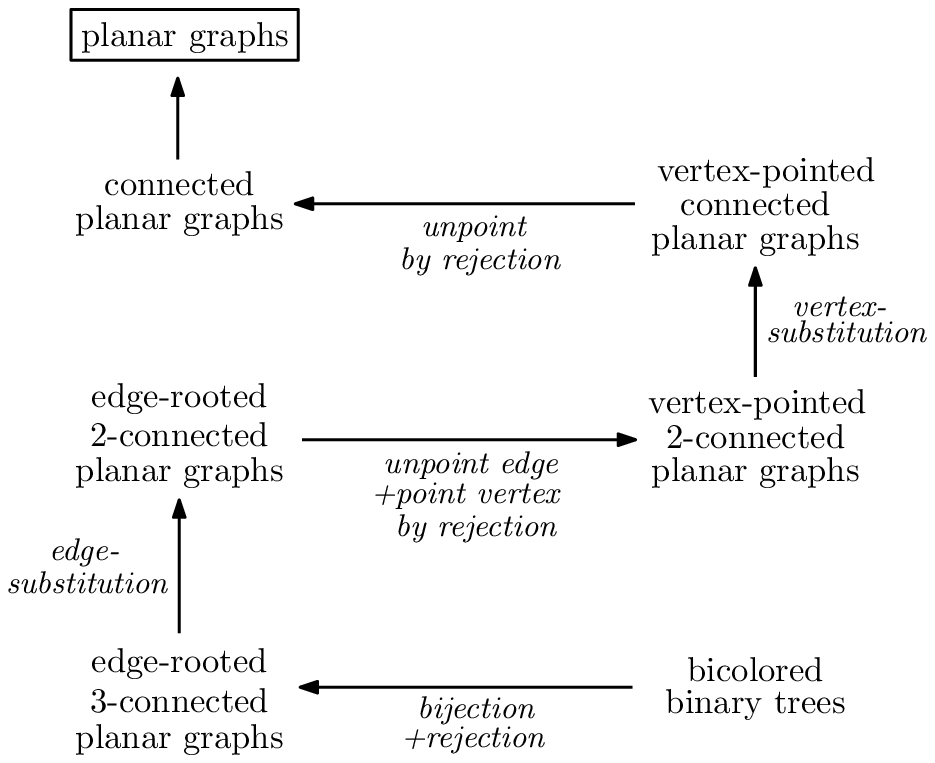}
    \caption{The chain of constructions from binary trees to planar graphs.}\label{fig:relations}
  \end{center}
\end{figure}

\section{Boltzmann samplers}
\label{sec:bolz}
In this section, we define Boltzmann samplers and describe the main properties which we will need
to handle planar graphs.
In particular, we have to extend the framework to the case of \emph{mixed classes}, meaning that 
the objects have two types of atoms. 
Indeed the decomposition of planar graphs involves both (labelled)
vertices and (unlabelled) edges. The constructions needed to formulate the decomposition of
planar graphs are classical ones in combinatorics: Sum, Product, Set, Substitutions~\cite{BeLaLe,fla}. 
In Section~\ref{sec:rule}, for each of
the constructions, we describe a \emph{sampling rule}, so that Boltzmann samplers can be 
assembled for any class that admits a decomposition in terms of these constructions.
Moreover, the decomposition of planar graphs involves rooting/unrooting operations,
which makes it necessary to develop new rejection techniques, 
as described in Section~\ref{sec:reject}.

\subsection{Definitions}
\label{sec:bolzdef}
A combinatorial class $\cC$ is a family of labelled objects (structures), that is, 
each object is made
of $n$ atoms that bear distinct labels in $[1..n]$. In addition,
the number of objects in any fixed size $n$ is finite; and any structure obtained
by relabelling a structure in $\cC$ is also in $\cC$. The
\emph{exponential}
generating function of $\mathcal{C}$ is defined as
$$C(x):=\sum_{\gamma\in\mathcal{C}}\frac{x^{|\gamma|}}{|\gamma|!},$$
where $|\gamma|$ is the size of an object $\gamma\in\cC$ (e.g., the
number of vertices of a graph). The radius of convergence of $C(x)$ is denoted
by $\rho$. A positive value $x$ is called \emph{admissible} if $x\in(0,\rho)$
(hence the sum defining $C(x)$ converges if $x$ is admissible).

Boltzmann samplers, as introduced and developed by Duchon \emph{et al.} 
in~\cite{DuFlLoSc04}, constitute a
general and efficient framework to produce a random generator
for any \emph{decomposable} combinatorial class $\mathcal{C}$.  
Instead of fixing a particular size for the random generation, objects
are drawn under a 
probability distribution spread over the whole class. 
Precisely, given an admissible value for $C(x)$, 
  the Boltzmann distribution assigns to each object of
$\mathcal{C}$ a weight
$$\mathbf{P}_x(\gamma)=\frac{x^{|\gamma|}}{|\gamma|!C(x)}\,
.$$
Notice that the distribution is uniform, i.e., two objects with the same 
size have the same probability to be chosen. A \emph{Boltzmann sampler} for the labelled class $\mathcal{C}$ is a procedure $\Gamma \cC(x)$
that, for each fixed admissible $x$, draws objects of $\mathcal{C}$ at random under the distribution $\mathbf{P}_x$.
 The authors
of~\cite{DuFlLoSc04} give sampling rules associated to classical combinatorial
constructions, such as Sum, Product, and Set. 
(For the unlabelled setting, we refer to the more recent article~\cite{FlFuPi07},
and to~\cite{BoFuPi06} for the specific case of plane partitions.)

In order to translate the
combinatorial decomposition of planar graphs into a Boltzmann sampler, 
we need to extend the framework
 of Boltzmann samplers to the bivariate case of  \emph{mixed} combinatorial
classes. A mixed class $\cC$ is a labelled combinatorial class where one takes into
account a second type of atoms, which are unlabelled. Precisely, an object in 
$\mathcal{C}=\cup_{n,m}\mathcal{C}_{n,m}$ has $n$ ``labelled
atoms'' and $m$ ``unlabelled atoms'', e.g., a graph has $n$ labelled
vertices and $m$ unlabelled edges. The labelled atoms are shortly called L-atoms, and 
the unlabelled atoms are shortly called U-atoms.
For $\gamma\in\mathcal{C}$, we write
$|\gamma|$ for the number of L-atoms of $\gamma$, called the \emph{L-size} of $\gamma$, and
$||\gamma||$ for the number of U-atoms of $\gamma$, called the \emph{U-size} of $\gamma$. 
The associated generating function $C(x,y)$ is
defined as
$$C(x,y):=\sum_{\gamma\in\mathcal{C}}\frac{x^{|\gamma|}}{|\gamma|!}y^{||\gamma||}.$$
For a fixed real value $y>0$, we denote by
$\rho_C(y)$ the radius of convergence of the function $x\mapsto C(x,y)$. A
pair $(x,y)$ is said to be \emph{admissible} if $x\in (0,\rho_C(y))$, which 
implies that $\sum_{\gamma\in\mathcal{C}}\frac{x^{|\gamma|}}{|\gamma|!}y^{||\gamma||}$ converges and that
$C(x,y)$ is well defined.  Given an admissible pair $(x,y)$, the
\emph{mixed Boltzmann distribution} is the probability distribution
$\mathbf{P}_{x,y}$ 
assigning to each
object $\gamma\in\mathcal{C}$ the probability
$$\mathbf{P}_{x,y}(\gamma)=\frac{1}{C(x,y)}\frac{x^{|\gamma|}}{|\gamma|!}y^{||\gamma||}.$$
 An important property of this distribution is that two objects with the same
size-parameters  have the same probability of occurring. 
 A \emph{mixed Boltzmann sampler}
at $(x,y)$ ---shortly called Boltzmann sampler hereafter--- is a procedure $\Gamma \cC(x,y)$
that draws objects of $\mathcal{C}$ at random under the 
distribution $\mathbf{P}_{x,y}$. Notice that the specialization $y=1$ yields a
classical 
Boltzmann sampler for $\mathcal{C}$.




\subsection{Basic classes and constructions}
\label{sec:rule}

We describe here a collection of basic classes and constructions that are used thereafter to 
formulate a decomposition for the family of planar graphs.

The basic classes we consider are: 
\begin{itemize}
\item
The 1-class, made of a unique object of size 0 (both the L-size and the U-size are equal to 0), 
called the 0-atom. The corresponding
mixed generating function is $C(x,y)=1$.
\item
The L-unit class, made of a unique object that is an L-atom; the corresponding 
mixed generating function is $C(x,y)=x$.
\item
The U-unit class, made of a unique object that is a U-atom; the corresponding 
mixed generating function is $C(x,y)=y$.
\end{itemize}

Let us now describe the five constructions that are used to
decompose planar graphs. In particular, we need two specific
substitution constructions, one at labelled atoms that is called L-substitution, the other
at unlabelled atoms that is called U-substitution. 
\vspace{0.3cm}

\noindent{\bf Sum.} The sum $\cC:=\mathcal{A}+\mathcal{B}$ of two classes
is meant as a \emph{disjoint union}, i.e., it is the union of two
distinct copies of $\mathcal{A}$ and $\mathcal{B}$. The generating function of $\cC$ satisfies
$$
C(x,y)=A(x,y)+B(x,y).
$$

\vspace{0.3cm}

\noindent{\bf Product.} The partitional product of two classes $\cA$ and $\cB$ is the class 
$\cC:=\mathcal{A}\star\mathcal{B}$ of objects that are obtained by taking a pair $\gamma=(\gamma_1\in\cA,\gamma_2\in\cB)$ and relabelling the L-atoms so that $\gamma$ bears distinct labels in $[1..|\gamma|]$. 
The generating function of $\cC$ satisfies
$$
C(x,y)=A(x,y)\cdot B(x,y).
$$
 
\vspace{0.3cm} 
 
\noindent{$\mathbf{Set_{\geq d}}$.} For $d\geq 0$ and a class $\cB$ having no object
of size 0, any object in $\cC:=\Set_{\geq d}(\cB)$ is  a finite set of at least $d$ objects
of $\cB$, relabelled so that the atoms of $\gamma$ bear distinct labels in $[1\,.\,.\,|\gamma|]$.
For $d=0$, this corresponds to the classical construction $\Set$.
The generating function of $\cC$ satisfies
$$
C(x,y)=\exp_{\geq d}(B(x,y)),\ \ \ \mathrm{where}\ \exp_{\geq d}(z):=\sum_{k\geq d}\frac{z^k}{k!}.
$$

\vspace{0.3cm} 

\noindent{\bf L-substitution.} Given $\cA$ and $\cB$ two classes such that 
$\cB$ has no object of size $0$, the class $\cC=\mathcal{A}\circ_L\mathcal{B}$ is the class
of objects that are obtained as follows: take an object 
$\rho\in\mathcal{A}$ called the \emph{core-object}, substitute each L-atom $v$ of $\rho$ by
an object $\gamma_v\in\mathcal{B}$, and  
relabel the L-atoms of $\cup_{v}\gamma_v$ with distinct labels
from $1$ to $\sum_v |\gamma_v|$. The generating function of $\cC$ satisfies
$$
C(x,y)=A(B(x,y),y).
$$

\vspace{0.3cm} 

\noindent{\bf U-substitution.} Given $\cA$ and $\cB$ two classes such that 
$\cB$ has no object of size $0$, the class $\cC=\mathcal{A}\circ_U\mathcal{B}$ is the class
of objects that are obtained as follows: take an object
$\rho\in\mathcal{A}$ called the \emph{core-object}, substitute each U-atom $e$ of $\rho$ by
an object $\gamma_e\in\mathcal{B}$, and 
relabel the L-atoms of $\rho\cup\left(\cup_{e}\gamma_e\right)$ with 
distinct labels
from $1$ to $|\rho|+\sum_e |\gamma_e|$.  We assume here that the U-atoms of an object of
$\cA$ are \emph{distinguishable}. In particular, 
this property is satisfied if $\cA$ is a family of
labelled graphs with no multiple edges, since two different edges are distinguished by the labels of their extremities. The generating function of $\cC$ satisfies
$$
C(x,y)=A(x,B(x,y)).
$$

\begin{figure}
\begin{center}
\includegraphics[width=13cm]{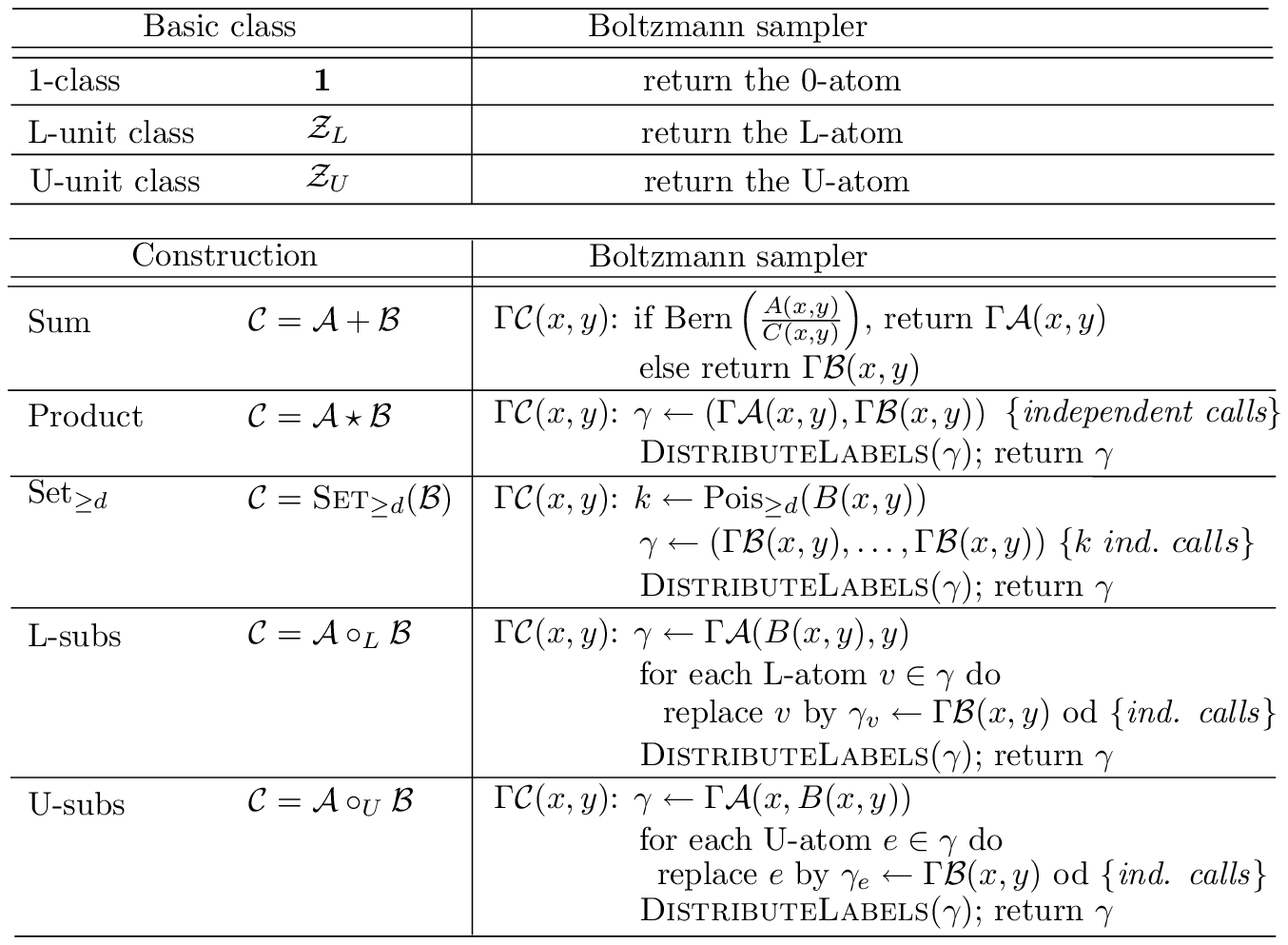}
\end{center}
\caption{The sampling rules associated with the basic classes and the constructions. For each rule involving
partitional products, there is a relabelling step  performed  by an 
auxiliary procedure
\textsc{DistributeLabels}. Given an object $\gamma$ with its 
L-atoms ranked from $1$ to $|\gamma|$,
\textsc{DistributeLabels}($\gamma$) draws a permutation 
$\sigma$ of $[1..|\gamma|]$ uniformly at random and gives label
$\sigma(i)$ to the atom of rank $i$.  }
\label{table:rules}
\end{figure}

\subsection{Sampling rules}
\label{sec:rules_boltzma}
A nice feature of Boltzmann samplers is that the basic
combinatorial constructions (Sum, Product, Set) give rise 
to simple rules for assembling the associated
Boltzmann samplers. To describe these rules, we assume that the exact
values of the
 generating functions at a given admissible pair $(x,y)$ are known. We will
also need two well-known probability distributions.
\begin{itemize}
\item
A random variable follows a \emph{Bernoulli law} of 
parameter $p\in (0,1)$
if it is equal to 1 (or true) with probability $p$ and equal to 0 (or false) 
with probability $1-p$.
\item
Given $\lambda\in\mathbb{R}_{+}$ and $d\in\mathbb{Z}_{+}$, the \emph{conditioned Poisson law}
$\Pois_{\geq d}(\lambda)$ is the probability distribution on $\mathbf{Z}_{\geq d}$ defined as follows:
$$
\mathbb{P}(k)=\frac{1}{\exp_{\geq d}(\lambda)}\frac{\lambda^k}{k!},\ \mathrm{where}\ \exp_{\geq d}(z):=\sum_{k\geq d}\frac{z^k}{k!}.
$$ 
For $d=0$, this corresponds to the classical Poisson law, abbreviated as $\Pois$.
\end{itemize}

Starting from combinatorial classes $\mathcal{A}$ and $\mathcal{B}$
endowed with Boltzmann samplers $\Gamma \cA(x,y)$ and
$\Gamma \cB(x,y)$, Figure~\ref{table:rules} describes how to assemble a
sampler for a class $\mathcal{C}$ obtained from $\mathcal{A}$ and
$\mathcal{B}$ (or from $\cA$ alone for the construction $\Set_{\geq d}$)
 using the five constructions described in this section.

\begin{proposition}
\label{prop:rules}
Let $\cA$ and $\cB$ be two mixed combinatorial classes
endowed with Boltzmann samplers $\Gamma \cA(x,y)$ and $\Gamma \cB(x,y)$.
 For each of the five constructions
$\{+$, $\star$, $\Set_{\geq d}$, L-subs, U-subs$\}$, the sampler
$\Gamma \cC(x,y)$, as specified in Figure~\ref{table:rules}, is a valid
Boltzmann sampler
 for the combinatorial class $\mathcal{C}$.
\end{proposition}
\begin{proof}
1) \emph{Sum:} $\mathcal{C}=\mathcal{A}+\mathcal{B}$. An object
of $\mathcal{A}$ has probability
$\frac{1}{A(x,y)}\frac{x^{|\gamma|}}{|\gamma|!}y^{||\gamma||}$ 
(by definition of $\Gamma \cA(x,y)$)
multiplied by  $\frac{A(x,y)}{C(x,y)}$ (because of the Bernoulli choice)  
of being drawn by
$\Gamma \cC(x,y)$. 
Hence, it has
probability $\frac{1}{C(x,y)}\frac{x^{|\gamma|}}{|\gamma|!}y^{||\gamma||}$ 
of being
drawn. Similarly, 
an object of
$\mathcal{B}$ has probability $\frac{1}{C(x,y)}\frac{x^{|\gamma|}}{|\gamma|!}y^{||\gamma||}$ of
being drawn. 
Hence $\Gamma \cC(x,y)$ is
a valid Boltzmann sampler for $\mathcal{C}$.

\vspace{0.2cm}

\noindent 2) \emph{Product:} 
$\mathcal{C}=\mathcal{A}\star\mathcal{B}$. Define a 
\emph{generation scenario} as a pair
$(\gamma_1\in\mathcal{A},\gamma_2\in\mathcal{B})$, together with a
function $\sigma$ that assigns to each L-atom in $\gamma_1\cup\gamma_2$
a label $i\in[1..|\gamma_1|+|\gamma_2|]$ in a bijective way.
By definition, $\Gamma \cC(x,y)$ draws a generation scenario and returns
the object $\gamma\in\cA\star\cB$ obtained by keeping the secondary labels (the ones
given by \textsc{DistributeLabels}).
 Each generation scenario has
probability
$$\left(\frac{1}{A(x,y)}\frac{x^{|\gamma_1|}}{|\gamma_1|!}y^{||\gamma_1||}\right)\left(\frac{1}{B(x,y)}\frac{x^{|\gamma_2|}}{|\gamma_2|!}y^{||\gamma_2||}\right)\frac{1}{(|\gamma_1|+|\gamma_2|)!}$$
of being drawn, the three factors corresponding respectively to
$\Gamma \cA(x,y)$, $\Gamma \cB(x,y)$, and
\textsc{DistributeLabels}($\gamma$). Observe that this
probability has the more compact form
$$
\frac{1}{|\gamma_1|!|\gamma_2|!}\frac{1}{C(x,y)}\frac{x^{|\gamma|}}{|\gamma|!}y^{||\gamma||}
.$$ 
Given
$\gamma\in\mathcal{A}\star\mathcal{B}$, let $\gamma_1$ be its first
component (in $\mathcal{A}$) and $\gamma_2$ be its second component
(in $\mathcal{B}$). Any relabelling of the labelled atoms of $\gamma_1$ from $1$
to $|\gamma_1|$ and of the labelled atoms of $\gamma_2$ from $1$ to
$|\gamma_2|$ induces a unique generation scenario producing
$\gamma$. Indeed, the two relabellings determine unambiguously the relabelling
permutation $\sigma$ of the generation scenario. Hence, $\gamma$ is
produced from $|\gamma_1|!|\gamma_2|!$ different scenarios, 
each having probability
$\frac{1}{|\gamma_1|!|\gamma_2|!C(x,y)}\frac{x^{|\gamma|}}{|\gamma|!}y^{||\gamma||}$.
As a consequence, $\gamma$ is drawn under the Boltzmann distribution.

\vspace{0.2cm}

\noindent 3) \emph{Set}$_{\geq d}$: $\mathcal{C}=\Set_{\geq d}(\cB)$. 
In the case of the construction $\Set_{\geq d}$, 
a \emph{generation scenario}
is defined as a sequence
$(\gamma_1\in\mathcal{B},\ldots,\gamma_k\in\mathcal{B})$ with $k\geq d$, together with a
function $\sigma$ that assigns to each L-atom in $\gamma_1\cup\cdots\cup\gamma_k$
a label $i\in[1..|\gamma_1|+\cdots+|\gamma_k|]$ in a bijective way. Such a generation
scenario produces an object $\gamma\in\Set_{\geq d}(\mathcal{B})$. By
definition of $\Gamma \cC(x,y)$, each 
scenario has probability 
$$\left(
\frac{1}{\exp_{\geq d}(B(x,y))}\frac{B(x,y)^k}{k!}\right)\left(\prod_{i=1}^k
\frac{x^{|\gamma_i|}y^{||\gamma_i||}}{B(x,y)|\gamma_i|!}\right)\frac{1}{(|\gamma_1|+\cdots+|\gamma_k|)!},$$
the three factors corresponding respectively to drawing $\Pois_{\geq d}(B(x,y))$, 
drawing the sequence, and the relabelling step. This probability has
the simpler form
$$\frac{1}{k!C(x,y)}\frac{x^{|\gamma|}}{|\gamma|!}y^{||\gamma||}\prod_{i=1}^k\frac{1}{|\gamma_i|!}.$$
For $k\geq d$, an object $\gamma\in\Set_{\geq d}(\mathcal{B})$ can be written as a sequence
$\gamma_1,\ldots,\gamma_k$ in $k!$ different ways. In addition, by a
similar argument as for the Product construction, a sequence
$\gamma_1,\ldots,\gamma_k$ is produced from $\prod_{i=1}^k|\gamma_i|!$
different scenarios. As a consequence, $\gamma$ is drawn under the
Boltzmann distribution.

\vspace{0.2cm}

\noindent 4) \emph{L-substitution}: 
$\mathcal{C}=\mathcal{A}\circ_L\mathcal{B}$. For this
construction, a \emph{generation scenario} is defined as a core-object
$\rho\in\mathcal{A}$, a sequence $\gamma_1,\ldots,\gamma_{|\rho|}$ of
objects of $\mathcal{B}$ ($\gamma_i$
stands for the object of $\mathcal{B}$ substituted at the atom $i$ of $\rho$), together with a
function $\sigma$ that assigns to each L-atom in $\gamma_1\cup\cdots\cup\gamma_{|\rho|}$
a label $i\in[1..|\gamma_1|+\cdots+|\gamma_{|\rho|}|]$ in a bijective way.
This corresponds to the scenario of generation of an object
$\gamma\in\mathcal{A}\circ_L\mathcal{B}$ by the algorithm 
$\Gamma \cC(x,y)$, and this 
scenario has probability
$$\left(\frac{1}{A(B(x,y),y)}\frac{B(x,y)^{|\rho|}}{|\rho|!}y^{||\rho||}\right)\left(\prod_{i=1}^{|\rho|}\frac{x^{|\gamma_i|}y^{||\gamma_i||}}{B(x,y)|\gamma_i|!}\right)\frac{1}{(|\gamma_1|+\cdots+|\gamma_{|\rho|}|)!},$$
which has the simpler form
$$\frac{x^{|\gamma|}y^{||\gamma||}}{C(x,y)|\gamma|!}\frac{1}{|\rho|!}\prod_{i=1}^{|\rho|}\frac{1}{|\gamma_i|!}.$$
Given $\gamma\in\mathcal{A}\circ_L\mathcal{B}$, labelling
 the core-object $\rho\in\mathcal{A}$ with distinct labels in $[1..|\rho|]$ and each component
 $(\gamma_i)_{1\leq i\leq|\rho|}$ with distinct labels in $[1..|\gamma_i|]$  
 induces a unique generation scenario
producing $\gamma$. As a consequence, $\gamma$ is produced from
$|\rho|!\prod_{i=1}^{|\rho|}|\gamma_i|!$ scenarios, each having
probability
$\frac{x^{|\gamma|}y^{||\gamma||}}{C(x,y)|\gamma|!}\frac{1}{|\rho|!}\prod_{i=1}^{|\rho|}\frac{1}{|\gamma_i|!}$.
Hence, $\gamma$ is drawn under the Boltzmann distribution.

\vspace{0.2cm}

\noindent 5) \emph{U-substitution}: 
$\mathcal{C}=\mathcal{A}\circ_U\mathcal{B}$. A \emph{generation scenario} 
is defined as a core-object
$\rho\in\mathcal{A}$, a sequence $\gamma_1,\ldots,\gamma_{||\rho||}$ of
objects of $\mathcal{B}$ (upon giving a rank to each unlabelled atom of $\rho$,
$\gamma_i$
stands for the object of $\mathcal{B}$ substituted at the U-atom of rank $i$ in $\rho$),
and a function $\sigma$ that assigns to each L-atom in $\rho\cup\gamma_1\cup\cdots\cup\gamma_{||\rho||}$ a label $i\in[1..|\rho|+|\gamma_1|+\cdots+|\gamma_{||\rho||}|]$. 
This corresponds to the scenario of generation of an object
$\gamma\in\mathcal{A}\circ_U\mathcal{B}$ by the algorithm 
$\Gamma \cC(x,y)$; this 
scenario has probability
$$\left(\frac{1}{A(x,B(x,y))}\frac{x^{|\rho|}}{|\rho|!}B(x,y)^{||\rho||}\right)\left(\prod_{i=1}^{||\rho||}\frac{x^{|\gamma_i|}y^{||\gamma_i||}}{B(x,y)|\gamma_i|!}\right)\left(\frac{1}{(|\rho|+|\gamma_1|+\cdots+|\gamma_{||\rho||}|)!}\right).$$
This expression has the simpler form
$$\frac{x^{|\gamma|}y^{||\gamma||}}{C(x,y)|\gamma|!}\frac{1}{|\rho|!}\prod_{i=1}^{||\rho||}\frac{1}{|\gamma_i|!}.$$
Given $\gamma\in\mathcal{A}\circ_U\mathcal{B}$, labelling  
the core-object $\rho\in\mathcal{A}$ with distinct labels in $[1..|\rho|]$ and each component
 $(\gamma_i)_{1\leq i\leq||\rho||}$ with distinct labels in $[1..|\gamma_i|]$ induces a unique generation scenario
producing $\gamma$. As a consequence, $\gamma$ is produced from
$|\rho|!\prod_{i=1}^{||\rho||}|\gamma_i|!$ scenarios, each having
probability
$\frac{x^{|\gamma|}y^{||\gamma||}}{C(x,y)|\gamma|!}\frac{1}{|\rho|!}\prod_{i=1}^{||\rho||}\frac{1}{|\gamma_i|!}$.
Hence, $\gamma$ is drawn under the Boltzmann distribution.
\end{proof}

\vspace{0.2cm}

\begin{example}\label{ex:binary}
Consider the class $\mathcal{C}$ of rooted binary
trees, where the (labelled) atoms are the inner nodes. The class $\mathcal{C}$ has
the following decomposition grammar,

$$\mathcal{C}= \left( \mathcal{C}+
\mathbf{1}\right)\star \cZ\star \left( \mathcal{C}+
\mathbf{1}\right).$$
Accordingly, the series $C(x)$ counting rooted binary trees satisfies 
$C(x)=x\left( 1+C(x)\right) ^2$. (Notice that $C(x)$ can be
easily evaluated for a fixed real parameter $x<\rho_C=1/4$.)

Using the sampling rules for Sum and Product, we obtain the following Boltzmann sampler for
binary trees, where $\{\bullet\}$ stands for a node:

\vspace{.2cm}

\begin{tabular}{ll}
$\Gamma \cC(x):$& return $(\Gamma(1+\cC)(x),\{\bullet\},\Gamma(1+\cC)(x))$ \{independent calls\}
\end{tabular}

\begin{tabular}{ll}
$\Gamma(1+\cC)(x):$& if $\Bern\left(\frac{1}{1+C(x)}\right)$ return leaf\\
& else return $\Gamma \cC(x)$
\end{tabular}

\noindent Distinct labels in $[1..|\gamma|]$ might then be distributed uniformly
at random on 
the atoms of the resulting tree $\gamma$, 
so as to make it well-labelled (see Remark~\ref{rk:labels}
below). 
Many more examples are given in~\cite{DuFlLoSc04} for labelled (and unlabelled) classes specified
using the constructions $\{+,\star,\Set\}$.
\end{example}

\vspace{0.2cm}

\begin{remark}\label{rk:labels} In the sampling rules (Figure~\ref{table:rules}), the procedure
\textsc{DistributeLabels}($\gamma$) throws distinct labels
uniformly at random on the L-atoms of $\gamma$. The fact that the
relabelling permutation is always chosen uniformly at random
ensures that the process of assigning the labels
has no memory of the past, hence 
\textsc{DistributeLabels} needs to be called just once, at the end of the 
generation procedure. (A similar remark is given by Flajolet
\emph{et al.} in~\cite[Sec. 3]{FlZiVa94} for the recursive method of sampling.)

In other words, when combining the sampling rules given in Figure~\ref{table:rules} in order to design 
a Boltzmann sampler, 
we can forget about the calls to \textsc{DistributeLabels}, see for instance the  Boltzmann 
sampler for binary trees above. In fact, we have included the \textsc{DistributeLabels} steps
in the definitions of the sampling rules only for the sake of writing the correctness proofs (Proposition~\ref{prop:rules}) in a proper way.
\end{remark}

\subsection{Additional techniques for Boltzmann sampling}
As the decomposition of planar graphs we consider is a bit involved, we need a few techniques in order
to properly translate this decomposition into a Boltzmann
sampler. These techniques, which are described in more detail below, are: 
bijections, pointing, and rejection.

\subsubsection{Combinatorial isomorphisms}
Two mixed classes $\cA$ and $\cB$ are said to be \emph{isomorphic}, shortly 
written as $\cA\simeq\cB$, if there exists a bijection $\Phi$ between $\cA$ and $\cB$
that preserves the size parameters, i.e., preserves the L-size and the U-size.
(This is equivalent to the fact that the mixed generating functions of $\cA$ and $\cB$ are equal.)
In that case,  a Boltzmann sampler $\Gamma \cA(x,y)$ for the class $\cA$ yields a Boltzmann sampler
for $\cB$ via the isomorphism: 
$\Gamma \cB(x,y): \gamma\leftarrow\Gamma \cA(x,y);\ \mathrm{return}\ \Phi(\gamma)$.

\subsubsection{L-derivation, U-derivation, and edge-rooting.}\label{sec:derive}
In order to describe our random sampler for planar graphs, 
we will make much use of \emph{derivative}
operators. The L-derived class of a mixed class $\cC=\cup_{n,m}\cC_{n,m}$ (shortly called
the derived class of $\cC$) is the mixed class
$\cC'=\cup_{n,m}\cC'_{n,m}$ of objects in $\cC$ where the greatest label is taken out, i.e., 
the L-atom with greatest label is discarded from the set of L-atoms (see the
book by Bergeron, Labelle, Leroux ~\cite{BeLaLe} for more details and examples). The class
$\cC'$ can be identified with the pointed class $\cC^{\bullet}$ of $\cC$, which is the 
class of objects of $\cC$ with a distinguished L-atom.  
Indeed the discarded atom  in an object of $\cC'$ 
plays the role of a pointed vertex. However the important
 difference between $\cC'$ and $\cC^{\bullet}$
is that the distinguished L-atom does not count in the L-size of an object in $\cC'$. 
In other words, $\cC^{\bullet}=\cZL\star\cC'$. 
Clearly, for any integers $n,m$, $\cC'_{n-1,m}$ identifies to $\cC_{n,m}$, so that the
generating function $C'(x,y)$ of $\cC'$ satisfies

\begin{equation}
C'(x,y)=\sum_{n,m} |\cC_{n,m}|\frac{x^{n-1}}{(n-1)!}y^m=\partial_x C(x,y).
\end{equation} 

The U-derived class of $\cC$ is the class $\ul{\cC}$ of objects obtained from objects of $\cC$ by 
discarding one U-atom from the set of U-atoms; in other words there is a distinguished U-atom
that does not count in the U-size. As in the definition of the U-substitution,  
we assume that all the U-atoms are distinguishable,
for instance the edges of a simple graph are distinguished by the labels of their extremities. In that case, 
$|\ul{\cC}_{n,m-1}|=m|\cC_{n,m}|$, so that 
the generating function $\ul{C}(x,y)$ 
of $\ul{\cC}$ satisfies
\begin{equation}
\ul{C}(x,y)=\sum_{n,m} m|\cC_{n,m}|\frac{x^{n}}{n!}y^{m-1}=\partial_y C(x,y).
\end{equation}

For the particular case of planar graphs, we will also
consider \emph{edge-rooted} objects (shortly called rooted objects), i.e., planar graphs where an edge 
is ``marked'' (distinguished) and directed. 
In addition, the root edge, shortly called the root, is not counted as an unlabelled atom, and
the two extremities of the root do not count as labelled atoms (i.e., are not labelled). 
The edge-rooted class of $\cC$ is denoted 
by $\vec{\cC}$. Clearly we have $\cZL^{\ 2}\star\vec{\cC}\simeq 2\star \ul{\cC}$.
Hence, the generating function $\vec{C}(x,y)$ of $\vec{\cC}$ satisfies 
\begin{equation}
\vec{C}(x,y)=\frac{2}{x^2}\partial_y C(x,y).
\end{equation}

\subsubsection{Rejection.}\label{sec:reject}
Using rejection techniques offers great flexibility to design Boltzmann samplers, since
it makes it possible to adjust
the distributions of the samplers.
\begin{lemma}[Rejection]
\label{lemma:rej}
Given a combinatorial class $\cC$, let $W:\cC\mapsto\mathbf{R}^+$ and $p:\cC\mapsto [0,1]$ be two functions, called \emph{weight-function} and \emph{rejection-function}, respectively. Assume 
that $W$ is summable, i.e., $\sum_{\gamma\in\cC}W(\gamma)$ is finite. Let $\frak{A}$ be a random
generator for $\cC$ that draws each object $\gamma\in\cC$ with probability proportional to $W(\gamma)$. Then, the procedure
$$
\frak{A}_{\mathrm{rej}}:\mathrm{repeat}\ \frak{A}\rightarrow\gamma\ \mathrm{until}\ \mathrm{Bern}(p(\gamma));\ \mathrm{return}\ \gamma
$$
is a random generator on $\cC$, which draws each object $\gamma\in\cC$ with probability proportional to $W(\gamma)p(\gamma)$.
\end{lemma}
\begin{proof}
Define $W:=\sum_{\gamma\in\cC}W(\gamma)$. By definition, $\frak{A}$ draws an object $\gamma\in\cC$ with probability $P(\gamma):=W(\gamma)/W$. Let $p_{\mathrm{rej}}$ be the probability of failure of $\frak{A}_{\mathrm{rej}}$
at each attempt. The probability $P_{\mathrm{rej}}(\gamma)$ that $\gamma$ is
drawn by $\frak{A}_{\mathrm{rej}}$ satisfies
$ P_{\mathrm{rej}}(\gamma)=P(\gamma)p(\gamma)+p_{\mathrm{rej}}P_{\mathrm{rej}}(\gamma),$ where the 
first (second) term is the probability that $\gamma$ is drawn at the first attempt (at a later 
 attempt, respectively). Hence, $P_{\mathrm{rej}}(\gamma)=P(\gamma)p(\gamma)/(1-p_{\mathrm{rej}})=W(\gamma)p(\gamma)/(W\cdot(1-p_{\mathrm{rej}}))$, i.e., $P_{\mathrm{rej}}(\gamma)$ is proportional to $W(\gamma)p(\gamma)$. 
\end{proof}

Rejection techniques are very useful for us to change the way objects are rooted.
Typically it helps us to obtain a Boltzmann sampler for $\cA'$ from a
Boltzmann sampler for  $\ul{\cA}$ and vice versa. As we will use this trick many 
times, we formalise it here by giving two explicit procedures, one from L-derived to U-derived objects, the other one from U-derived to L-derived objects.

\vspace{.5cm}

\fbox{
\begin{minipage}{12cm}
\LtoU\\
\phantom{1}\hspace{.5cm} INPUT: a mixed class $\cA$ such that $\ds\alphaUL:=\mathrm{sup}_{\gamma\in\cA}\frac{||\gamma||}{|\gamma|}$ is finite,\\
\phantom{1}\hspace{2cm}a Boltzmann sampler $\Gamma \cA'(x,y)$ for the L-derived class $\cA'$\\[0.2cm]
\phantom{1}\hspace{.5cm} OUTPUT: a Boltzmann sampler for the U-derived class $\ul{\cA}$, defined as:\\[0.2cm]
\begin{tabular}{ll}
$\Gamma \ul{\cA}(x,y)$:& repeat $\gamma\leftarrow\Gamma \cA'(x,y)$ \{at this point $\gamma\in\cA'$\}\\
& \phantom{1}\hspace{.2cm}give label $|\gamma|+1$ to the discarded L-atom of $\gamma$;\\
& \phantom{1}\hspace{.2cm}\{so $|\gamma|$ increases by $1$, and $\gamma\in\cA$\}\\
& until $\ds\mathrm{Bern}\left(\frac{1}{\alphaUL}\frac{||\gamma||}{|\gamma|}\right)$;\\
& choose a U-atom  uniformly at random and discard it\\
& $\ \ $ from the set of U-atoms; \{so $||\gamma||$ decreases by $1$, and $\gamma\in\ul{\cA}$\}\\
& return $\gamma$
\end{tabular}

\end{minipage}}

\begin{lemma}\label{lem:LtoU}
The procedure \LtoU yields a Boltzmann sampler for the class $\ul{\cA}$ from a 
Boltzmann sampler for the class $\cA'$.
\end{lemma}
\begin{proof}
First, observe that the sampler is well defined. Indeed, by definition of the 
parameter $\alphaUL$, the Bernoulli choice is always valid (i.e., its parameter is always
in $[0,1]$). Notice that the sampler\\ 
\phantom{1}\hspace{.4cm}$\gamma\leftarrow\Gamma \cA'(x,y)$;\\ 
\phantom{1}\hspace{.4cm}give label $|\gamma|+1$ to the discarded L-atom of $\gamma$;\\ 
\phantom{1}\hspace{.4cm}return $\gamma$\\
\noindent is a sampler for $\cA$ that outputs each object $\gamma\in\cA$ with probability $\frac{1}{A'(x,y)}\frac{x^{|\gamma|-1}}{(|\gamma|-1)!}y^{||\gamma||}$,
because $\cA_{n,m}$ identifies to $\cA'_{n-1,m}$.
In other words, this sampler draws each object $\gamma\in\cA$
with probability proportional to $|\gamma|\frac{x^{|\gamma|}}{|\gamma|!}y^{||\gamma||}$.
Hence, according to Lemma~\ref{lemma:rej}, the repeat-until loop of the sampler
$\Gamma \ul{\cA}(x,y)$ yields a sampler for $\cA$ such that each object has 
probability proportional to $||\gamma||\frac{x^{|\gamma|}}{|\gamma|!}y^{||\gamma||}$.
As each U-atom has probability $1/||\gamma||$ of being discarded, the final sampler is such that each object $\gamma\in\ul{\cA}$ has probability proportional to $\frac{x^{|\gamma|}}{|\gamma|!}y^{||\gamma||}$. So $\Gamma\ul{\cA}(x,y)$ is a Boltzmann sampler for $\ul{\cA}$.
\end{proof}

We define a similar procedure to go from a U-derived class to an L-derived class:

\vspace{.5cm}

\fbox{
\begin{minipage}{12cm}
\UtoL\\
\phantom{1}\hspace{.5cm} INPUT: a mixed class $\cA$ such that $\ds\alphaLU:=\mathrm{sup}_{\gamma\in\cA}\frac{|\gamma|}{||\gamma||}$ is finite,\\
\phantom{1}\hspace{2cm}a Boltzmann sampler $\Gamma \ul{\cA}(x,y)$ for the U-derived class $\ul{\cA}$\\[0.2cm]
\phantom{1}\hspace{.5cm} OUTPUT: a Boltzmann sampler for the L-derived class 
$\cA'$, defined as:\\[0.2cm]
\begin{tabular}{ll}
$\Gamma \cA'(x,y)$:& repeat $\gamma\leftarrow\Gamma \ul{\cA}(x,y)$ \{at this point $\gamma\in\ul{\cA}$\}\\
& \phantom{1}\hspace{.2cm}take the discarded U-atom of $\gamma$ back in the set of U-atoms;\\
& \phantom{1}\hspace{.2cm} \{so $||\gamma||$ increases 
by $1$, and $\gamma\in\cA$\}\\
& until $\ds\mathrm{Bern}\left(\frac{1}{\alphaLU}\frac{|\gamma|}{||\gamma||}\right)$;\\
& discard the L-atom  with greatest label from the set of L-atoms;\\
& \{so $|\gamma|$ decreases by $1$, and $\gamma\in\cA'$\}\\
& return $\gamma$
\end{tabular}

\end{minipage}}

\begin{lemma}\label{lem:UtoL}
The procedure \UtoL yields a Boltzmann sampler for the class $\cA'$ from a 
Boltzmann sampler for the class $\ul{\cA}$.
\end{lemma}
\begin{proof}
Similar to the proof of Lemma~\ref{lem:LtoU}. The sampler $\Gamma \cA'(x,y)$ is well defined, as the Bernoulli choice is always valid (i.e., its parameter is always
in $[0,1]$). Notice that the sampler\\ 
\phantom{1}\hspace{.4cm}$\gamma\leftarrow\Gamma \ul{\cA}(x,y)$;\\ 
\phantom{1}\hspace{.4cm}take the discarded U-atom back to the set of U-atoms of $\gamma$;\\ 
\phantom{1}\hspace{.4cm}return $\gamma$\\
\noindent is a sampler for $\cA$ that outputs each object $\gamma\in\cA$ with probability $\frac{1}{\ul{A}(x,y)}||\gamma||\frac{x^{|\gamma|}}{|\gamma|!}y^{||\gamma||-1}$,
(because an object $\gamma\cA_{n,m}$ 
gives rise to $m$ objects in $\ul{\cA}_{n,m-1}$), i.e., with probability proportional to $||\gamma||\frac{x^{|\gamma|}}{|\gamma|!}y^{||\gamma||}$.
Hence, according to Lemma~\ref{lemma:rej}, the repeat-until loop of the sampler
$\Gamma \cA'(x,y)$ yields a sampler for $\cA$ such that each object $\gamma\in\cA$ has 
probability proportional to $|\gamma|\frac{x^{|\gamma|}}{|\gamma|!}y^{||\gamma||}$,
i.e., proportional to $\frac{x^{|\gamma|-1}}{(|\gamma|-1)!}y^{||\gamma||}$.
Hence, by discarding the greatest L-atom (i.e., $|\gamma|\leftarrow|\gamma|-1$), 
we get a probability proportional to $\frac{x^{|\gamma|}}{|\gamma|!}y^{||\gamma||}$
for every object $\gamma\in\cA'$, i.e., a Boltzmann sampler for $\cA'$.
\end{proof}

\begin{remark}\label{remark:greatest_delete}
We have stated in Remark~\ref{rk:labels} that, during a generation process, 
it is more convenient in practice to manipulate the shapes of the objects
without systematically assigning labels to them.
However, in the definition of the sampler $\Gamma \cA'(x,y)$, one step is to remove the greatest label, so it seems we need to look at the labels at that step.
In fact, as we consider here classes that are stable under relabelling, 
it is equivalent in practice to draw uniformly at random one vertex
to play the role of the 
discarded L-atom. 
\end{remark}

\section{Decomposition of planar graphs and Boltzmann samplers}
\label{sec:decomp}

Our algorithm starts with the generation of 3-connected planar graphs,
which have the nice feature that they are combinatorially tractable.
Indeed, according to a theorem of Whitney~\cite{Whitney33}, 3-connected planar graphs
have a unique embedding (up to reflection), so they are equivalent to 3-connected planar maps.
Following the general  approach introduced by Schaeffer~\cite{S-these}, 
 a bijection has been described by the author, Poulalhon, and Schaeffer~\cite{FuPoSc05} 
 to enumerate 3-connected maps~\cite{FuPoSc05} from binary trees, 
 which yields an explicit Boltzmann sampler for (rooted) 3-connected maps, as described in 
Section~\ref{sec:bolz3conn}. 

The next step is to generate
2-connected planar graphs from 3-connected ones. We take advantage of a decomposition
of 2-connected planar graphs into 3-connected planar components, which has been 
formalised by Trakhtenbrot~\cite{trak} 
(and later used by Walsh~\cite{Wa} to count 2-connected 
planar graphs and by Bender, Gao, Wormald to obtain asymptotic enumeration~\cite{BeGa}).
Finally, connected planar graphs are generated from 2-connected ones by using the well-known
decomposition into blocks, and planar graphs are generated from their connected components.
Let us mention that the decomposition of planar graphs into 3-connected components has been
 completely formalised  by Tutte~\cite{Tut} 
(though we rather use here formulations of this decomposition on \emph{rooted} graphs, as  Trakhtenbrot did).

The complete scheme we follow is illustrated in 
Figure~\ref{fig:scheme_unrooted}.

\begin{figure}
  \begin{center}
  \includegraphics[width=13.6cm]{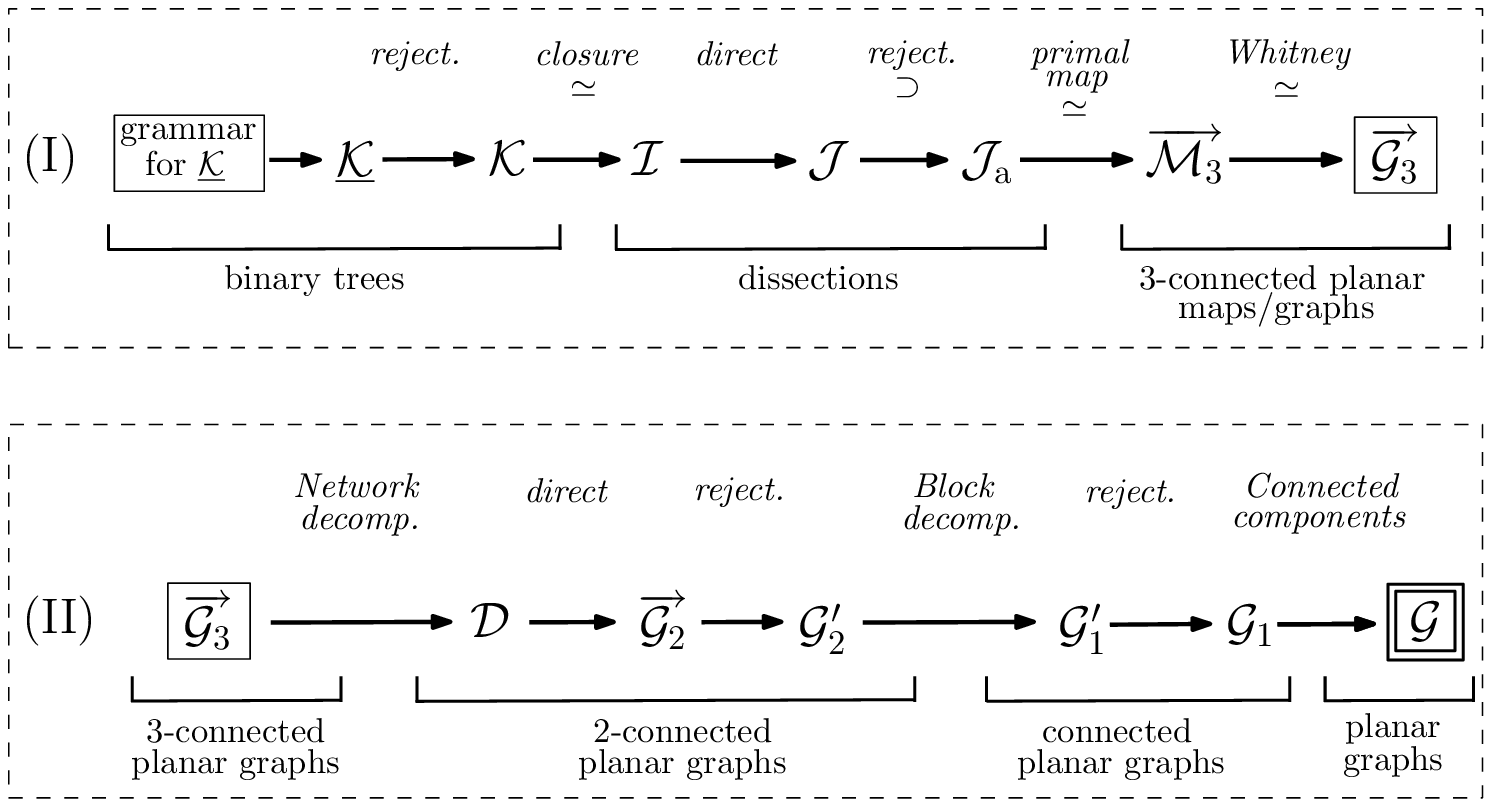}
    \caption{The complete scheme to obtain a Boltzmann sampler for 
     planar graphs. The classes are to be defined all along Section~\ref{sec:decomp}.}\label{fig:scheme_unrooted}
  \end{center}
\end{figure}



\vspace{0.2cm}

\noindent\textbf{Notations.} 
Recall that a graph is $k$-connected  if
the removal of any set of $k-1$ vertices does not disconnect the graph.
In the sequel, we consider the following classes of planar graphs: 

\vspace{0.2cm}

\begin{tabular}{l}
$\cG$: the class of all planar graphs, including the empty graph,\\
 $\cGc$: the class of connected planar graphs with at least one vertex,\\
  $\cGb$: the class
of 2-connected planar graphs with at least two vertices,\\ 
$\cGt$: the class of 3-connected planar graphs with at least four vertices.
\end{tabular}

\vspace{0.2cm}

\begin{figure}
\begin{center}
\input{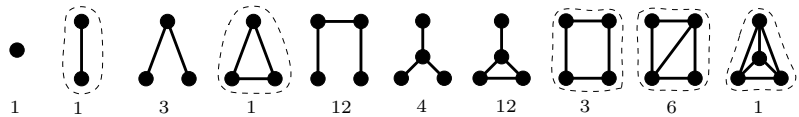}
\end{center}
\caption{The connected planar graphs with at most four vertices (the 2-connected ones are surrounded). 
Below each graph is indicated the number of 
distinct labellings.}
\label{fig:firstTerms}
\end{figure}

All these classes are considered as mixed, with labelled vertices and unlabelled edges, i.e.,
the L-atoms are the vertices and the U-atoms are the edges.
Let us give the first few terms of their mixed generating functions (see also Figure~\ref{fig:firstTerms}, which 
displays the first connected planar graphs):
$$
\displaystyle
\begin{array}{rcl}
G(x,y)&$\!\!\!=\!\!\!$&1+x+\frac{x^2}{2!}(1+y)+\frac{x^3}{3!}(1+3y+3y^2+y^3)+\cdots\\[0.1cm]
G_1(x,y)&$\!\!\!=\!\!\!$&x+\frac{x^2}{2!}y+\frac{x^3}{3!}(3y^2+y^3)+\frac{x^4}{4!}(16y^3+15y^4+6y^5+y^6)+\cdots\\[0.1cm]
G_2(x,y)&$\!\!\!=\!\!\!$&\frac{x^2}{2!}y+\frac{x^3}{3!}y^3+\frac{x^4}{4!}(3y^4\!+\!6y^5\!+\!y^6)+\frac{x^5}{5!}(12y^5\!+\!70y^6\!+\!100y^7\!+\!15y^8\!+\!10y^9)+\cdots\\[0.1cm]
G_3(x,y)&$\!\!\!=\!\!\!$&\frac{x^4}{4!}y^6+\frac{x^5}{5!}(15y^8+10y^9)+\frac{x^6}{6!}(60y^9+432y^{10}+540y^{11}+195y^{12})+\cdots
\end{array}
$$

Observe that, for a mixed class $\cA$ of \emph{graphs}, 
the derived class $\cA'$, as defined in Section~\ref{sec:derive},
 is the class of graphs in $\cA$ that have  
one vertex  discarded from the set of L-atoms (this vertex plays the role of a distinguished vertex);
 $\ul{\cA}$ is the class of graph in $\cA$ with one edge discarded from the set of U-atoms 
 (this edge plays the role of a 
distinguished edge); and $\vec{\cA}$ is the class of graphs in $\cA$ with an ordered pair of adjacent vertices $(u,v)$ discarded from the set of L-atoms and the edge $(u,v)$ discarded from the set 
of U-atoms (such a graph can be considered as rooted at the directed edge $(u,v)$).  


\subsection{Boltzmann sampler for 3-connected planar graphs}
\label{sec:bolz3conn}
In this section we develop a Boltzmann sampler for 
3-connected planar graphs, more precisely for \emph{edge-rooted} ones, i.e., for the 
class $\vec{\cGt}$. Our sampler 
relies on two results. First, we recall the equivalence between  
3-connected planar graphs and 3-connected maps, where the terminology of map refers to an 
explicit embedding. Second, we take advantage of a bijection linking the families of rooted 3-connected maps
and the (very simple) family of binary trees, via intermediate objects that are certain 
quadrangular dissections of the hexagon. Using the bijection, a Boltzmann sampler for rooted 
binary trees
is translated into a Boltzmann sampler for rooted 3-connected maps. 

\subsubsection{Maps}
A \emph{map on the sphere} (\emph{planar map}, resp.) is 
a connected planar graph embedded on
the sphere (on the plane, resp.) up to continuous deformation of the surface, the embedded
graph carrying distinct labels on its vertices (as usual, the labels range from $1$ to $n$,  the number
of vertices). A planar map
is in fact equivalent to a map on the sphere with a distinguished face, which plays the role
of the unbounded face. The unbounded face of a planar map 
is called the \emph{outer face}, and the other faces
are called the  \emph{inner faces}. The vertices and edges of a planar map are said to be \emph{outer}
or \emph{inner} whether they are incident to the outer face or not.
A map is said to be \emph{rooted} if the embedded graph is edge-rooted.
The \emph{root vertex} is the origin of the root.
Classically, rooted planar maps are always assumed to have the outer face 
 on the right of the root. With that convention, rooted planar maps are 
equivalent to rooted maps on the sphere (given a rooted map on the sphere, take the face on
the right of the root as the outer face).
See Figure~\ref{fig:primal}(c) for an example of rooted planar map, where the labels are forgotten\footnote{Classically, rooted maps are considered in the literature without labels on the vertices,
as the root is enough to avoid symmetries. Nevertheless, it is convenient here to keep 
the framework of mixed classes for maps, as we do for graphs.}. 

\subsubsection{Equivalence between 3-connected planar graphs and 3-connected maps}\label{sec:equiv}
A well known result due to Whitney~\cite{Whitney33} 
states that a labelled 3-connected planar graph has a
unique embedding on the sphere up to continuous deformation and reflection 
(in general a planar graph can have many
embeddings). Notice that any 3-connected map on the sphere 
with at least 4 vertices differs from its mirror-image, due to the labels on the vertices. 
Hence every 3-connected planar graph with at least 4 vertices gives rise exactly
to  two maps on the sphere.
The 
class of 3-connected maps on the sphere with at least 4 vertices 
is denoted by $\cM_3$. As usual, the class is mixed, the L-atoms being the
vertices and the U-atoms being the edges.
Whitney's theorem ensures that
\begin{equation}
\label{eq:M}
\cM_3\simeq 2\star\cGt.
\end{equation}


Here we make use of the formulation of this isomorphism for \emph{edge-rooted} objects.
The mixed class of rooted 3-connected planar 
maps with at least 4 vertices is denoted by $\vec{\cM_3}$, where ---as for edge-rooted graphs--- the L-atoms are the vertices not incident to the root-edge and the U-atoms are the edges
except the root. Equation~(\ref{eq:M}) becomes, for edge-rooted objects:
\begin{equation}
\vec{\cM_3}\simeq2\star\vec{\cGt}.
\end{equation}





Thanks to this isomorphism, finding a Boltzmann sampler $\Gamma \vec{\cG_3}(z,w)$ 
for edge-rooted 3-connected planar graphs reduces to finding  
a Boltzmann sampler $\Gamma \vec{\cM_3}(z,w)$ for rooted 3-connected maps, 
upon forgetting the embedding.





\subsubsection{3-connected maps and irreducible dissections}\label{sec:primal_map}
We consider here some quadrangular dissections of the hexagon that are closely related
to 3-connected planar maps. (We will see that these dissections can be efficiently generated at random,
as they are in bijection with binary trees.) 

Precisely, a \emph{quadrangulated map} is a planar map (with no loop nor multiple edges) such that 
all faces except maybe the outer one have degree 4; it is called a quadrangulation if the 
outer face has degree 4. A quadrangulated map is called \emph{bicolored} if the vertices are colored  black or white such that any edge connects two vertices of different colors.
A rooted quadrangulated
map (as usual with planar maps, the root has the outer face on its right) is always assumed 
to be endowed
with the unique vertex bicoloration such that the root vertex is \emph{black}  (such a bicoloration exists, as all inner faces have even degree). 
A quadrangulated map with an outer face of degree more than 4 is called \emph{irreducible}  if 
each  4-cycle is the contour of a face. 
In particular, we define an \emph{irreducible dissection of the hexagon} ---shortly called irreducible dissection hereafter---  as an irreducible quadrangulated map with an 
hexagonal outer face, see Figure~\ref{fig:primal}(b) for an example.
A quadrangulation is called irreducible
if it has at least 2 inner vertices and if every 4-cycle, except the outer one,
delimits a face.
 Notice that the smallest irreducible dissection has one inner edge and no inner vertex
 (see Figure~\ref{fig:asymmetric}), whereas 
the smallest irreducible quadrangulation is the embedded cube,
which has 4 inner vertices and 5 inner faces. 
We consider irreducible dissections as objects of the mixed type, the L-atoms are the black inner vertices and the U-atoms are the inner faces.
It proves more convenient to consider here the irreducible dissections that are \emph{asymmetric},
meaning that there is no rotation fixing the dissection. The four non-asymmetric irreducible dissections
are displayed in Figure~\ref{fig:asymmetric}(b), all the other ones are asymmetric either due to an
asymmetric shape or due to the labels on the black inner vertices.
We denote by $\cI$ the mixed class of \emph{asymmetric} bicolored irreducible dissections.  
We define also $\cJ$ as the class of asymmetric irreducible dissections that carry a root (outer edge directed so as to have a black origin and the outer face on its right), where this time 
the L-atoms are the black vertices
except two of them (say, the origin of the root and the next black vertex in ccw order around the outer face) and the U-atoms are all the faces, including the outer one.  
Finally, we define $\cQ$ as the mixed class of rooted 
irreducible quadrangulations, where the L-atoms are the black vertices except those two incident
to the outer face, and the U-atoms  are the inner faces.

Irreducible dissections are closely related to 3-connected maps, via a classical correspondence
between planar maps and quadrangulations. 
Given a bicolored rooted quadrangulation $\kappa$, the 
\emph{primal map} of $\kappa$ is the rooted map $\mu$ whose vertex set is the set of black vertices of $\kappa$,
each face $f$ of $\kappa$ giving rise to an edge of $\mu$ 
connecting the two (opposite) black vertices of $f$,
 see Figure~\ref{fig:primal}(c)-(d).
The map $\mu$ is naturally rooted so as to have the same root-vertex as $\kappa$.

\begin{theorem}[Mullin and Schellenberg~\cite{Mu}]
The primal-map construction is a bijection between rooted irreducible 
quadrangulations
with $n$ black vertices and $m$ faces, and rooted 
3-connected maps with $n$ vertices and $m$ edges\footnote{More generally, the bijection holds
between rooted quadrangulations and rooted 2-connected maps.}. In other words, 
the primal-map construction yields the combinatorial isomorphism
\begin{equation}
\cQ\simeq\vec{\cM_3}.
\end{equation} 
In addition, the construction of a 3-connected map from an irreducible quadrangulation takes
linear time.
\end{theorem}

The link between $\cJ$ and $\vec{\cM_3}$ is established via the family $\cQ$, which is 
at the same time isomorphic to $\vec{\cM_3}$ and closely related to $\cJ$.
Let $\kappa$ be a rooted irreducible quadrangulation, and let $e$ be the edge
following the root in cw order around the outer face. 
Then, deleting $e$ yields
a rooted irreducible dissection $\delta$.
In addition it is easily checked that $\delta$ is  asymmetric, i.e.,
the four non-asymmetric irreducible dissections, which are shown in Figure~\ref{fig:asymmetric}(b), can not be obtained in this way. 
Hence the so-called \emph{root-deletion mapping} 
is injective from $\cQ$ to $\cJ$. 
The inverse operation---called the \emph{root-addition mapping}---starts from a 
rooted irreducible dissection $\delta$, and adds an outer
edge from the root-vertex of $\delta$ to the opposite outer vertex. Notice that the rooted quadrangulation  obtained in this way might not be irreducible. 
Precisely, a non-separating 4-cycle appears iff $\delta$ has an internal path (i.e., a path using at least one inner edge) of length 3 connecting the root vertex to the opposite outer vertex.  
A rooted irreducible dissection $\delta$ 
is called \emph{admissible} iff it has no such path. The subclass of rooted  irreducible
dissections that are admissible is denoted by $\cJa$. We obtain the following result, already given in~\cite{FuPoSc05}:
\begin{lemma}
The root-addition mapping is a bijection between admissible rooted irreducible dissections 
with $n$ black vertices and $m$ faces, and rooted irreducible quadrangulations with $n$ black 
vertices and $m$ inner faces. In other words, the root-addition mapping realises the combinatorial
isomorphism
\begin{equation}
\cJa\simeq\cQ.
\end{equation}
\end{lemma}

To sum up, we have the following link between rooted irreducible dissections and rooted 3-connected
maps: 
$$
\cJ\supset\ \cJa\simeq\cQ\simeq\vec{\cM_3}.
$$ 
Notice that  we have a combinatorial isomorphism between $\cJa$ and $\vec{\cM_3}$: 
 the root-edge addition combined with  the primal map construction.
For $\delta\in\cJa$, the rooted 3-connected map associated with $\delta$ is denoted
$\mathrm{Primal}(\delta)$.


\begin{figure}
\begin{center}
\includegraphics[width=12cm]{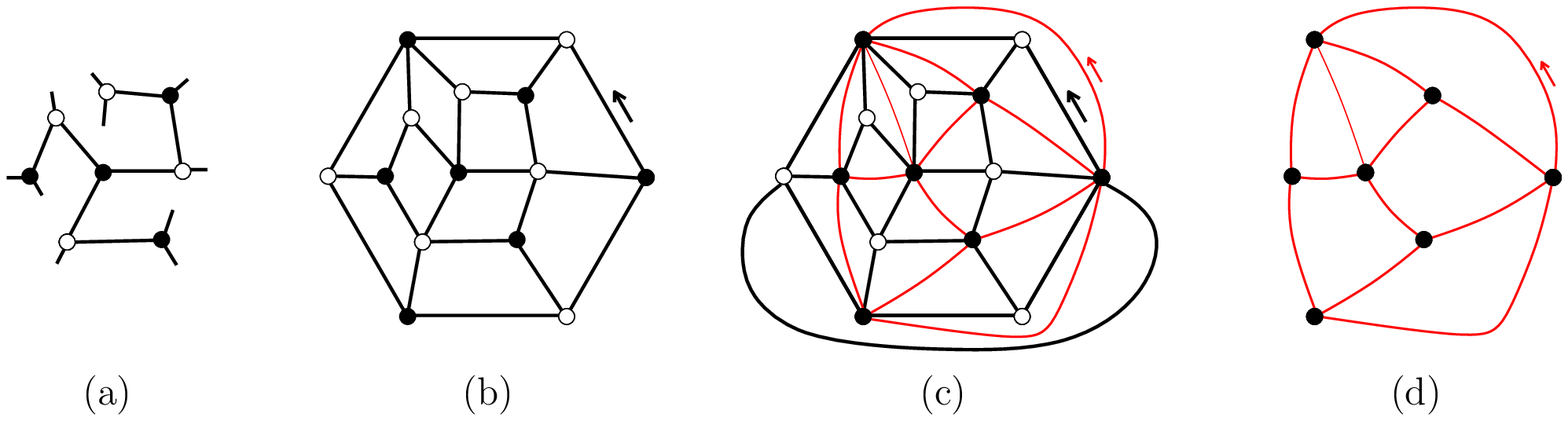}
\end{center}
\caption{(a) A binary tree, (b) the associated  irreducible dissection $\delta$ (rooted and 
 admissible), (c) the associated rooted
irreducible quadrangulation $\kappa=\mathrm{Add}(\delta)$, (d)
the associated rooted 3-connected map $\mu=\mathrm{Primal}(\delta)$.}
\label{fig:primal}
\end{figure}

As we see next, the class $\cI$ (and also the associated rooted class $\cJ$) is combinatorially tractable, as it is  
in bijection with the simple class of binary trees; hence irreducible dissections are easily
 generated at random.

\subsubsection{Bijection between binary trees and irreducible dissections}
There exist by now several elegant bijections between families of planar maps and families
of plane trees that satisfy simple context-free decomposition grammars. 
Such constructions
have first been described by Schaeffer in his thesis~\cite{S-these}, and
many other families of rooted maps have been counted in this way~\cite{Fusy06a,PS03a,PS03b,BoDiGu04}.
 The advantage of bijective constructions over recursive methods for counting maps~\cite{Tu63} 
  is that the bijections
 yield efficient ---linear-time--- generators for maps, as random sampling
 of maps is reduced to the much easier task of random sampling of trees, see~\cite{Sc99}.
 The method has been recently applied  to the family 
 of 3-connected maps, which 
 is of interest here. Precisely, as described in~\cite{FuPoSc05}, there is a bijection between binary trees and irreducible dissections
 of the hexagon, 
 which, as we have seen, are closely related
 to 3-connected maps. 
 
We define an  \emph{unrooted binary tree}, shortly called a binary tree hereafter, 
as a plane tree (i.e., a planar map with a unique face) 
where the degree of each vertex is either 1 or 3. 
 The vertices of degree 1 (3) are called leaves (nodes, resp.).
 A binary tree is said to be bicolored if its nodes are bicolored so that any two adjacent nodes 
 have different colors, see Figure~\ref{fig:primal}(a) for an example. 
In a bicolored binary tree the L-atoms are the black nodes  and the U-atoms are the leaves. A bicolored
 binary tree is called
 \emph{asymmetric} if there is no rotation-symmetry fixing it. Figure~\ref{fig:asymmetric} displays the four non-asymmetric
 bicolored binary trees; all the other bicolored binary trees are asymmetric, either due to the 
 shape being asymmetric, or due to the labels on the black nodes.
  We denote by $\cK$ the mixed class of \emph{asymmetric} bicolored binary trees (the requirement of asymmetry is necessary so that the leaves are distinguishable).
  
  \begin{figure}
  \begin{center}
  \includegraphics[width=12cm]{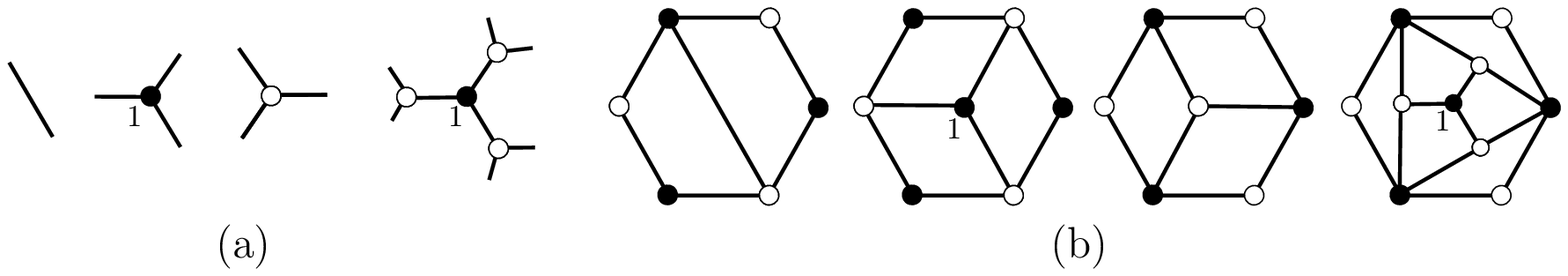}
  \end{center}
  \caption{(a) The four non-asymmetric bicolored binary trees. (b) The four non-asymmetric bicolored irreducible
  dissections.}
  \label{fig:asymmetric}
  \end{figure}

 The terminology of binary tree refers to the fact that, upon rooting a binary tree
  at an arbitrary leaf, the neighbours
 in clockwise order around each node can be classified as a father (the neighbour closest to the root), a right son, and
 a left son, which corresponds to the classical definition of rooted binary trees, as considered
 in Example~\ref{ex:binary}.


\begin{proposition}[Fusy, Poulalhon, and Schaeffer~\cite{FuPoSc05}]
\label{prop:bijbin3conn}
For $n\geq 0$ and $m\geq 2$, there exists an explicit bijection, called the  \emph{closure-mapping}, between bicolored binary trees
with $n$ black nodes and $m$ leaves, and bicolored irreducible dissections with $n$ black
inner nodes and $m$ inner faces; moreover the 4 non-asymmetric bicolored binary trees are mapped
to the 4 non-asymmetric irreducible dissections. 
In other words, the closure-mapping realises the combinatorial
isomorphism
\begin{equation}\cK\simeq \cI.\end{equation} 
The construction of  a dissection from a binary tree takes linear time. 
\end{proposition}
Let us comment a bit on this bijective construction, which is described in detail in~\cite{FuPoSc05}.
Starting from a binary tree, the closure-mapping builds the dissection face by face, each leaf 
of the tree giving rise to an inner face of the dissection. More precisely, at each step,
a ``leg" (i.e., an edge incident to a leaf) is completed into an edge connecting two nodes, so as
to ``close" a quadrangular face. At the end, an hexagon is created outside of the figure, and the
leaves attached to the remaining non-completed legs are merged with vertices of the hexagon
so as to form only quadrangular faces. For instance the dissection of Figure~\ref{fig:primal}(b)
is obtained by ``closing'' the tree of Figure~\ref{fig:primal}(a).

\subsubsection{Boltzmann sampler for rooted bicolored binary trees}
\label{sec:boltz_binary_trees}
We define a rooted bicolored binary tree as a  binary tree with a marked leaf discarded from the set
of U-atoms. 
Notice that the class of rooted bicolored binary trees such that the underlying unrooted binary tree
is asymmetric   
is the U-derived class $\ul{\cK}$.

In order to write down a decomposition grammar for the class $\ul{\cK}$---to be translated into a Boltzmann sampler---we define some refined classes of rooted bicolored binary trees (decomposing $\ul{\cK}$ is a bit involved since we have 
to forbid the 4 non-asymmetric binary trees):
 $\cRb$ is the class of \emph{black-rooted} binary trees (the root leaf is connected to a black node) with at least one node,
 and $\cRw$ is the class of \emph{white-rooted}
binary trees (the root leaf is connected to a white node) with at least one node. 
We also  define $\cRbas$ ($\cRwas$) as the class of black-rooted (white-rooted, resp.)
bicolored binary trees such that the underlying unrooted binary tree is asymmetric. 
Hence $\ul{\cK}=\cRbas+\cRwas$.
We introduce two auxiliary classes; $\cRbh$ is the class of black-rooted binary trees except the (unique) one
with one black node and two white nodes; and $\cRwh$ 
is the class of white-rooted
binary trees except the two ones resulting from rooting the (unique) bicolored binary tree with one
black node and three white nodes (the 4th one in Figure~\ref{fig:asymmetric}(a)), in addition, the rooted bicolored binary tree with two leaves (the first one in Figure~\ref{fig:asymmetric}(a)) is also included in the class $\cRwh$. 

The decomposition of a bicolored binary tree at the root 
yields a complete decomposition grammar, given in Figure~\ref{fig:grammar}, for the class $\ul{\cK}=\cRbas+\cRwas$.
This grammar translates to a decomposition grammar involving only the basic classes
$\{\cZL,\cZU\}$ and the constructions $\{+,\star\}$  
($\cZL$ stands for a black node and $\cZU$ stands for a non-root leaf):
\begin{equation}\label{eq:grammar}
\left\{
\begin{array}{rcl}
\ul{\cK}&=&\cRbas+\cRwas,\\
\cRbas&=&\cRw\star\cZL\star\cZU+\cZU\star\cZL\star\cRw+\cZL\star\cRw^2,\\
\cRwas&=&\cRbh\star\cZU+\cZU\star\cRbh+\cRb^2,\\
\cRbh&=&\cRwh\star\cZL\star\cZU^2+\cZU^2\star\cZL\star\cRwh+\cRwh\star\cZL\star\cRwh,\\
\cRwh&=&\cZU+\cRb\star\cZU+\cZU\star\cRb+\cRb^2,\\
\cRb&=&(\cZU+\cRw)\star\cZL\star(\cZU+\cRw),\\
\cRw&=&(\cZU+\cRb)\star(\cZU+\cRb).
\end{array}
\right.
\end{equation}

\begin{figure}
\centerline{\includegraphics[width=14cm]{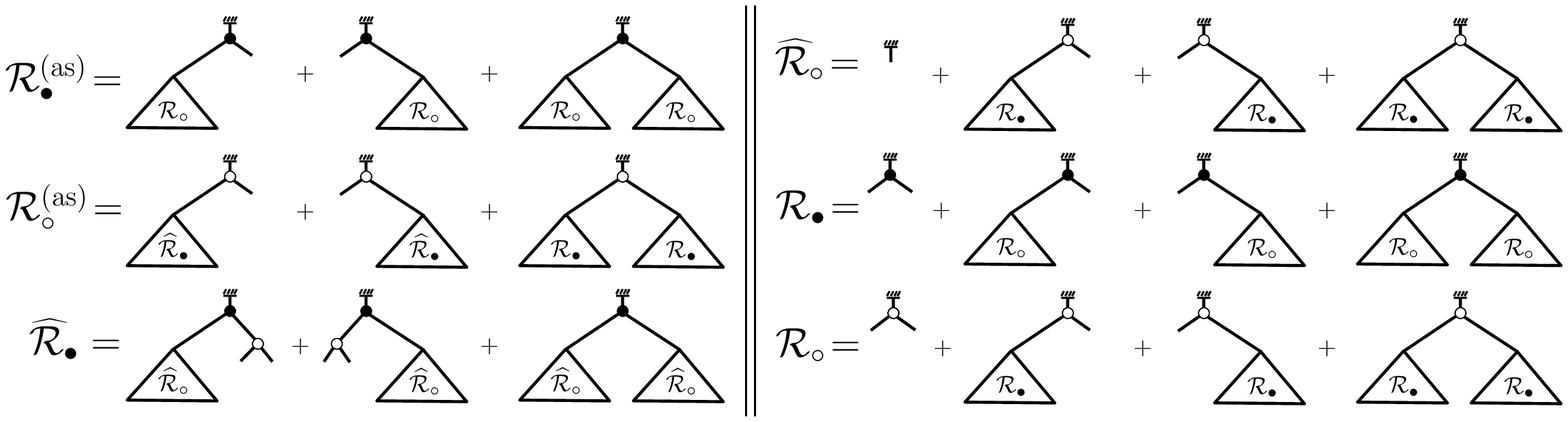}}
\caption{The decomposition grammar for the two classes $\cRbas$ and $\cRwas$ of rooted bicolored
binary trees such that the underlying binary tree is asymmetric.}
\label{fig:grammar}
\end{figure}

In turn, this grammar is translated into a Boltzmann 
sampler $\Gamma \ul{\cK}(z,w)$ 
for the class $\ul{\cK}$ using 
the sampling rules given in Figure~\ref{table:rules}, similarly as
we have done for the (simpler) class of complete binary trees in Example~1.

\subsubsection{Boltzmann sampler for bicolored binary trees}\label{sec:Ksamp}
We describe in this section a Boltzmann sampler $\Gamma \cK(z,w)$ for asymmetric bicolored binary trees, which is 
derived from the Boltzmann sampler $\Gamma\ul{\cK}(x,y)$ 
described in the previous section. Observe that each \emph{asymmetric} binary tree in $\cK_{n,m}$ gives rise
to $m$ rooted binary trees in $\ul{\cK}_{n,m-1}$, as each of the $m$ leaves, which are 
\emph{distinguishable}, might be chosen
to be discarded from the set of U-atoms. Hence, each object of $\cK_{n,m}$ has probability
$\ul{K}(z,w)^{-1}mz^n/n!y^{m-1}$ to be chosen when calling $\Gamma\ul{\cK}(z,w)$ and taking
the distinguished atom back into the set of U-atoms. Hence, from the rejection lemma 
(Lemma~\ref{lemma:rej}), the sampler
\begin{center}
\begin{tabular}{l}
repeat $\gamma\leftarrow\Gamma\ul{\cK}(z,w)$;\\ 
\hspace{.2cm}take the distinguished U-atom back into the set of U-atoms;\\
\hspace{.2cm}\{so $||\gamma||$ increases by $1$ and now $\gamma\in\cK$\}\\
until $\mathrm{Bern}\left(\frac{2}{||\gamma||}\right)$;\\
return $\gamma$ 
\end{tabular}
\end{center}
is a Boltzmann sampler for $\cK$.

However, this sampler is not efficient enough, as it uses a massive amount of rejection
to draw a tree of large size. Instead, we use an early-abort
rejection algorithm, which allows us to ``simulate" the rejection step all along the generation, thus
making it possible to reject before the entire object is generated. We find it more convenient to use the
number of nodes, instead of leaves, as the parameter for rejection (the subtle advantage is that the 
generation process $\Gamma\ul{\cK}(z,w)$ builds the tree node by node). Notice that the number
 of leaves in an unrooted binary tree $\gamma$ is equal to $2+N(\gamma)$, with $N(\gamma)$ 
 the number of nodes of $\gamma$. Hence, the rejection step
in the sampler above can be replaced by a Bernoulli choice with parameter $2/(N(\gamma)+2)$. 
We now give the early-abort algorithm, which  repeats calling $\Gamma\ul{\cK}(z,w)$ while using
 a global counter $N$ that records the number of nodes  of the tree under construction.
\vspace{0.2cm}

\fbox{
\begin{tabular}{ll}
$\Gamma \cK(z,w)$:$\!\!$& repeat \\
&\hspace{0.2cm}$N:=0$; \{counter for nodes\}\\
&\hspace{0.2cm}Call $\Gamma\ul{\cK}(z,w)$\\
&\hspace{0.2cm}each time a node is built do\\
&\hspace{0.4cm}$N:=N+1$;\\
& \hspace{0.4cm}if $\mathrm{Bern}((N+1)/(N+2))$ continue;\\
&\hspace{0.4cm}otherwise reject and restart from the first line; od\\
&until the generation finishes;\\
&return the object generated by $\Gamma\ul{\cK}(z,w)$\\
&(taking the distinguished leaf back into the set of U-atoms)
\end{tabular}
} 

\begin{lemma}\label{lem:BoltzK}
The algorithm $\Gamma \cK(z,w)$ is a Boltzmann sampler for the class $\cK$ of asymmetric bicolored binary trees.
\end{lemma}
\begin{proof}
At each attempt, the call to $\Gamma\ul{\cK}(z,w)$ would output a rooted binary tree $\gamma$ if there was no early interruption. 
Clearly, the probability that the generation of $\gamma$ finishes without interruption is $\prod_{i=1}^{N(\gamma)}(i+1)/(i+2)=2/(N(\gamma)+2)$. Hence, each attempt is equivalent to doing\\ 
\centerline{$\gamma\leftarrow\Gamma\ul{\cK}(z,w)$; if $\mathrm{Bern}\left(\frac{2}{N(\gamma)+2}\right)$ return $\gamma$ else reject;}\\
Thus, the algorithm $\Gamma \cK(z,w)$ is equivalent to the algorithm given in the discussion 
preceding Lemma~\ref{lem:BoltzK}, hence $\Gamma \cK(z,w)$ is a Boltzmann sampler for the family 
$\cK$.
\end{proof}

\subsubsection{Boltzmann sampler for  irreducible dissections}\label{sec:sampI}
As stated in Proposition~\ref{prop:bijbin3conn}, 
the closure-mapping realises a combinatorial isomorphism between asymmetric 
bicolored binary trees (class $\cK$) and asymmetric bicolored irreducible dissections (class $\cI$).
Hence, the algorithm

\vspace{0.2cm}

\fbox{\begin{tabular}{ll}
$\Gamma \cI(z,w)$:$\!\!$& $\tau\leftarrow \Gamma \cK(z,w)$;\\
& return $\mathrm{closure}(\tau)$
\end{tabular}}

\vspace{0.2cm}

\noindent is a Boltzmann sampler for $\cI$.
In turn this easily yields a Boltzmann
sampler for the corresponding rooted class $\cJ$. Precisely, starting from an \emph{asymmetric} bicolored irreducible dissection,
 each of the 3 outer black vertices, which are \emph{distinguishable}, might be chosen as the root-vertex in order to obtain a rooted irreducible dissection. 
 Moreover the sets of L-atoms and U-atoms are slightly different for the classes $\cI$
 and $\cJ$; indeed, a rooted dissection has
 one more L-atom (the black vertex following the root-vertex in cw order around the outer face)
 and one more U-atom (all faces are U-atoms in $\cJ$, whereas only the inner faces are U-atoms 
 in $\cI$)\footnote{We have chosen to specify the sets of L-atoms and U-atoms in this way in order to state the isomorphisms $\cK\simeq\cI$ and $\cJa\simeq\vec{\cM_3}$.}.  This yields the identity
 \begin{equation}
 \cJ= 3\star\cZL\star\cZU\star\cI,
 \end{equation}
 which directly yields (by the sampling rules of Figure~\ref{table:rules}) a Boltzmann sampler $\Gamma\cJ(z,w)$
 for $\cJ$ from the Boltzmann sampler $\Gamma \cI(z,w)$.
 
 Finally, we obtain a Boltzmann sampler for rooted admissible dissections by a simple rejection procedure

 \vspace{.2cm}
 
 \fbox{
\begin{tabular}{ll}
$\Gamma \cJa(z,w)$:$\!\!$& repeat $\delta\leftarrow\Gamma \cJ(z,w)$ until $\delta\in\cJa$;\\
& return $\delta$
\end{tabular}
} 

\vspace{.2cm}

\subsubsection{Boltzmann sampler for rooted 3-connected maps}
The Boltzmann sampler for rooted irreducible dissections and the primal-map construction  
yield the following sampler for rooted 3-connected maps: 

\vspace{0.2cm}

\fbox{
\begin{tabular}{ll}
$\Gamma \vec{\cM_3}(z,w)$:$\!\!$& $\delta\leftarrow\Gamma \cJa(z,w)$;\\
& return $\mathrm{Primal}(\delta)$
\end{tabular}
} 

\vspace{0.2cm}

\noindent where  $\mathrm{Primal}(\delta)$ is 
 the rooted 3-connected map associated to $\delta$ (see Section~\ref{sec:primal_map}).

\subsubsection{Boltzmann sampler for edge-rooted 3-connected planar graphs}

To conclude,
 the Boltzmann sampler $\Gamma \vec{\cM_3}(z,w)$ yields a Boltzmann sampler
  $\Gamma \vec{\cG_3}(z,w)$ for edge-rooted 3-connected planar graphs, according to the
  isomorphism (Whitney's theorem) $\vec{\cMt}\simeq 2\star\vec{\cGt}$,
  
  \vspace{0.2cm}

  \fbox{
\begin{tabular}{ll}
$\Gamma \vec{\cG_3}(z,w)$:$\!\!$& return $\Gamma \vec{\cM_3}(z,w)$ (forgetting the embedding)
\end{tabular}
}

\subsection{Boltzmann sampler for 2-connected planar graphs}
\label{sec:2conn3conn}
The next step is to realise a Boltzmann sampler for 2-connected planar 
graphs from the Boltzmann sampler for edge-rooted 3-connected planar graphs obtained
in Section~\ref{sec:bolz3conn}. Precisely, we first describe a Boltzmann sampler for the 
class $\vec{\cGb}$ of edge-rooted 2-connected 
planar graphs, and subsequently obtain, by using rejection techniques, a Boltzmann sampler for the class $\cGbp$ of derived  
2-connected planar graphs (having a Boltzmann sampler for $\cGbp$ allows
us to go subsequently to connected planar graphs).

To generate edge-rooted 2-connected planar graphs, we use a  well-known decomposition, due 
to Trakhtenbrot~\cite{trak}, which ensures
that an edge-rooted 2-connected planar graph can be assembled from
edge-rooted 3-connected planar components. 
This decomposition deals with so-called \emph{networks} (following the terminology
of Walsh~\cite{Wa}), 
where a network is defined as a connected graph $N$ with
two distinguished vertices  $0$ and $\infty$ called \emph{poles}, such that the graph $N^*$
obtained by adding an edge between $0$ and $\infty$ is a 2-connected
planar graph. Accordingly, we refer to Trakhtenbrot's decomposition as the \emph{network decomposition}. Notice that networks are closely related to edge-rooted 2-connected planar graphs,
though not completely equivalent (see Equation~\eqref{eq:DB} below for the precise relation).

We rely
on~\cite{Wa} for the
description of the network decomposition. 
 A \emph{series-network} or $s$-network is a network made
of at least 2 networks connected \emph{in chain} at their poles, 
the $\infty$-pole of a network coinciding with the $0$-pole of the
following network in the chain. A
\emph{parallel network} or $p$-network is a network made of at least 2
networks connected \emph{in parallel}, so that their respective
$\infty$-poles and $0$-poles coincide. A \emph{pseudo-brick}
 is a network $N$ whose poles are not adjacent and such that $N^*$
is a 3-connected planar graph with at least 4 vertices. 
 A \emph{polyhedral network} or $h$-network is a network obtained by taking 
a pseudo-brick and substituting each edge $e$ of the pseudo-brick by a network $N_e$
(polyhedral networks 
establish a link between 2-connected and 3-connected planar graphs).

\begin{proposition}[Trakhtenbrot]
\label{prop:trak}
Networks with at least 2 edges are
partitioned into $s$-networks, $p$-networks and $h$-networks.
\end{proposition}

Let us explain how to obtain a recursive decomposition involving 
 the different families of networks. (We simply adapt the decomposition 
 formalised by Walsh~\cite{Wa} so as to have only positive signs.) 
 Let $\cD$, $\cS$, $\cP$, and $\cH$
be respectively the 
classes of networks, $s$-networks, $p$-networks, and $h$-networks,
where  the L-atoms are the vertices except the two poles, and the U-atoms
are the edges. In particular, $\cZU$ stands here for the class containing the link-graph as only object, i.e., the graph with one edge
connecting the two poles.
Proposition~\ref{prop:trak} ensures that
$$
\cD=\cZU+\cS+\cP+\cH.
$$

An $s$-network can be uniquely decomposed into a non-$s$-network (the
head of the chain) followed by a network (the trail of the chain),
which yields

$$
\cS=(\cZU+\cP+\cH)\star\cZL\star\cD.
$$

A $p$-network has a unique \emph{maximal} parallel decomposition into
a collection of at least two components that are not $p$-networks. Observe that
we consider here graphs without multiple edges, so that at most one of
these components is an edge. Whether there is one or no such
edge-component yields

$$
\cP=\cZU\star\Set_{\geq 1}(\cS+\cH)+\Set_{\geq 2}(\cS+\cH).
$$

 By definition, the class of $h$-networks corresponds to a 
U-substitution of networks in pseudo-bricks; and pseudo-bricks are 
exactly  edge-rooted 3-connected 
planar graphs. As a consequence (recall that $\cGt$ stands for the family of
 3-connected planar graphs),

$$
\cH=\vec{\cGt}\circ_U\cD.
$$

To sum up, we have the following grammar corresponding to the 
decomposition of networks into edge-rooted 3-connected planar graphs:

\vspace{0.3cm}

\includegraphics[width=10cm]{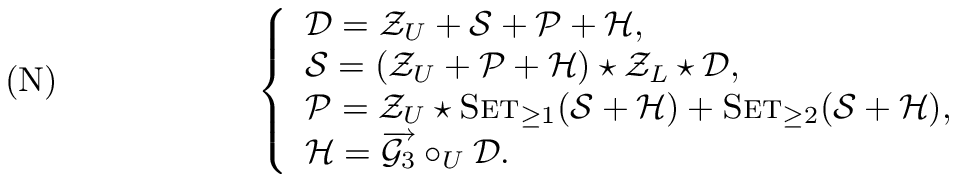}

\vspace{0.3cm}


Using the sampling rules (Figure~\ref{table:rules}), 
the decomposition grammar (N) is directly translated into a
Boltzmann sampler $\Gamma \cD(z,y)$ for networks, as given in Figure~\ref{fig:samp_networks}.
A network generated by $\Gamma \cD(z,y)$ is made of a series-parallel backbone $\beta$ (resulting
from the branching structures of the calls to $\Gamma\cS(z,y)$ and $\Gamma\cP(z,y)$) and a collection
of rooted 3-connected planar graphs that are attached at edges of $\beta$; clearly all these
3-connected components are obtained from independent calls to the Boltzmann sampler $\Gamma\cGtr(z,w)$, with $w=D(z,y)$.

\begin{figure}
\fbox{\begin{minipage}{12cm}
\ni\begin{tabular}{ll}
$\Gamma\cD(z,y)$:& Call $\Gamma\cZU(z,y)$ or $\Gamma\cS(z,y)$ or $\Gamma\cP(z,y)$ or $\Gamma\cH(z,y)$\\[.1cm]
&\hspace{.5cm}with respective 
probabilities $\frac{y}{D}$, $\frac{S}{D}$, $\frac{P}{D}$, $\frac{H}{D}$;\\[.1cm]
&return the network generated
\end{tabular}

\vspace{.4cm}

\ni\begin{tabular}{ll}
$\Gamma\cZU(z,y)$:& return the link-graph
\end{tabular}

\vspace{.4cm}

\ni\begin{tabular}{ll}
$\Gamma\cS(z,y)$:& $\gamma_1\leftarrow\Gamma(\cZU+\cP+\cH)(z,y)$;\\
&$\gamma_2\leftarrow\Gamma\cD(z,y)$;\\
&$\gamma\leftarrow\gamma_1$ in series with $\gamma_2$;\\
& return $\gamma$
\end{tabular}

\vspace{.4cm}

\ni\begin{tabular}{ll}
$\Gamma\cP(z,y)$:& Call $\Gamma\cP_1(z,y)$ or $\Gamma\cP_2(z,y)$\\
&\hspace{.5cm}with resp. probabilities $\frac{y\exp_{\geq 1}(S+H)}{P}$, $\frac{\exp_{\geq 2}(S+H)}{P}$;\\
&return the network generated
\end{tabular}

\vspace{.4cm}

\ni\begin{tabular}{ll}
$\Gamma\cP_1(z,y)$:& $k\leftarrow\Pois_{\geq 1}(S+H)$;\\
& $\gamma_1\leftarrow\Gamma(\cS+\cH)(z,w),\ldots,\gamma_k\leftarrow\Gamma(\cS+\cH)(z,w)$; \{ind. calls\}\\
&$\gamma\leftarrow(\gamma_1,\ldots,\gamma_k)$ in parallel;\\
& add to $\gamma$ an edge connecting the 2 poles;\\
& return $\gamma$
\end{tabular}

\vspace{.4cm}

\ni\begin{tabular}{ll}
$\Gamma\cP_2(z,y)$:& $k\leftarrow\Pois_{\geq 2}(S+H)$;\\
& $\gamma_1\leftarrow\Gamma(\cS+\cH)(z,w),\ldots,\gamma_k\leftarrow\Gamma(\cS+\cH)(z,w)$; \{ind. calls\}\\
&$\gamma\leftarrow(\gamma_1,\ldots,\gamma_k)$ in parallel;\\
& return $\gamma$
\end{tabular}

\vspace{.4cm}

\ni\begin{tabular}{ll}
$\Gamma \cH(z,y)$:& $\gamma\leftarrow\Gamma\cGtr(z,w)$,\ with $w=D(z,y)$;\\
&for each edge $e$ of $\gamma$ do\\
&\hspace{.5cm} $\gamma_e\leftarrow\Gamma\cD(z,y)$;\\
&\hspace{.5cm} substitute $e$ by $\gamma_e$;\\
&\hspace{.5cm} \{the poles of $\gamma_e$ are identified with the ends of $e$\\
&\hspace{.5cm} $\ $in a canonical way\}\\
& od;\\
& return $\gamma$
\end{tabular}

\vspace{.4cm}

\ni\begin{tabular}{ll}
$\Gamma(\cS+\cH)(z,y)$:&Call $\Gamma\cS(z,y)$ or $\Gamma\cH(z,y)$\\
&\hspace{.5cm}with resp. probabilities $\frac{S}{S+H}$, $\frac{H}{S+H}$;\\
&return the network generated
\end{tabular}

\vspace{.4cm}

\ni\begin{tabular}{ll}
$\Gamma(\cZU\!+\!\cP\!+\!\cH)(z,y)$:&$\!\!$Call $\Gamma\cZU(z,y)$ or $\Gamma\cP(z,y)$ or $\Gamma\cH(z,y)$\\
&\hspace{.5cm}$\!\!$with resp. probabilities $\frac{y}{y+P+H}$, $\frac{P}{y+P+H}$, $\frac{H}{y+P+H}$;\\
&$\!\!$return the network generated
\end{tabular}

\end{minipage}}
\caption{Boltzmann samplers for networks. All generating functions are assumed to be evaluated at $(z,y)$, i.e., $D:=D(z,y)$, $S:=S(z,y)$, $P:=P(z,y)$, and $H:=H(z,y)$.}
\label{fig:samp_networks}
\end{figure}

 The only terminal nodes of the
decomposition grammar are the classes $\cZL$, $\cZU$ (which are explicit), and the class $\vec{\cGt}$. 
Thus, the sampler $\Gamma \cD(z,y)$ and the auxiliary samplers $\Gamma \cS(z,y)$, $\Gamma \cP(z,y)$,
and $\Gamma \cH(z,y)$ are recursively specified in terms of $\Gamma \vec{\cG_3}(z,w)$,
where $w$ and $z$ are linked by 
$w=D(z,y)$.

Observe that each edge-rooted 2-connected planar graph different from the link-graph gives rise to two networks, obtained respectively by keeping or deleting the root-edge. This yields the identity
\begin{equation}
\label{eq:DB}
(1+\cZU)\star\vec{\cGb}=(1+\cD).
\end{equation}
From that point, a Boltzmann sampler is easily obtained for the family $\vec{\cGb}$ 
of edge-rooted 2-connected 
planar graphs. Define a procedure \textsc{AddRootEdge} that adds an edge connecting the two poles $0$ and $\infty$
of a network if they are not already adjacent, 
and roots the obtained graph at the edge $(0,\infty)$ directed
from $0$ to $\infty$.
The following sampler 
for $\vec{\cGb}$ is the counterpart of Equation~(\ref{eq:DB}).

\vspace{0.3cm}

\fbox{
\begin{tabular}{rl}
$\Gamma (1+\cD)(z,y)$: & $\!\!\!$ if $\mathrm{Bern}\left(\frac{1}{1+D(z,y)}\right)$ return the link-graph else return $\Gamma \cD(z,y)$; \\
$\Gamma \vec{\cG_2}(z,y)$: & $\!\!\!$ $\gamma \leftarrow \Gamma (1+\cD)(z,y)$;  \textsc{AddRootEdge}($\gamma$); return $\gamma$
\end{tabular}
}

\begin{lemma}\label{lem:netto2conn}
The algorithm $\Gamma \vec{\cG_2}(z,y)$ is a Boltzmann sampler for the class $\vec{\cGb}$ 
of edge-rooted 2-connected planar graphs.
\end{lemma}
\begin{proof}
Firstly, observe that  $\Gamma \vec{\cG_2}(z,y)$ outputs the link-graph either if 
the initial Bernoulli choice $X$ is 0, or if $X=1$ and the sampler $\Gamma\cD(z,y)$ 
picks up the link-graph. Hence the link-graph is returned with probability $(1+y)/(1+D(z,y))$, i.e.,
with probability $1/\vec{G_2}(z,y)$.

Apart from the link-graph, 
each graph $\gamma\in\vec{\cGb}$ appears twice in the class $\cE:=1+\cD$: once in 
$\cE_{|\gamma|,||\gamma||+1}$ (keeping the root-edge) and once in $\cE_{|\gamma|,||\gamma||}$ (deleting the root-edge).
Therefore, $\gamma$ has probability $E(z,y)^{-1}z^{|\gamma|}/|\gamma|!(y^{||\gamma||+1}+y^{||\gamma||})$ of being drawn by $\Gamma \vec{\cG_2}(z,y)$, where $E(z,y)=1+D(z,y)$ is the series of $\cE$. This probability simplifies to $z^{|\gamma|}/|\gamma|!y^{||\gamma||}/\vec{G_2}(z,y)$. Hence, $\Gamma \vec{\cG_2}(z,y)$ is a Boltzmann sampler
for the class $\vec{\cGb}$.  
\end{proof}

The last step is to obtain a Boltzmann sampler for derived 2-connected planar graphs 
(i.e., with a distinguished vertex that is not labelled and does not count
for the L-size) from
the Boltzmann sampler for edge-rooted 2-connected planar graphs (as we will see
in Section~\ref{sec:conn2conn}, derived 2-connected
planar graphs constitute the blocks to construct connected planar graphs).

We proceed in two steps. Firstly, 
we obtain a Boltzmann sampler for the U-derived class $\ul{\cGb}$ 
(i.e., with a distinguished undirected edge that does not count in the U-size).
Note that $\cF:=2\star\ul{\cGb}$ satisfies $\cF=\cZL\ \!\!\!^2\star\vec{\cGb}$.
Hence, $\Gamma\vec{\cGb}(z,y)$ directly yields a Boltzmann sampler $\Gamma\cF(z,y)$
(see the sampling rules in Figure~\ref{table:rules}). Since $\cF=2\star\ul{\cGb}$,  
a Boltzmann sampler for $\ul{\cGb}$ is obtained  by calling $\Gamma\cF(z,y)$ and then forgetting the direction 
 of the root.

Secondly, once we have a Boltzmann sampler $\Gamma\ul{\cG_2}(z,y)$ for the U-derived class $\ul{\cG_2}$, we just have to apply the procedure \UtoL (described in Section~\ref{sec:reject}) to the class $\cG_2$ in order to obtain a Boltzmann sampler $\Gamma \cGbp(z,y)$ for the L-derived class $\cGbp$. The procedure \UtoL can be successfully applied, because the ratio vertices/edges is bounded. Indeed,  each connected graph $\gamma$ 
satisfies $|\gamma|\leq ||\gamma||+1$, which easily yields  
$\alphaLU=2$
for the class $\cG_2$ (attained by the link-graph).

\subsection{Boltzmann sampler for connected planar graphs}
\label{sec:conn2conn}
Another well known graph decomposition, called the \emph{block-decomposition},
ensures that  a connected graph can be decomposed 
into 2-connected components. We take advantage of this decomposition in order to specify
a Boltzmann sampler for derived connected planar graphs from the Boltzmann sampler for 
derived 2-connected planar graphs obtained in the last section. 
Then, a further rejection step yields a Boltzmann sampler for connected planar graphs.

The  \emph{block-decomposition}  (see~\cite[p.10]{Ha} for
a detailed description) ensures that each derived
connected planar graph can be uniquely constructed in
the following way: take a set of derived 2-connected planar graphs and
attach them together, by merging their marked vertices into a unique marked
vertex. Then, for each unmarked vertex $v$ of each 2-connected
component, take a derived connected planar graph $\gamma_v$ and merge 
the marked vertex of
$\gamma_v$ with $v$ (this operation corresponds to an  
L-substitution). 
The block-decomposition gives rise to the following identity relating the classes $\cGcp$ and $\cGbp$:
\begin{equation}
\label{eq:2conn}
\cGcp=\Set\left(\cGbp\circ_L(\cZL\star\cGcp)\right).
\end{equation}
This is directly translated into the following Boltzmann sampler for $\cGcp$ using the
sampling rules of Figure~\ref{table:rules}. (Notice that the 2-connected blocks of a connected graph are built independently, each
block resulting from a  call to the Boltzmann sampler $\Gamma \cGbp(z,y)$,
where $z=xG_1\ \!\!\!'(x,y)$.)

\vspace{0.2cm}

\fbox{
\begin{tabular}{ll}
$\Gamma \cGcp(x,y)$:&$k\leftarrow \Pois (G_2\ \!\!\!'(z,y));\ \ [\mathrm{with}\ z=xG_1\ \!\!\!'(x,y)]$\\
& $\gamma\leftarrow (\Gamma \cGbp(z,y),\ldots,\Gamma
\cGbp(z,y))$; \{$k$ independent calls\} \\
& merge the $k$ components of $\gamma$ at their marked vertices;\\
&for each unmarked vertex $v$ of $\gamma$ do\\
& $\ \ \ \ \ $ $\gamma_v\leftarrow \Gamma \cGcp(x,y)$;\\
& $\ \ \ \ \ $ merge the marked vertex of $\gamma_v$ with $v$\\
& od;\\
&return $\gamma$.
\end{tabular}  
}

\vspace{.2cm}


Then, a Boltzmann sampler for connected planar graphs is simply obtained from
$\Gamma \cGcp(x,y)$ by using a rejection step so as to adjust the probability distribution:

\vspace{0.3cm}

\fbox{
\begin{tabular}{ll}
$\Gamma \cG_1(x,y)$:& repeat $\gamma\leftarrow \Gamma \cGcp(x,y)$\\
& \phantom{1}\hspace{.2cm}take the marked vertex $v$ back to the set of L-atoms;\\
& \phantom{1}\hspace{.2cm}(if we consider the labels, $v$ receives label $|\gamma|+1$)\\
& \phantom{1}\hspace{.2cm}\{this makes $|\gamma|$ increase by $1$, and $\gamma\in\cG_1$\}\\
& until $\ds\mathrm{Bern}\left(\frac{1}{|\gamma|}\right)$;\\
& return $\gamma$
\end{tabular}
}

\begin{lemma}
\label{lemma:connconnpoint}
The sampler $\Gamma \cG_1(x,y)$ is a Boltzmann sampler for connected planar graphs.
\end{lemma}
\begin{proof}
The proof is similar to the proof of Lemma~\ref{lem:LtoU}.
Due to the general property that $\cC_{n,m}$ identifies to $\cC'_{n-1,m}$, 
the sampler delimited inside the repeat/until loop draws each object $\gamma\in\cG_1$
with probability $G_1\ \!\!\!'(x,y)^{-1}\frac{x^{|\gamma|-1}}{(|\gamma|-1)!}y^{||\gamma||}$, i.e.,
with probability proportional to $|\gamma|\frac{x^{|\gamma|}}{|\gamma|!}y^{||\gamma||}$. Hence, according to Lemma~\ref{lemma:rej}, the sampler $\Gamma \cG_1(x,w)$
draws each object $\gamma\in\cG_1$ with probability proportional to $\frac{x^{|\gamma|}}{|\gamma|!}y^{||\gamma||}$, i.e., is a Boltzmann sampler for $\cG_1$.
\end{proof}

\subsection{Boltzmann sampler for planar graphs}
\label{sec:planconn}
A planar graph is classically decomposed into the set of its
connected components, yielding
\begin{equation}
\label{eq:CtoG} 
\cG=\Set(\cGc), 
\end{equation}
which translates to the following Boltzmann sampler for the class $\cG$ of planar graphs 
(the Set construction gives rise to a Poisson law, see Figure~\ref{table:rules}):

\vspace{0.2cm}

\fbox{\begin{tabular}{ll}
$\Gamma \cG(x,y)$:& $k\leftarrow\Pois(G_1(x,y))$;\\
& return $(\Gamma \cG_1(x,y),\ldots,\Gamma \cG_1(x,y))$ \{k independent calls\} 
\end{tabular}}

\begin{proposition}
\label{lemma:planconn}
The procedure  $\Gamma \cG(x,y)$ is a Boltzmann sampler for planar graphs. 
\end{proposition}

\section{Deriving an efficient sampler}
\label{sec:efficient}
We have completely described in Section~\ref{sec:decomp}  a mixed Boltzmann sampler 
$\Gamma \cG(x,y)$ for planar graphs. 
This sampler yields an exact-size uniform sampler and an approximate-size
uniform sampler for planar graphs: to sample at size $n$, call the sampler $\Gamma \cG(x,1)$ until the graph generated
has size $n$; to sample in a range of sizes $[n(1-\epsilon),n(1+\epsilon)]$, call the sampler $\Gamma \cG(x,1)$ until the graph generated has size in the range. 
These targetted samplers can be shown to have expected polynomial complexity,
of order $n^{5/2}$ for approximate-size sampling and  $n^{7/2}$
for exact-size sampling (we omit the proof since we will describe more
efficient samplers in this section). 

However, more is needed to
achieve the complexity 
stated in Theorem~\ref{theo:planarsamp1}, i.e., $O(n/\epsilon)$ 
for approximate-size sampling and
$O(n^2)$ for exact-size sampling.
The main problem of the sampler $\Gamma \cG(x,1)$ 
is that the typical size of a graph 
generated 
is small, 
so that the number of attempts to reach a large target size is prohibitive. 

In order to correct this effect, we design in this section a Boltzmann sampler 
for ``bi-derived" planar graphs, which are equivalent to bi-pointed planar 
graphs, i.e.,  with 2 distinguished 
vertices\footnote{In an earlier version of the article and in the conference version~\cite{Fu05a}, we derived  3 times---as prescribed by~\cite{DuFlLoSc04}---in order to get a singularity type $(1-x/\rho)^{-1/2}$ (efficient targetted samplers are obtained when taking $x=\rho(1-1/(2n))$). We have recently discovered that deriving 2 times (which 
yields a square-root singularity 
type $(1-x/\rho)^{1/2}$) and taking again $x=\rho(1-1/(2n))$ yields the same
complexities for the targetted samplers, with the advantage that the description and analysis  is significantly simpler 
(in the original article~\cite{DuFlLoSc04}, they prescribe
to take $x=\rho$ and to use some early abort techniques for square-root singularity
type, but it seems difficult to analyse the gain due to early abortion here,
since the Boltzmann sampler for planar graphs makes use of rejection techniques).   
 }. 
The intuition is that a Boltzmann sampler for
bi-pointed planar graphs gives more weight to large graphs, because a graph of size $n$
gives rise to $n(n-1)$ bi-pointed graphs. Hence, the probability of reaching a large size
is better (upon choosing suitably the value of the Boltzmann parameter). The fact that 
the graphs have to be pointed 2 times is due to the specific asymptotic behaviour of the
coefficients counting planar graphs, which has been recently analysed by 
Gim\'enez and Noy~\cite{gimeneznoy}.

\subsection{Targetted samplers for classes with square-root singularities.}
As we describe here, a mixed class $\cC$ with a 
certain type of singularities (square-root type)
gives rise to efficient approximate-size and exact-size samplers, provided
 $\cC$ has
a 
Boltzmann sampler such that the expected cost of generation is of the same order as the 
expected size of the object generated. 

\begin{definition}
Given a mixed class $\cC$, we define a \emph{singular point} 
of $\cC$ as a pair $x_0>0$,
$y_0>0$ such that the function $x\mapsto C(x,y_0)$ has a dominant
 singularity at $x_0$ (the radius of convergence is $x_0$).
\end{definition}

\begin{definition}\label{def:alpha_sing}
For $\alpha\in\mathbb{R}\backslash\mathbb{Z}_{\geq 0}$, a mixed class $\cC$
is called $\alpha$-singular if, for each singular point $(x_0,y_0)$ of $\cC$,
the function $x\mapsto C(x,y_0)$ has a unique dominant singularity at $x_0$ (i.e.,
$x_0$ is the unique singularity on the circle $|z|=x_0$) 
and admits a singular expansion of the form
$$
C(x,y_0)=P(x)+c_{\alpha}\cdot\left( x_0-x  \right)^{\alpha}+o\left( (x_0-x )^{\alpha}\right),
$$
where $c_{\alpha}$ is a constant, $P(x)$ is rational with no poles in the disk $|z|\leq x_0$, and where the expansion holds
in a so-called $\Delta$-neighbourhood of $x_0$, see~\cite{fla,flaod}.
In the special case $\alpha=1/2$, the class is said to have square-root singularities.
\end{definition}



\begin{lemma}\label{lem:square}
Let $\cC$ be a mixed class with square-root singularities,  and endowed 
with a Boltzmann sampler $\Gamma\cC(x,y)$. Let $(x_0,y_0)$ be a singular point of $\cC$. 
For any $n> 0$, define
$$
x_n:=\big(1-\tfrac{1}{2n}\big)\cdot x_0.
$$
Call $\pi_n$ ($\pi_{n,\epsilon}$, resp.) the probability that an object $\gamma$ generated by $\Gamma \cC(x_n,y_0)$ satisfies $|\gamma|=n$ 
($|\gamma|\in I_{n,\epsilon}:=[n(1-\epsilon),n(1+\epsilon)]$, resp.); and call $\sigma_n$ the expected size of the output of $\Gamma \cC(x_n,y_0)$.

Then $1/\pi_n$ is $O(n^{3/2})$, $1/\pi_{n,\epsilon}$ is $O(n^{1/2}/\epsilon)$, and $\sigma_n$ is $O(n^{1/2})$.

\end{lemma}
\begin{proof}
The so-called transfer theorems of singularity analysis~\cite{flaod} ensure that the coefficient
$a_n:=[x^n]C(x,y_0)$ satisfies, as $n\to\infty$, $a_n\mathop{\sim}_{n\to\infty}c\ \!x_0^{-n}n^{-3/2}$,
where $c$ is a positive constant. 
This easily yields the asymptotic bounds for $1/\pi_n$ and $1/\pi_{n,\epsilon}$, 
using the expressions $\pi_n=a_nx_n\ \!\!\!^n/C(x_n,y_0)$ and $\pi_{n,\epsilon}=\sum_{k\in I_{n,\epsilon}}a_kx_n\ \!\!\!^k/C(x_n,y_0)$. 

It is also an easy exercise to find the asymptotics of $\sigma_n$, using the 
formula (given in~\cite{DuFlLoSc04}) $\sigma_n=x_n\cdot\partial_x C(x_n,y_0)/C(x_n,y_0)$.
\end{proof}
Lemma~\ref{lem:square} suggests the following simple heuristic to obtain efficient 
targetted samplers.
For approximate-size sampling (exact-size sampling, resp.), repeat calling $\Gamma \cC(x_n,1)$ until the size of the object is in $I_{n,\epsilon}$ (is exactly $n$, resp.). 
(The parameter $y$ is useful if a target U-size $m$ is 
also given, as we will see for planar graphs in Section~\ref{sec:sample_edges}.)
The complexity of sampling will be good for a class $\cC$ that has square-root singularities and that has an efficient Boltzmann sampler. Indeed, for approximate-size sampling, the number of attempts to reach the target-domain $I_{n,\epsilon}$ (i.e., $\pi_{n,\epsilon}^{-1}$) is of order $n^{1/2}$, and for exact-size sampling, the number of attempts to reach the size $n$ (i.e., $\pi_{n}^{-1}$) is of order $n^{3/2}$.
If $\cC$ is endowed with a  
Boltzmann sampler  $\Gamma \cC(x,y)$ such that the expected
complexity of sampling at $(x_n,y_0)$ is of order $\sqrt{n}$ 
 (same order as the expected size $\sigma_n$), 
 then  
 the expected complexity 
 is typically $O(n/\epsilon)$ for approximate-size sampling and $O(n^2)$ for 
exact-size sampling, as we will see for planar graphs. 

Let us mention that the original article~\cite{DuFlLoSc04} uses a different heuristic.
The targetted samplers also repeat calling the Boltzmann sampler until 
the size of the object is in the target domain, but 
the parameter $x$ is chosen to be \emph{exactly} at the singularity $\rho$. 
The second difference is that, at each attempt, the generation is interrupted
if the size of the object goes beyond the target domain. We prefer to use
the simple heuristic discussed above, which does not require early
interruption techniques. In this way the samplers are easier to describe 
and to analyse.

In order to apply these techniques to  planar graphs, we have to derive
two times the class of planar graphs,
as indicated by the following two lemmas.

\begin{lemma}[\cite{fla}]
If a class $\cC$ is $\alpha$-singular, then the class $\cC'$ is $(\alpha-1)$-singular
(by the effect of derivation).
\end{lemma}

\begin{lemma}[\cite{gimeneznoy}]\label{lem:bi_der}
The class $\cG$ of planar graphs is $5/2$-singular, hence the 
class $\cG''$ of bi-derived planar graphs has square-root singularities.
\end{lemma} 

\subsection{Derivation rules for Boltzmann samplers}
As suggested by Lemma~\ref{lem:square}  and Lemma~\ref{lem:bi_der}, we will get good 
targetted samplers for planar graphs if we can describe 
 an efficient Boltzmann sampler for the class $\cG''$ of 
bi-derived planar graphs (a graph in $\cG''$ has two unlabelled vertices that are marked
specifically, say the first one is marked $\ast$ and the second one is marked $\star$).
Our Boltzmann sampler $\Gamma \cG''(x,y)$ ---to be presented in this section--- 
makes use of the decomposition of 
planar graphs into 3-connected components 
which we have already successfully used to obtain a Boltzmann sampler
for planar graphs in Section~\ref{sec:decomp}. This decomposition 
 can be formally translated into a decomposition grammar (with additional
unpointing/pointing operations). To obtain a Boltzmann sampler
for bi-derived planar graphs instead of planar graphs, the idea is simply to \emph{derive} 
this grammar 2 times.

As we explain here and as is well known in general, a decomposition grammar can be derived automatically. (In our framework, a decomposition grammar involves
the 5 constructions $\{+,\star,\Set_{\geq d},\circ_{L},\circ_{U}\}$.)

\begin{proposition}[derivation rules]\label{prop:der_rules}
The basic finite classes satisfy
$$
(\mathbf{1})'=0,\ \ \ (\cZ_L)'=1,\ \ \ (\cZ_U)'=0.
$$
The 5 constructions  satisfy the following derivation rules:
\begin{equation}
\left\{
\begin{array}{rcl}
(\cA+\cB)'&=&\cA'+\cB',\\
(\cA\star\cB)'&=&\cA'\star\cB+\cA\star\cB',\\
(\Set_{\geq d}(\cB))'&=&\cB'\star\Set_{\geq d-1}(\cB)\ \mathrm{for}\ d\geq 0,\ \ \ \mathrm{(with}\ \Set_{\geq -1}=\Set)\\
(\cA\circ_L\cB)'&=&\cB'\star(\cA'\circ_L\cB),\\
(\cA\circ_U\cB)'&=&\cA'\circ_U\cB+\cB'\star(\ul{\cA}\circ_U\cB).
\end{array}
\right.
\end{equation}
\end{proposition}
\begin{proof}
The derivation formulas for basic classes are trivial. 
The proof of the derivation rules for $\{+,\star,\circ_L\}$ are given in~\cite{BeLaLe}. Notice that the rule for $\Set_{\geq d}$ 
follows from the rule for $\circ_L$. (Indeed, $\Set_{\geq d}(\cB)=\cA\circ_{L}\cB$, 
where $\cA=\Set_{\geq d}(\cZL)$, which clearly satisfies $\cA'=\Set_{\geq d-1}(\cZL)$.) Finally, the proof of the rule for $\circ_U$ uses similar arguments as the proof of
the rule for $\circ_L$. In an object of $(\cA\circ_U\cB)'$, the distinguished atom is either on the core-structure
(in $\cA$), or is in a certain component (in $\cB$) that is substituted at a certain
U-atom of the core-structure. The first case yields the term $\cA'\circ_U\cB$,
and the second case yields the term $\cB'\star(\ul{\cA}\circ_U\cB)$.
\end{proof}
According to Proposition~\ref{prop:der_rules}, it is completely automatic to find a decomposition
grammar for a derived class $\cC'$  if we are given a decomposition grammar for $\cC$.

\subsection{Boltzmann sampler for bi-derived planar graphs}
We present in this section our Boltzmann sampler $\Gamma \cG''(x,y)$ for 
bi-derived planar graphs, with a quite similar approach 
to the one adopted in Section~\ref{sec:decomp}, and again a bottom-to-top presentation. 
At first  the closure-mapping allows us to obtain Boltzmann samplers for 3-connected
planar graphs marked in various ways. Then we go from 3-connected to bi-derived planar
graphs via networks, bi-derived 2-connected, and bi-derived connected planar graphs. 



 The complete scheme is illustrated in 
Figure~\ref{fig:scheme_bi_derived}, which is the counterpart of Figure~\ref{fig:scheme_unrooted}.

\begin{figure}
  \begin{center}
  \includegraphics[width=13.6cm]{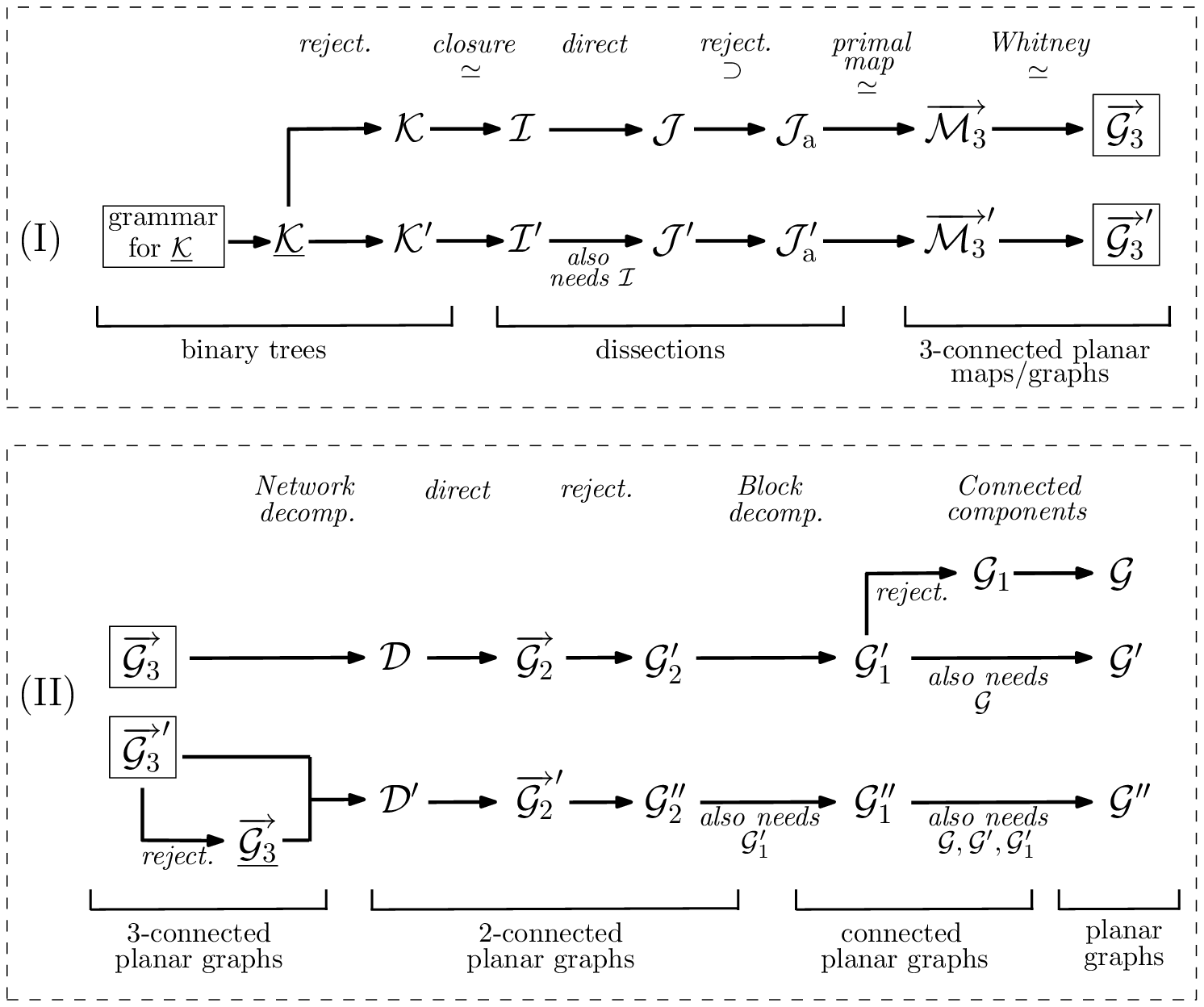}
\end{center}
    \caption{The complete scheme to obtain a Boltzmann sampler for 
    bi-derived planar graphs.}
  \label{fig:scheme_bi_derived}
\end{figure}

\subsubsection{Boltzmann samplers for derived binary trees.}\label{sec:sampKp}
We have already obtained in Section~\ref{sec:boltz_binary_trees} a Boltzmann sampler for the class
$\cK$ of unrooted asymmetric binary trees. Our purpose here is to derive a Boltzmann
sampler for the derived class $\cK'$. Recall  that we have also described in Section~\ref{sec:boltz_binary_trees} a Boltzmann sampler for the U-derived
class $\ul{\cK}$, which satisfies the completely recursive decomposition grammar~(\ref{eq:grammar})
(see also Figure~\ref{fig:grammar}). 
Hence, we have to apply the procedure \UtoL described in Section~\ref{sec:reject} to the class $\cK$
in order to obtain
a Boltzmann sampler $\Gamma \cK'(z,w)$ from $\Gamma \ul{\cK}(z,w)$.
For this we have to check that $\alphaLU$ is finite for the class $\cK$.
It is easily proved that a bicolored binary tree with $m$ leaves has $m-2$ nodes,
and that at most $\lfloor 2(m-3)/3\rfloor$ of the nodes are  black.
In addition, there exist trees with $3i+3$ leaves and $2i$ black nodes 
(those with all leaves incident to black nodes). 
Hence, for the class $\cK$, the parameter $\alphaLU$
is equal to $2/3$.
Therefore the procedure \UtoL can be applied to the class $\cK$.

\subsubsection{Boltzmann samplers for derived rooted dissections and 3-connected maps}\label{sec:sampIp}
Our next step is to obtain Boltzmann samplers for derived irreducible dissections, in order to go subsequently to 3-connected maps. As expected
we take advantage of the closure-mapping. Recall that the closure-mapping
realises the isomorphism $\cK\simeq\cI$ between the class $\cK$ of asymmetric
binary trees and the class $\cI$ of asymmetric irreducible dissections.
There is no problem in deriving an isomorphism, so the closure-mapping
also realises the isomorphism $\cK'\simeq\cI'$.
Accordingly we have the following Boltzmann sampler for the class $\cI'$:

\vspace{.2cm}

\fbox{
\begin{tabular}{ll}
$\Gamma \cI'(z,w)$:$\!\!$& $\tau\leftarrow\Gamma \cK'(z,w)$;\\
& $\delta\leftarrow\mathrm{closure}(\tau)$;\\
& return $\delta$
\end{tabular}
}

\vspace{.2cm}

\noindent where the discarded L-atom is the same in $\tau$ and in $\delta$.

Then, we easily obtain a Boltzmann sampler for the corresponding \emph{rooted}
class $\cJ'$. Indeed, the equation $\cJ=3\star\cZ_L\star\cZ_U\star\cI$ that relates $\cI$ and $\cJ$ yields $\cJ'=3\star\cZ_U\star\cI+3\star\cZ_L\star\cZ_U\star\cI'$.
Hence,  using the sampling rules of Figure~\ref{table:rules}, we obtain a Boltzmann sampler 
$\Gamma \cJ'(z,w)$ from the Boltzmann samplers $\Gamma \cI(z,w)$ and $\Gamma \cI'(z,w)$. 

From that point, 
we obtain a Boltzmann sampler for the derived rooted dissections that are admissible. 
As $\cJa\subset\cI$,
we also have $\cJa'\subset\cJ'$, which yields
the following Boltzmann sampler for $\cJa'$:

\vspace{.2cm}

\fbox{
\begin{tabular}{ll}
$\Gamma \cJa'(z,w)$:$\!\!$& repeat $\delta\leftarrow\Gamma \cJ'(z,w)$\\
& until $\delta\in\cJa'$;\\
& return $\delta$
\end{tabular}
}

\vspace{.2cm}

Finally, using the isomorphism $\cJa\simeq\cMtr$ 
(primal map construction, Section~\ref{sec:primal_map}),
which yields $\cJa'\simeq\cMtr'$, we obtain a Boltzmann 
samplers for derived  rooted 3-connected maps:

\vspace{.2cm}

\fbox{
\begin{tabular}{ll}
$\Gamma \cMtr'(z,w)$:$\!\!$& $\delta\leftarrow\Gamma \cJa'(z,w)$;\\
& return $\mathrm{Primal}(\delta)$
\end{tabular}
} 

\vspace{.2cm}

\noindent where the returned rooted 3-connected map inherits the distinguished L-atom of $\delta$.

\subsubsection{Boltzmann samplers for derived rooted 3-connected planar graphs.}
\label{sec:derived_3conn}
As we have seen in Section~\ref{sec:equiv}, Whitney's theorem states that any 3-connected planar graph has two embeddings on the sphere (which differ by a reflection).
Clearly the same property holds for 3-connected planar graphs that have additional
marks. (We have already used this observation in Section~\ref{sec:equiv} 
for rooted graphs,
$\cMtr\simeq 2\star\cGtr$, in order to obtain a Boltzmann sampler for $\cGtr$.)
Hence $\cMtr'\simeq 2\star\cGtr'$, which yields
the following Boltzmann sampler for $\cGtr'$:

\vspace{.2cm}

\fbox{
\begin{tabular}{ll}
$\Gamma \cGtr'(z,w)$:$\!\!$& return $\Gamma \cMtr'(z,w)$;\\
& (forgetting the embedding)
\end{tabular}
}

\vspace{.2cm}

The next step (in Section~\ref{sec:derived_networks}) is to go to derived networks.
This asks for a derivation of the decomposition grammar for networks,
which involves not only the classes $\cGtr$, $\cGtr'$, but also the U-derived 
class $\ul{\cGtr}$. Hence, we also need a Boltzmann sampler for $\ul{\cGtr}$.

To this aim we just have to apply the procedure \LtoU to the class $\cGtr$.
By the Euler relation, a 3-connected planar graph with $n$ vertices has at most
$3n-6$ edges (equality holds for triangulations). Hence, the parameter $\alphaUL$ is equal to $3$ for the class $\cGtr$,
so \LtoU can be successfully applied to $\cGtr$, yielding a Boltzmann sampler
for $\ul{\cGtr}$ from the Boltzmann sampler for $\cGtr'$.

\subsection{Boltzmann samplers for derived networks.}\label{sec:derived_networks}
Following the general scheme shown in Figure~\ref{fig:scheme_bi_derived}, our aim is now
to obtain a Boltzmann samplers for the class $\cD'$ of derived
 networks. Recall that the decomposition grammar for $\cD$ has allowed us to obtain a Boltzmann sampler for $\cD$ from a Boltzmann sampler for $\cGtr$. Using the derivation rules (Proposition~\ref{prop:der_rules}) injected in the grammar~(N), we obtain the following decomposition grammar for $\cD'$:

\vspace{.2cm}

\noindent\includegraphics[width=11.3cm]{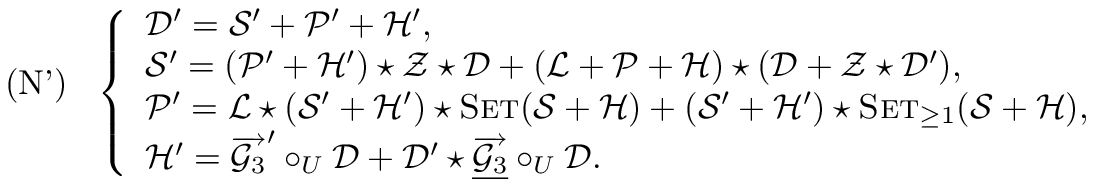}

The only terminal classes in this grammar are $\cGtr'$ and $\ul{\cGtr}$.
Hence, the sampling rules of Figure~\ref{table:rules} yield a Boltzmann 
sampler for $\cD'$ from the Boltzmann samplers for $\cGtr'$ and $\ul{\cGtr}$ which
we have obtained in Section~\ref{sec:derived_3conn}. The sampler $\Gamma\cD'(z,y)$
looks similar (though with more cases) to the one for $\Gamma\cD(z,y)$ 
given in Figure~\ref{fig:samp_networks}.

\subsection{Boltzmann samplers for bi-derived 2-connected planar graphs.}\label{sec:sampDp}
The aim of this section is to obtain Boltzmann samplers for the class $\cGbp'$
of bi-derived 2-connected planar graphs (after the Boltzmann sampler for $\cGbp$ obtained in Section~\ref{sec:2conn3conn}),
in order to go subsequently to bi-derived connected planar graphs.

At first, the Boltzmann sampler for $\cD'$  yields a Boltzmann sampler
for the class $\cGbr'$. Indeed the identity $(1+\cD)=(1+\cZ_U)\star\vec{\cG_2}$ is derived as $\cD'=(1+\cZ_U)\star\cGbr'$, which yields the following sampler,

\vspace{.2cm}

\fbox{
\begin{tabular}{ll}
$\Gamma \vec{\cG_2}'(z,y)$:$\!\!$& $\gamma\leftarrow\Gamma \cD'(z,y)$;\\
& $\textsc{AddRootEdge}(\gamma)$;\\
& return $\gamma$
\end{tabular}
} 

\vspace{.2cm}

\noindent where \textsc{AddRootEdge} has been defined in Section~\ref{sec:2conn3conn}.
The proof that this is a Boltzmann sampler for $\cGbr'$ 
is similar to the proof of Lemma~\ref{lem:netto2conn}.

Next we describe a Boltzmann sampler for the class $\ul{\cGb}'$.
As we have seen in Section~\ref{sec:2conn3conn}, $\ul{\cGb}$ and $\cGbr$ are related  by the 
identity $2\star\ul{\cGb}=\cZ_L\ \!\!\!^2\star\cGbr$. Hence, if we define $\cF:=2\star\ul{\cGb}$,
we have $\cF'=\cZL\ \!\!\!^2\star\cGbr'+2\star\cZL\star\cGbr$.
Hence, the sampling rules of Figure~\ref{table:rules} yield a Boltzmann sampler $\Gamma \cF'(z,y)$  for the class
$\cF'$. Clearly, as $\cF'=2\star\ul{\cGb}'$, a Boltzmann sampler for 
$\ul{\cGb}'$ is obtained by calling $\Gamma \cF'(z,y)$ and forgetting the direction of the root.

Finally, the procedure \UtoL yields (when applied to $\cGbp$) from the Boltzmann
sampler for $\ul{\cGb}'$ to a Boltzmann sampler for $\cGbp'$.
The procedure can be successfully applied, as the class $\cGbp$ satisfies
$\alphaLU=1$ (attained by the link-graph).

\subsubsection{Boltzmann sampler for bi-derived connected planar graphs.}\label{sec:sampCp}
The block-decomposition makes it easy to obtain a Boltzmann sampler
for the class $\cGcp'$ of bi-derived connected planar graphs (this decomposition has already allowed us to obtain a Boltzmann
sampler for $\cGcp$ in Section~\ref{sec:conn2conn}). Recall that the block-decomposition yields the
identity
$$
\cGcp=\Set\left(\cGbp\circ_L(\cZ_L\star\cGcp)\right), 
$$ 
which is derived as
$$
\cGcp'=(\cGcp+\cZ_L\star\cGcp')\star\cGbp'\circ_L(\cZ_L\star\cGcp)\star\cGcp.
$$

As we already have Boltzmann samplers for the classes $\cGbp'$ and $\cGcp$,
the sampling rules of Figure~\ref{table:rules} yield a Boltzmann sampler
$\Gamma \cGcp'(x,y)$  for the class $\cGcp'$. 
Observe that the 2-connected blocks of a graph generated
by $\Gamma \cGcp'(x,y)$ 
are obtained as independent calls to $\Gamma \cGbp(z,y)$ and 
$\Gamma \cGbp'(z,y)$, where $z$ and $x$ are related
by the change of variable $z=xG_1\ \!\!\!'(x,y)$.

\subsubsection{Boltzmann samplers for bi-derived planar graphs}\label{sec:sampGp}
We can now achieve our goal, i.e., obtain a Boltzmann sampler for the class
$\cG''$ of bi-derived planar graphs. For this purpose, we simply derive
twice the identity
$$
\cG=\Set(\cGc),
$$
which yields successively the identities
$$
\cG'=\cGcp\star\cG,
$$
and
$$
\cG''=\cGcp'\star\cG+\cGcp\star\cG'.
$$
From the first identity and $\Gamma \cG(x,y)$, $\Gamma \cGcp(x,y)$, we get 
a Boltzmann sampler $\Gamma \cG'(x,y)$ for the class $\cG'$. Then, 
from the second identity and $\Gamma \cG(x,y)$, $\Gamma \cG'(x,y)$, $\Gamma \cGcp(x,y)$, $\Gamma \cGcp'(x,y)$, we get 
a Boltzmann sampler $\Gamma \cG''(x,y)$ for the class $\cG''$.

\section{The targetted samplers for planar graphs}\label{sec:final_smap}
The Boltzmann sampler $\Gamma\cG''(x,y)$---when tuned  as indicated in 
Lemma~\ref{lem:square}---yields efficient exact-size and 
approximate-size random samplers for planar graphs, 
with the complexities as stated in Theorem~\ref{theo:planarsamp1} and Theorem~\ref{theo:planarsamp2}. Define the algorithm:

\vspace{.2cm}

\begin{tabular}{ll}
$\textsc{SamplePlanar}(x,y)$:& $\gamma\leftarrow\Gamma\cG''(x,y)$;\\
& give label $|\gamma|+1$ to the vertex marked $\star$\\
& and label $|\gamma|+2$
to the marked vertex $\ast$\\
&(thus $|\gamma|$ increases by $2$, and $\gamma\in\cG$);\\
& return $\gamma$
\end{tabular}

\subsection{Samplers according to the number of vertices}
\label{sec:sample_vertices} 

Let $\rho_G$ be the radius of convergence of $x\mapsto G(x,1)$. Define
$$x_n:=\big(1-\tfrac{1}{2n}\big)\cdot \rho_G.$$

\noindent For $n\geq 1$, the exact-size sampler is

\vspace{0.2cm}

$\frak{A}_n$: repeat $\gamma \leftarrow \textsc{SamplePlanar}(x_n,1)$ until 
$|\gamma|=n$; return~$\gamma$.

\vspace{0.2cm}

\noindent For $n\geq 1$ and $\epsilon >0$, the approximate-size sampler is 

\vspace{0.2cm}

 $\frak{A}_{n,\epsilon}$: repeat $\gamma\leftarrow \textsc{SamplePlanar}(x_n,1)$ until $|\gamma|\in [n(1-\epsilon),n(1+\epsilon)]$; return $\gamma$.

\subsection{Samplers according to the numbers of vertices and edges}
\label{sec:sample_edges} 

For any $y>0$, we denote by $\rho_G(y)$ 
the radius of convergence of $x\mapsto G(x,y)$. Let $\mu(y)$ be the function defined as
$$
\mu(y):=-y\frac{\mathrm{d}\rho_G}{\mathrm{d}y}(y)/\rho_G(y).
$$
As proved in~\cite{gimeneznoy} (using the so-called quasi-power theorem), for a fixed $y>0$, 
a large graph drawn by the Boltzmann sampler $\Gamma \cG''(x,y)$ has a ratio
edges/vertices concentrated around the value $\mu(y)$ as $x$ approaches the 
radius of convergence of $x\mapsto G(x,y)$. This yields a relation between 
the secondary parameter $y$ and the ratio edges/vertices. If we want a 
ratio edges/vertices close to a target value $\mu$, we have to choose $y$
so that $\mu(y)=\mu$. 
It is shown in~\cite{gimeneznoy} that the function $\mu(y)$ is strictly 
increasing on $(0,+\infty)$, with $\lim \mu(y)=1$ as  $y\to 0$ and $\lim \mu(y)=3$ as $y\to +\infty$. 
As a consequence, $\mu(y)$ 
has an inverse function $y(\mu)$ defined on $(1,3)$. (In addition, $\mu\mapsto y(\mu)$  can be evaluated
with good precision from the analytic equation it satisfies.) 
We define $$x_n(\mu):=\big(1-\tfrac{1}{2n}\big)\cdot\rho_G(y(\mu)).$$ 
For $n\geq 1$ and $\mu\in(1,3)$, the exact-size sampler  is

\vspace{0.2cm}

$\ol{\frak{A}}_{n,\mu}$:$\!$ repeat $\gamma\leftarrow \textsc{SamplePlanar}(x_n(\mu),y(\mu))$ 
until ($|\gamma|\!=\!n$ and $||\gamma||\!=\!\lfloor \mu n\rfloor)$; return~$\gamma$.

\vspace{0.2cm}

\noindent For $n\geq 1$, $\mu\in(1,3)$, and $\epsilon>0$, the approximate-size sampler is 

\vspace{0.2cm}

\begin{tabular}{ll}
 $\ol{\frak{A}}_{n,\mu,\epsilon}$:& repeat $\gamma\leftarrow \textsc{SamplePlanar}(x_n(\mu),y(\mu))$\\
 & until ($|\gamma|\in [n(1-\epsilon),n(1+\epsilon)]$ and $\frac{||\gamma||}{|\gamma|}\in [\mu (1-\epsilon),\mu (1+\epsilon)]$); \\
& return $\gamma$.
\end{tabular}

\vspace{.3cm}

\noindent The complexity of the samplers is analysed in Section~\ref{sec:complexity}.


\section{Implementation and experimental results}\label{sec:implement}

\subsection{Implementation}
We have completely implemented the random samplers for planar graphs described
in Section~\ref{sec:efficient}.
First we evaluated with good precision---typically 20 digits---the generating functions of the families of planar graphs that intervene in the 
decomposition (general, connected, 2-connected, 3-connected), derived up to 2 times.
The calculations have been carried out in Maple using the analytic expressions of Gim\'enez and
Noy for the generating functions~\cite{gimeneznoy}.  
We have performed the evaluations for values of the parameter $x$ associated with a bunch of reference target sizes in logarithmic scale,
$n=\{10^2, 10^3, 10^4,10^5,10^6\}$. 
From the evaluations of the generating functions, we have  
computed the vectors of real values that are associated to the random choices to 
be performed during the generation, e.g., a Poisson law vector with parameter 
$G_1(x)$ (the EGF
of connected planar graphs) is used for drawing the number of connected components of the
graph. 

The second step has been the implementation of the random sampler in Java. 
To build the graph all along the generation process, it proves more convenient to manipulate a data structure specific to  planar maps rather than planar graphs. The advantage is also  that the graph to be generated will be equipped
with an explicit (arbitrary) planar embedding. Thus if the graph generated is to be drawn in the
plane, we do not need to call the rather involved algorithms for embedding a planar graph.
Planar maps are suitably manipulated using the so-called
\emph{half-edge structure}, where each half-edge occupies a memory block containing a pointer
to the opposite half-edge along the same edge and to the next half-edge in ccw order around the
incident vertex. Using the half-edge structure, it proves very easy to implement in cost
$O(1)$ all primitives
used for building the graph---typically, merging two components at a common vertex or edge.  
Doing this, the actual complexity of implementation corresponds to the complexity of the random samplers as stated in Theorem~\ref{theo:planarsamp1} and Theorem~\ref{theo:planarsamp2}: 
linear for approximate-size sampling
and quadratic for exact-size sampling. In practice, generating a graph of size of order $10^5$
takes a few seconds on a standard computer.

\begin{figure}
  \begin{center}
    \includegraphics[width=8cm]{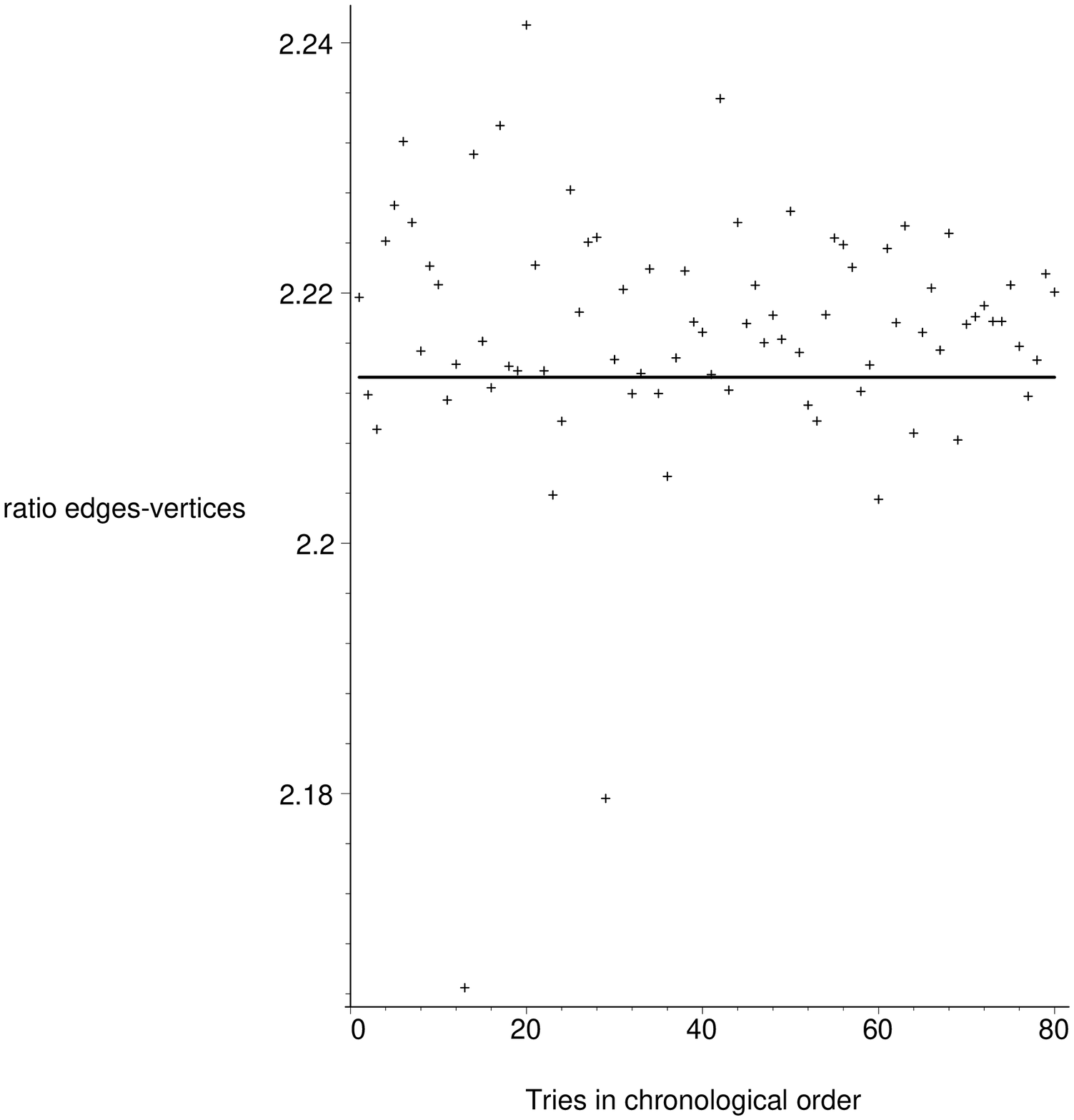}
  \end{center}
  \caption{Ratio edges/vertices observed on
  a collection $\gamma_1,\ldots,\gamma_{80}$ 
  of 80 random connected planar graphs of size at least $10^4$; each graph $\gamma_i$ yields
  a point at coordinates $(i,\mathrm{Rat}(\gamma_i))$, where $\mathrm{Rat}(\gamma)$
  is the ratio given by the number of edges divided by the number of vertices of $\gamma$.} 
  \label{fig:exp_ratio}
\end{figure}

\begin{figure}
  \begin{center}
    \includegraphics[width=8cm]{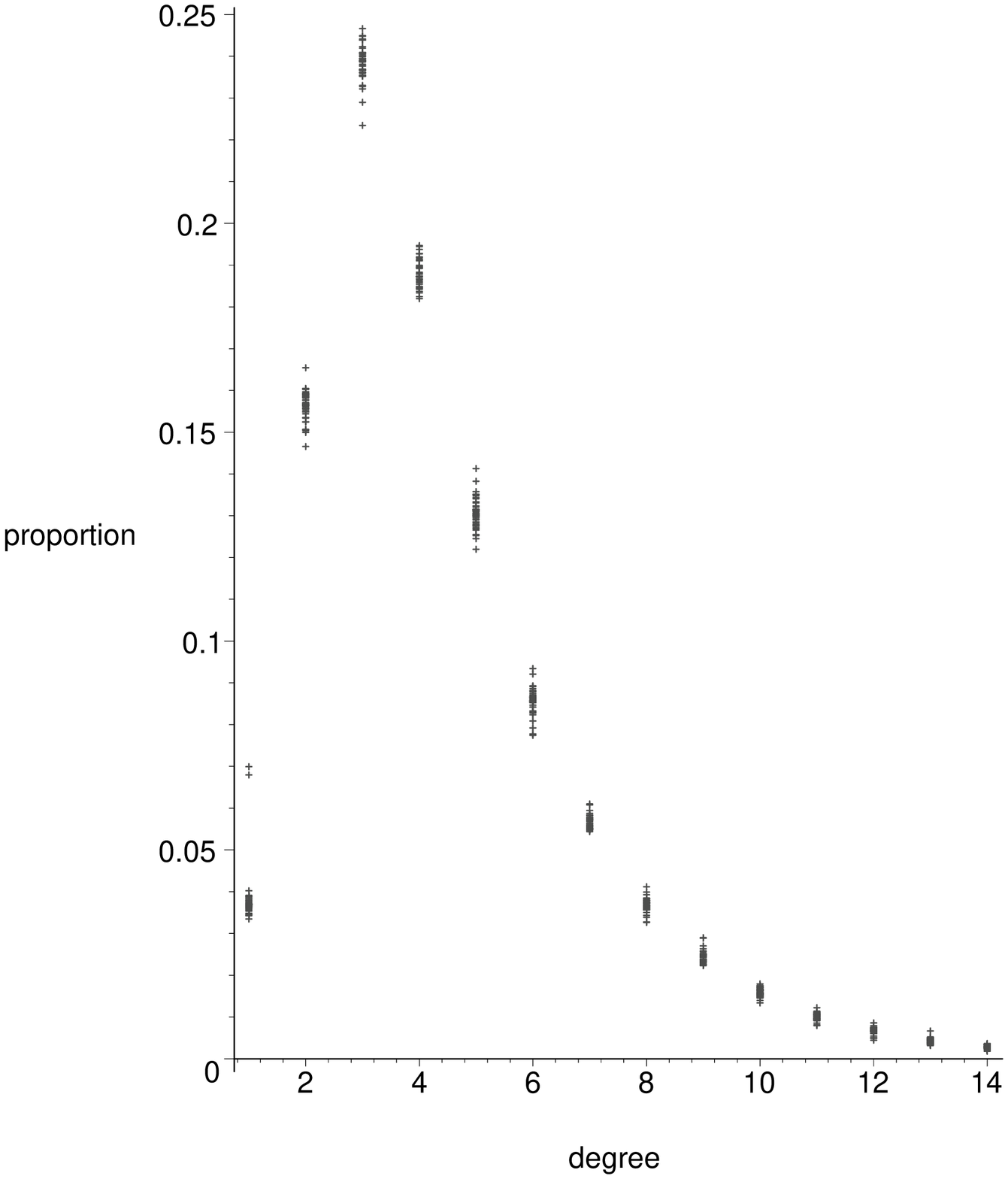}
  \end{center}
  \caption{The distribution of vertex degrees observed on
  a collection $\gamma_1,\ldots,\gamma_{80}$ 
  of 80 random connected planar graphs of size at least $10^4$. 
  Each graph $\gamma$ yields points at coordinates
  $(1,Z^{(1)}(\gamma)), (2,Z^{(2)}(\gamma)),\ldots,(d,Z^{(d)}(\gamma))$, 
  where $d$ is the maximal degree of $\gamma$ and,
  for $1\leq k\leq d$, $Z^{(k)}(\gamma)$
  is the proportion of vertices of $\gamma$ that have degree $k$.}
  \label{fig:exp_degree}
\end{figure}

\subsection{Experimentations.}  The good complexity of our random samplers allows us to observe
statistical properties of parameters
on very large random planar graphs---in the range of sizes $10^5$---where the asymptotic regime is already visible. We focus here on
 parameters that are known or
expected to be concentrated around a limit value. Note that the 
experimentations are on connected planar graphs instead of general
planar graphs. (It is slightly easier to restrict the implementation to
 connected graphs, which are conveniently manipulated using the half-edge 
 data structure.) However, from the works of Gim\'enez and Noy~\cite{gimeneznoy} and previous work
by MacDiarmid et al.~\cite{McD05}, a random planar graph consists of a 
huge connected component, plus other components whose total expected size is 
$O(1)$. Thus, statistical properties like those stated in 
Conjecture~\ref{conj:planar} should be the same for
random planar graphs as for random connected planar graphs.

\vspace{0.1cm}

\noindent\emph{Number of edges.} First we have checked that the random variable $X_n$ that counts the number of edges in a random
connected planar graph with $n$ vertices is concentrated. 
Precisely, Gim\'enez and Noy have proved that
$Y_n:=X_n/n$ converges in law to a constant $\mu\approx 2.213$, 
(they also show that the fluctuations are gaussian of magnitude
$1/\sqrt{n}$).  Figure~\ref{fig:exp_ratio} shows in ordinate the ratio edges/vertices for a collection of 80
random connected planar graphs of size at least $10^4$ drawn by our sampler. As we can see, the ratios are concentrated
around the horizontal line $y=\mu$, agreeing 
with the convergence result of Gim\'enez
and Noy.  

\vspace{0.1cm}

\noindent\emph{Degrees of vertices.} Another parameter of interest 
is the distribution of the degrees of vertices in a random planar graph. For a planar graph $\gamma$
  with $n$ vertices, 
we denote by $N^{(k)}(\gamma)$ the number of vertices of $\gamma$ that have $k$ neighbours.
Accordingly, $Z^{(k)}(\gamma):=N^{(k)}(\gamma)/n$ is the proportion of vertices of degree $k$ in $\gamma$.
It is known from Gim\'enez and Noy that, for $k=1,2$, the random variable $Z^{(k)}$ converges in law to an explicit constant.
Figure~\ref{fig:exp_degree} shows in abscissa the parameter $k$ and in ordinate the value of $Z^{(k)}$ for 
a collection of 80 random connected planar graphs of size at least $10^4$ drawn
by our sampler. Hence, the vertical line at abscissa $k$
is occupied by 80 points whose ordinates correspond to the values taken by $Z^{(k)}$ for each of 
the graphs. As we can see, for $k$ small---typically $k<<\log n$---the values of $Z^{(k)}$ are concentrated around a constant. This leads us to the following conjecture.
 
\begin{conjecture}\label{conj:planar}
For every $k\geq 1$, let $Z^{(k)}_n$ be the random variable denoting the proportion of vertices of degree $k$ in 
a random planar graph with $n$ vertices taken uniformly at random. Then $Z_n^{(k)}$ converges in law  
to an explicit constant $\pi^{(k)}$ as $n\to\infty$; and $\sum_k\pi^{(k)}=1$. 
\end{conjecture}

Let us mention some progress on this conjecture. It has recently been proved
in~\cite{DrGiNo07} that
the expected values $\mathbb{E}(Z_n^{(k)})$ converge as $n\to\infty$ to constants
$\pi^{(k)}$ that are computable and satisfy $\sum_k\pi^{(k)}=1$. Hence, what 
remains to be shown regarding the conjecture is the concentration property.

\section{Analysis of the time complexity}\label{sec:complexity}
This whole section is dedicated to the proof of the complexities of the targetted 
random samplers. We show that the expected complexities of the targetted
samplers $\frak{A}_n$, $\frak{A}_{n,\epsilon}$, $\ol{\frak{A}}_{n,\mu}$, and $\ol{\frak{A}}_{n,\mu,\epsilon}$, as described in Section~\ref{sec:final_smap}, 
are respectively
$O(n^2)$, $O(n/\epsilon)$, $O_{\mu}(n^{5/2})$, and $O_{\mu}(n/\epsilon)$
respectively (the dependency in $\mu$ in not analysed for the sake of simplicity).

Recall that the targetted samplers call $\Gamma\cG''(x,y)$ (with suitable values of $x$
and $y$) until the size parameters are in the target domain. Accordingly,
the complexity analysis is done in two steps. In the first step, we estimate the probability
of hitting the target domain, which allows us to reduce the complexity analysis
to the analysis of the 
expected complexity of the pure Boltzmann sampler $\Gamma\cG''(x,y)$. 
We use a specific
notation to denote such an expected complexity:

\begin{definition}
Given a class $\cC$ endowed with a Boltzmann sampler $\Gamma\cC(x,y)$,
we denote by $\Lambda\cC(x,y)$ the expected combinatorial complexity\footnote{See the discussion
on the complexity model after the statement of Theorem~\ref{theo:planarsamp2} in the introduction.} 
of a call to $\Gamma\cC(x,y)$ (note that $\Lambda\cC(x,y)$ depends not only
on $\cC$, but also on a specific Boltzmann sampler for $\cC$).
\end{definition}

Typically the values $(x,y)$ have to be close to a singular point of $\cG$ in order
to draw graphs of large size. Hence, in the second step, our aim is to bound
$\Lambda\cG''(x,y)$ when $(x,y)$ converges to a given singular point $(x_0,y_0)$
of $\cG$. To analyse $\Lambda\cG''(x,y)$, our approach is again from bottom to top,
as the description of the sampler in Section~\ref{sec:efficient} 
(see also the general scheme
summarized in Figure~\ref{fig:scheme_bi_derived}). At each step we give asymptotic bounds 
for the expected complexities of the Boltzmann samplers when the parameters
approach a singular point. This study requires the knowledge of the singular behaviours
of all series involved in the decomposition of bi-derived planar graphs,
which are recalled in Section~\ref{sec:sing_beh}.

\subsection{Complexity of rejection: the key lemma}
The following simple lemma will be extensively used, firstly
to reduce the complexity analysis of the targetted samplers to the one 
of pure Boltzmann samplers, secondly to estimate the effect of the rejection steps
on the expected complexities of the Boltzmann samplers.

\begin{lemma}[rejection complexity]
\label{lem:target}
Let $\frak{A}$ be a random sampler on a combinatorial class $\cC$ according to a probability distribution $\mathbb{P}$, and let $p: \cC\to [0,1]$ be a function on $\cC$, called the rejection function.
Consider the rejection algorithm

$\frak{A}_{\mathrm{rej}}$: repeat $\gamma\leftarrow\frak{A}$ until $\Bern(p(\gamma))$ return $\gamma$.

Then the expected complexity $\mathbb{E}(\frak{A}_{\mathrm{rej}})$ of $\frak{A}_{\mathrm{rej}}$ 
and the expected complexity $\mathbb{E}(\frak{A})$ of $\frak{A}$ are related by
\begin{equation}
\mathbb{E}(\frak{A}_{\mathrm{rej}})=\frac{1}{p_{\mathrm{acc}}}\mathbb{E}(\frak{A}),
\end{equation}
where $p_{\mathrm{acc}}:=\sum_{\gamma\in\cC}\mathbb{P}(\gamma)p(\gamma)$ is the probability of success
of $\frak{A}_{\mathrm{rej}}$ at each attempt.
\end{lemma}
\begin{proof}
The quantity $\mathbb{E}(\frak{A}_{\mathrm{rej}})$ satisfies the recursive equation
$$
\mathbb{E}(\frak{A}_{\mathrm{rej}})=\mathbb{E}(\frak{A})+(1-p_{\mathrm{acc}})\mathbb{E}(\frak{A}_{\mathrm{rej}}).
$$
Indeed, a first attempt, with expected complexity $\mathbb{E}(\frak{A})$, is always 
needed; and in case of rejection, occurring with probability $(1-p_{\mathrm{acc}})$, the sampler restarts in the 
same way as when it is launched.
\end{proof}

As   a corollary we obtain the following useful formulas to estimate the effect
of rejection in Boltzmann samplers when going from L-derived (vertex-pointed) to U-derived (edge-pointed)
graphs and vice-versa.

\begin{corollary}[Complexity of changing the root]\label{lem:change_root}
Let $\cA$ be a mixed combinatorial class such that the constants $\alphaUL:=\mathrm{max}_{\gamma\in\cA}\frac{||\gamma||}{|\gamma|}$ and $\alphaLU:=\mathrm{max}_{\gamma\in\cA}\frac{|\gamma|}{||\gamma||}$ are finite. Define 
$c:=\alphaUL\cdot\alphaLU$.
\begin{itemize}
\item
Assume $\cA'$ is equipped with a Boltzmann sampler, and let $\Gamma \ul{\cA}(x,y)$ be the Boltzmann sampler for $\ul{\cA}$ obtained by applying $\LtoU$---as defined
in Section~\ref{sec:reject}---to $\cA$. Then
$$
\Lambda \ul{\cA}(x,y)\leq c\cdot \Lambda \cA'(x,y).
$$
\item
Assume $\ul{\cA}$ is equipped with a Boltzmann sampler, and let $\Gamma \cA'(x,y)$ be the Boltzmann sampler for $\cA'$ obtained by applying $\UtoL$---as defined
in Section~\ref{sec:reject}---to $\cA$. Then
$$
\Lambda \cA'(x,y)\leq c\cdot \Lambda \ul{\cA}(x,y).
$$
\end{itemize}
\end{corollary}
\begin{proof}
Let us give the proof for \LtoU (the other case is proved in a similar way). By definition of \LtoU
the probability of the Bernoulli choice at each attempt in $\Gamma\ul{\cA}(x,y)$ 
is at least $\frac{1}{\alphaUL}\mathrm{min}_{\gamma\in\cA}\frac{||\gamma||}{|\gamma|}$,
i.e.,  at least $1/(\alphaUL\cdot\alphaLU)$. Hence the probability $\pacc$ of success
at each attempt is at least $1/c$.
Therefore, by Corollary~\ref{lem:change_root},
$\Lambda \ul{\cA}(x,y)=\Lambda \cA'(x,y)/\pacc\leq c\cdot\Lambda\cA'(x,y)$.
\end{proof}

\subsection{Reduction to analysing the expected complexity of Boltzmann samplers}

We prove here that analysing the expected complexities of the targetted samplers
reduces to analysing the expected complexity $\Lambda\cG''(x,y)$ when $(x,y)$
approaches a singular point. (Recall that a singular point $(x_0,y_0)$ for a class $\cC$ is such that the function $x\mapsto C(x,y_0)$ has a dominant singularity at $x_0$.)

\begin{claim}
\label{claim:eq}
Assume that for every singular point $(x_0,y_0)$ of $\cG$, the expected complexity
of the Boltzmann sampler for $\cG''$ satisfies\footnote{In this article all convergence 
statements are meant ``from below'', 
i.e., $x\to x_0$ means that $x$ approaches $x_0$ while staying smaller than $x_0$.}
\begin{equation}\label{eq:claim}
\Lambda\cG''(x,y_0)=O((x_0-x)^{-1/2})\ \ \mathrm{as}\ x\to x_0.
\end{equation}
Then the expected complexities of the targetted samplers $\frak{A}_n$, 
$\frak{A}_{n,\epsilon}$, 
 $\ol{\frak{A}}_{n,\mu}$, and $\ol{\frak{A}}_{n,\mu,\epsilon}$---as defined in Section~\ref{sec:final_smap}---are respectively
 $O(n^2)$, $O(n/\epsilon)$, $O_{\mu}(n^{5/2})$, and $O_{\mu}(n/\epsilon)$.

In other words, proving~\eqref{eq:claim} is enough to prove the 
complexities of the random samplers for planar graphs, as stated in 
Theorem~\ref{theo:planarsamp1}
 and Theorem~\ref{theo:planarsamp2}.
\end{claim}
\begin{proof}
Assume that  (\ref{eq:claim}) holds.
Let $\pi_{n,\epsilon}$ ($\pi_n$, resp.) be the probability that the output
of $\textsc{SamplePlanar}(x_n,1)$ ---with $x_n=(1-1/2n)\cdot\rho_{G}$--- has size in $I_{n,\epsilon}:=[n(1-\epsilon),n(1+\epsilon)]$ (has size $n$, resp.).
According to Lemma~\ref{lem:target}, the expected complexities of the 
exact-size and approximate-size samplers with respect to vertices ---as
described in Section~\ref{sec:sample_vertices}--- satisfy
$$
\mathbb{E}(\frak{A_n})=\frac{\Lambda \cG''(x_n,1)}{\pi_n},\ \ \ \ \ \ \  \mathbb{E}(\frak{A_{n,\epsilon}})=\frac{\Lambda \cG''(x_n,1)}{\pi_{n,\epsilon}}. 
$$
Equation~(\ref{eq:claim}) ensures that, when $n\to\infty$, $\Lambda \cG''(x_n,1)$ is $O(n^{1/2})$.
In addition, according to Lemma~\ref{lem:bi_der}, $\cG''$ is $1/2$-singular (square-root
singularities). Hence, by Lemma~\ref{lem:square}, $1/\pi_n$ is $O(n^{3/2})$ 
 and $1/\pi_{n,\epsilon}$  is $O(n^{1/2}/\epsilon)$. Thus, $\mathbb{E}(\frak{A_n})$ is $O(n^2)$ and 
$\mathbb{E}(\frak{A_{n,\epsilon}})$ is $O(n/\epsilon)$. 

The proof for the samplers with respect to vertices and edges is a bit more technical.
Consider a planar graph $\gamma$ drawn by the sampler 
$\textsc{SamplePlanar}(x_n(\mu),y(\mu))$. 
In view of the proof for the exact-size sampler, define 
$$\barpi_{n\wedge\mu}:=\mathbb{P}(||\gamma||\!=\!
\lfloor \mu n\rfloor, |\gamma|=n),\ \  \barpi_{\mu|n}:=\mathbb{P}(||\gamma||\!\!=\!\!\lfloor \mu n\rfloor\ |\ |\gamma|\!\!=\!\!n),\ \ \pi_n:=\mathbb{P}(|\gamma|\!\!=\!\!n).$$ 

\noindent In view of the proof for the approximate-size sampler, define 
$$\barpi_{n\wedge\mu,\epsilon}:=\mathbb{P}(|\gamma|\in[n(1-\epsilon),n(1+\epsilon)],\ ||\gamma||/|\gamma|\in[\mu(1-\epsilon),\mu(1+\epsilon)]),$$ 
$$\barpi_{\mu | n,\epsilon}:=\mathbb{P}(||\gamma||/|\gamma|\in[\mu(1-\epsilon),\mu(1+\epsilon)]\ |\ |\gamma|\in[n(1-\epsilon),n(1+\epsilon)]),$$ and 
$$\pi_{n,\epsilon}:=\mathbb{P}(|\gamma|\in[n(1-\epsilon),n(1+\epsilon)]).$$ 
Notice that $\barpi_{n\wedge\mu}=\barpi_{\mu|n}\cdot\pi_n$ and $\barpi_{n\wedge\mu,\epsilon}=\barpi_{\mu | n,\epsilon}\cdot\pi_{n,\epsilon}$. Moreover, Lemma~\ref{lem:target} 
ensures that
$$
\mathbb{E}(\ol{\frak{A}}_{n,\mu})=\frac{\Lambda \cG''(x_n(\mu),y(\mu))}{\barpi_{n\wedge\mu}},\ \ \ \ \ \ \  \mathbb{E}(\ol{\frak{A}}_{n,\mu,\epsilon})=\frac{\Lambda \cG''(x_n(\mu),y(\mu))}{\barpi_{n\wedge\mu,\epsilon}}. 
$$
 It has been shown by Gim\'enez and 
Noy~\cite{gimeneznoy} (based
on the quasi-power theorem) that, for a fixed $\mu\in(1,3)$,  
$1/\barpi_{\mu|n}$ is 
$O_\mu(n^{1/2})$ as $n\to\infty$ (the dependency in $\mu$ is not discussed here
for the sake of simplicity).
Moreover, Lemma~\ref{lem:square} ensures that 
$1/\pi_n$ is $O_{\mu}(n^{3/2})$ as $n\to\infty$. 
Hence, $1/\barpi_{n,\mu}$ is $O_{\mu}(n^{2})$.
Finally Equation~(\ref{eq:claim}) ensures that  
$\Lambda \cG''(x_n(\mu),y(\mu))$ is $O_{\mu}(n^{1/2})$, 
therefore $\mathbb{E}(\ol{\frak{A}}_{n,\mu})$ is $O_{\mu}(n^{5/2})$.

For the approximate-size samplers, the results of Gim\'enez and Noy (central limit theorems) ensure that,
when $\mu\in(1,3)$ and $\epsilon>0$ are fixed and $n\to \infty$, 
$\barpi_{\mu | n,\epsilon}$ converges to 1. 
In addition, Lemma~\ref{lem:square} ensures that $1/\pi_{n,\epsilon}$ is $O_{\mu}(n^{1/2}/\epsilon)$. 
Hence, $1/\barpi_{n\wedge\mu,\epsilon}$ is $O_{\mu}(n^{1/2}/\epsilon)$. Equation~(\ref{eq:claim}) implies that  $\Lambda \cG''(x_n(\mu),y(\mu))$ is $O_{\mu}(n^{1/2})$, hence 
$\mathbb{E}(\ol{\frak{A}}_{n,\mu,\epsilon})$ is $O_{\mu}(n/\epsilon)$.
\end{proof}

From now on, our aim is to prove that, for any singular point $(x_0,y_0)$ of $\cG$,
$\Lambda\cG''(x,y_0)$ is $O((x_0-x)^{-1/2})$ as $x\to x_0$.

\subsection{Expected sizes of Boltzmann samplers}
Similarly as for the expected complexities, 
it proves convenient to use specific notations for the expected sizes 
associated to  Boltzmann samplers, and to state some of their basic properties.

\begin{definition}[expected sizes]
Let $\cC$ be a mixed combinatorial class, and let $(x,y)$ be admissible for $\cC$ (i.e.,
 $C(x,y)$ converges). Define respectively the expected L-size and the expected U-size 
at $(x,y)$ as the quantities
$$
|\cC|_{(x,y)}:=\frac{1}{C(x,y)}\sum_{\gamma\in\cC}|\gamma|\frac{x^{|\gamma|}}{|\gamma|!}y^{||\gamma||}=x\frac{\partial_x C(x,y)}{C(x,y)},
$$
$$||\cC||_{(x,y)}:=\frac{1}{C(x,y)}\sum_{\gamma\in\cC}||\gamma||\frac{x^{|\gamma|}}{|\gamma|!}y^{||\gamma||}=y\frac{\partial_yC(x,y)}{C(x,y)}.
$$
\end{definition}

We will need the following two simple lemmas at some points of the analysis.

\begin{lemma}[monotonicity of expected sizes]\label{lem:monotonicity}
Let $\cC$ be a mixed class. 
\begin{itemize}
\item
For each fixed $y_0>0$, the expected L-size $x\mapsto |\cC|_{(x,y_0)}$ is increasing with $x$. 
\item
For each fixed $x_0>0$, the expected U-size $y\mapsto |\cC|_{(x_0,y)}$ is increasing with $y$.
\end{itemize}
\end{lemma}
\begin{proof}
As noticed in~\cite{DuFlLoSc04} 
(in the labelled framework), the derivative of the function 
$f(x):=|\cC|_{(x,y_0)}$ is equal to $1/x$ multiplied by the variance of the L-size of an object under the 
Boltzmann distribution at $(x,y_0)$. Hence $f'(x)\geq 0$ for $x>0$, so $f(x)$ is increasing with $x$.
Similarly the derivative of 
$g(y):=||\cC||_{(x_0,y)}$ is equal to $1/y$ multiplied by the variance of the U-size of an object under the 
Boltzmann distribution at $(x_0,y)$, hence $g(y)$ is increasing with $y$ for $y>0$. 
\end{proof}

\begin{lemma}[divergence of expected sizes at singular points]\label{lem:exp_size}
Let $\cC$ be an $\alpha$-singular class and let $(x_0,y_0)$
be a singular point of $\cC$. Then, as $x\to x_0$:
\begin{itemize}
\item
if $\alpha>1$, the expected size $x\mapsto |\cC|_{(x,y_0)}$ converges to a 
positive constant,
\item
if $0<\alpha<1$, the expected size $x\mapsto |\cC|_{(x,y_0)}$ diverges and is of order $(x_0-x)^{\alpha-1}$.
\end{itemize}
\end{lemma}
\begin{proof}
Recall that $|\cC|_{(x,y_0)}=x\cdot C'(x,y_0)/C(x,y_0)$, and $\cC'$ is $(\alpha-1)$-singular
if $\cC$ is $\alpha$-singular. 
Hence, if $\alpha>1$, both functions $C(x,y_0)$ and $C'(x,y_0)$ converge to positive 
constants as $x\to x_0$, so that $|\cC|_{(x,y_0)}$ also converges to a positive constant.
 If $0<\alpha<1$, $C(x,y_0)$ still converges, but $C'(x,y_0)$ diverges, of order $(x_0-x)^{\alpha-1}$
 as $x\to x_0$. Hence $|\cC|_{(x,y_0)}$ is also of order $(x_0-x)^{\alpha-1}$.
\end{proof}

\subsection{Computation rules for the expected complexities of Boltzmann samplers}

Thanks to Claim~\ref{claim:eq}, the complexity analysis is now reduced to estimating the 
expected complexity $\Lambda\cG''(x,y)$ when $(x,y)$ is close to a singular point of $\cG$.
For this purpose, 
we introduce explicit rules to compute $\Lambda\cC(x,y)$ if $\cC$
is specified from other classes by a decomposition grammar. These rules will be combined with 
 Lemma~\ref{lem:target} and Corollary~\ref{lem:change_root} (complexity due to the rejection steps) 
in order to get a precise asymptotic bound for $\Lambda\cG''(x,y)$.


We can now formulate the computation rules for the expected complexities.

\begin{figure}
\begin{tabular}{l|l}
Construction & $\ \ \ \ \ \ \ \ \ \ $Expected complexity\\
\hline\hline
$\cC=\cA+\cB$ & $\Lambda \cC(x,y)=1+\frac{A(x,y)^{\phantom{f}}}{C(x,y)}\Lambda \cA(x,y)+\frac{B(x,y)}{C(x,y)}\Lambda \cB(x,y)$\\[.2cm]
$\cC=\cA\star\cB$ & $\Lambda \cC(x,y)=\Lambda \cA(x,y)+\Lambda \cB(x,y)$\\[.1cm]
$\cC=\Set_{\geq d}(\cB)$ & $\Lambda \cC(x,y)=\frac{\exp_{\geq d-1}(B(x,y))}{\exp_{\geq d}(B(x,y))}B(x,y)\cdot(1+\Lambda \cB(x,y))$\\[.2cm]
$\cC=\cA\circ_L\cB$ & $\Lambda \cC(x,y)=\Lambda \cA(B(x,y),y)+|\cA|_{(B(x,y),y)}\cdot\Lambda \cB(x,y)$\\[.2cm]
$\cC=\cA\circ_U\cB$ & $\Lambda \cC(x,y)=\Lambda \cA(x,B(x,y))+||\cA||_{(x,B(x,y))}\cdot\Lambda \cB(x,y)$
\end{tabular}
\caption{The expected complexities of Boltzmann samplers specified using the sampling rules
for the constructions $\{+,\star,\Set_{\geq d},\circ_L,\circ_U\}$ (as given in Figure~\ref{table:rules}) satisfy 
explicit equations. There $\exp_{\geq -1}(z)=\exp(z)$ and, for $d\geq 0$, $\exp_{\geq d}(z)=\sum_{k\geq d}z^k/k!$.}
\label{fig:comp_rules}
\end{figure}

\begin{lemma}[computation rules for expected complexities]\label{lem:comp_rules}
Let $\cC$ be a class obtained from simpler classes $\cA$, $\cB$
by means of one of the constructions $\{+,\star,\Set_{\geq d},\circ_L,\circ_U\}$.

If $\cA$ and $\cB$ are equipped with Boltzmann samplers, let $\Gamma\cC(x,y)$
be the Boltzmann sampler for $\cC$ obtained from the sampling rules of Figure~\ref{table:rules}.
Then there are explicit rules, as given in Figure~\ref{fig:comp_rules}, 
to compute the expected complexity of $\Gamma\cC(x,y)$ from
 the expected complexities of $\Gamma\cA(x,y)$ and $\Gamma\cB(x,y)$. 
\end{lemma}
\begin{proof}
\noindent\emph{Disjoint union:}  $\Gamma\cC(x,y)$ first flips a coin, which (by convention) has 
 unit cost in the combinatorial complexity. 
 Then $\Gamma\cC(x,y)$ either calls
$\Gamma\cA(x,y)$ or $\Gamma\cB(x,y)$ with respective probabilities $A(x,y)/C(x,y)$ and $B(x,y)/C(x,y)$.

\noindent\emph{Product:}  $\Gamma\cC(x,y)$ calls $\Gamma\cA(x,y)$ and then $\Gamma\cB(x,y)$,
which yields the formula. 

\noindent\emph{L-substitution:} $\Gamma\cC(x,y)$ calls $\gamma\leftarrow\Gamma\cA(B(x,y),y)$ and then replaces each L-atom of $\gamma$ by an  object generated by $\Gamma\cB(x,y)$. Hence, in average, the first step takes time $\Lambda\cA(B(x,y),y)$ and the second step
takes time $|\cA|_{(B(x,y),y)}\cdot\Lambda\cB(x,y)$.  

\noindent\emph{$\Set_{\geq d}$:} note that $\Set_{\geq d}(\cB)$ is equivalent to $\cA\circ_L\cB$, where $\cA:=\Set_{\geq d}(\cZL)$, which has generating function $\exp_{\geq d}(z):=\sum_{k\geq d}z^k/k!$. A Boltzmann sampler $\Gamma\cA(z,y)$ simply consists in drawing an integer under a conditioned Poisson law
$\Pois_{\geq d}(z)$, which is done by a simple iterative loop. As the number of iterations is equal
to the value that is returned (see~\cite{DuFlLoSc04} for a more detailed discussion), the expected cost of generation for $\cA$ is equal to the expected size, i.e.,
$$
\Lambda\cA(z,y)=|\cA|_{(z,y)}=z\frac{\exp_{\geq d}\ \!\!\!'(z)}{\exp_{\geq d}(z)}=z\frac{\exp_{\geq d-1}(z)}{\exp_{\geq d}(z)}.
$$ 
Hence, from the formula for $\Lambda(\cA\circ_L\cB)(x,y)$, we obtain the formula for $\Set_{\geq d}$.

\noindent\emph{U-substitution:} the formula for $\circ_U$ is proved similarly as the one for $\circ_L$.
\end{proof}

\begin{remark}\label{rk:finite}
When using the computation rules of Figure~\ref{fig:comp_rules} in a recursive way,
we have to be careful to check beforehand that all the expected complexities
that are involved are finite. Otherwise there is the danger of getting weird identities like ``$\sum_{k\geq 0}2^k=1+2\sum_{k\geq 0}2^k$, so $\sum_{k\geq 0}2^k=-1$.''
\end{remark}

\subsection{Analytic combinatorics of planar graphs}\label{sec:sing_beh}

Let $\cC$ be an  $\alpha$-singular class (see Definition~\ref{def:alpha_sing}).
  A very useful remark to be used all along the 
analysis of the expected complexities is the following:
 if $\alpha\geq 0$, the function $C(x,y_0)$ converges when $x\to x_0$,
and the limit has to be a positive constant;
whereas if $\alpha< 0$, the function $C(x,y_0)$  diverges to $+\infty$ 
 and is of order $(x_0-x)^{\alpha}$.

In this section, we review the degrees of singularities of the series of all classes (binary trees, dissections, 3-connected, 2-connected, connected, and general planar graphs) that are involved 
in the decomposition of planar graphs. We will use extensively this information to estimate
the expected complexities of the Boltzmann samplers in Section~\ref{sec:as_bounds}.

\begin{lemma}[bicolored binary trees]\label{lem:sing_bin}
Let $\cR=\cRb+\cRw$ be the class of rooted bicolored binary trees, which is specified by
the system
$$
\cRb=\cZ_L\star(\cZ_U+\cRw)^2,\ \ \ \cRw=(\cZU+\cRb)^2.
$$
Then the classes $\cRb$, $\cRw$ are $1/2$-singular.
The class $\ul{\cK}$ ($\cK$) of rooted (unrooted, resp.) asymmetric bicolored binary
trees is $1/2$-singular ($3/2$-singular, resp.). In addition, these two classes have
the same singular points as $\cR$.
\end{lemma}
\begin{proof}
The classes $\cRb$ and $\cRw$ satisfy a decomposition grammar that has a 
strongly connected dependency graph.
Hence, by a classical theorem of  Drmota, Lalley, Woods~\cite{fla}, the generating functions of these classes
 have square-root singular type. 
Notice that, from the decomposition grammar~\eqref{eq:grammar},
 the class $\ul{\cK}$ can be expressed 
as a positive polynomial in $\cZL$, $\cZU$, $\cRb$, and $\cRw$.
Hence $\ul{\cK}$ inherits the singular points and the square-root
singular type from $\cRb, \cRw$.
Finally, the generating function of $\cK$ is classically obtained as a subtraction
(a tree has one more vertices than edges, so subtract the series counting the trees rooted
at an edge from the series counting the trees rooted at a vertex). The leading square-root singular terms cancel out due to the subtraction, leaving  a leading singular term of degree $3/2$.
\end{proof}

\begin{lemma}[irreducible dissections, from~\cite{FuPoSc05}]\label{lem:sing_diss}
The class $\cJ$ of rooted irreducible dissections is $3/2$-singular and has 
the same singularities as $\cK$.
\end{lemma}
\begin{proof}
The class $\cJ$ is equal to $3\star\cZL\star\cZU\star\cI$, which 
is isomorphic to $3\star\cZL\star\cZU\star\cK$, so $\cJ$ has the same
singular points and singularity type as $\cK$. 
\end{proof}

\begin{lemma}[rooted 3-connected planar graphs~\cite{BeRi}]\label{lem:sing_3_conn}
The class $\vec{\cG_3}$ of edge-rooted 3-connected planar graphs is $3/2$-singular;
and the class $\ul{\vec{\cG_3}}$ of U-derived edge-rooted 3-connected planar graphs
is $1/2$-singular. These classes have the same singular points as~$\cK$.
\end{lemma}
\begin{proof}
The series $\vec{G_3}(z,w)$ has been proved in~\cite{Mu} to have a rational 
expression in terms of the two series $\Rb(z,w)$ and $\Rw(z,w)$ of rooted
bicolored binary trees. This property is easily shown to be stable by taking derivatives, so the same property holds for the series $\ul{\vec{G_3}}(z,w)$.
It is proved in~\cite{BeRi,BeGa} 
that the singular points of  $\cGtr$ 
are  the same as those of  $\cRb$ and $\cRw$.
Hence, the singular expansion of $\vec{G_3}(z,w)$ at any singular point is simply obtained from
the ones  of $\Rb(z,w)$ and $\Rw(z,w)$; one finds that the square-root terms
cancel out, leaving a leading singular term of degree $3/2$.  
The study of $\ul{\vec{\cG_3}}$ is similar. First, the rooting operator does not change the 
singular points (as it multiplies a coefficient $(n,m)$ only by a factor $m$),
hence, $\ul{\vec{\cG_3}}$ has the same singular points as $\cRb,\cRw$, which
ensures that the singular expansion of $\ul{\vec{G_3}}(z,w)$ can be obtained from those
of $\cRb$ and $\cRw$. One finds that the leading singular term is this time of the 
square-root type.
\end{proof}

\begin{lemma}[networks, from~\cite{BeGa}]\label{lem:sing_networks}
The classes $\cD$, $\cS$, $\cP$, and $\cH$ of networks are $3/2$-singular, and
these classes have the same singular points.
\end{lemma}

\begin{lemma}[2-connected, connected, and general planar graphs~\cite{gimeneznoy}]\label{lem:sing_planar}
The classes $\cG_2$, $\cG_1$, $\cG$ of 2-connected, connected, and general planar graphs are all $5/2$-singular. In addition, the singular points of $\cG_2$ are the same
as those of networks, and the singular points are the same in $\cG_1$ as in $\cG$.
\end{lemma}

\subsection{Asymptotic bounds on the expected complexities of Boltzmann samplers}
\label{sec:as_bounds}
This section is dedicated to proving the asymptotic bound $\Lambda\cG''(x,y_0)=O((x_0-x)^{-1/2})$.
For this purpose we adopt again a bottom-to-top approach, following 
the scheme of Figure~\ref{fig:scheme_bi_derived}. 
For each class $\cC$ appearing in this scheme,  
we provide an asymptotic bound for the expected complexity of the Boltzmann sampler 
in a neighbourhood of any fixed  singular point of $\cC$. In the end we arrive
at the desired estimate of $\Lambda\cG''(x,y_0)$.

\subsubsection{Complexity of the Boltzmann samplers for binary trees}\label{sec:comp_binary_trees}

\begin{lemma}[U-derived bicolored binary trees]\label{lem:comp_uK}
Let $(z_0,w_0)$ be a singular point of $\cK$.
Then,  the expected complexity of the Boltzmann sampler for $\ul{\cK}$---given
in Section~\ref{sec:boltz_binary_trees}---satisfies, 
$$
\Lambda \ul{\cK}(z,w)=O\Big((z_0-z)^{-1/2}\Big)\ \mathrm{as}\ (z,w)\to (z_0,w_0).
$$
\end{lemma}
\begin{proof}
The Boltzmann sampler $\Gamma \ul{\cK}(z,w)$ is just obtained by translating a completely
recursive decomposition grammar. Hence, the generation process consists in 
building the tree node by node following certain branching rules. 
Accordingly, the cost of generation is just equal to the number of nodes of the 
tree that is finally returned, assuming unit cost for building a node\footnote{ 
We could also use the computation rules for the expected complexities, but here there is  the simpler argument that the expected complexity is equal to the expected size, as there is no rejection yet.}.
As an unrooted binary 
tree has two more leaves than nodes, we have
$$
\Lambda \ul{\cK}(z,w)\leq ||\ul{\cK}||_{(z,w)}\leq ||\ul{\cK}||_{(z,w_0)},
$$
where the second inequality results from the monotonicity property of expected sizes (Lemma~\ref{lem:monotonicity}).

Notice that, for
$\tau\in\ul{\cK}$, the 
number of nodes is not greater than $(3|\tau|+1)$, where $|\tau|$ is as usual the number
of black nodes. Hence the number of nodes is at most $4|\tau|$.
As a consequence,
$$
\Lambda \ul{\cK}(z,w)\leq 4\cdot |\ul{\cK}|_{(z,w_0)}.
$$
According to Lemma~\ref{lem:sing_bin}, the class $\ul{\cK}$ is $1/2$-singular. Hence, by Lemma~\ref{lem:exp_size},  $|\ul{\cK}|_{(z,w_0)}$  is $O((z_0-z)^{-1/2})$ as $z\to z_0$. So $\Lambda \ul{\cK}(z,w)$
is also $O((z_0-z)^{-1/2})$.
\end{proof}

\begin{lemma}[derived bicolored binary trees]\label{lem:comp_der_bin}
Let $(z_0,w_0)$ be a singular point of $\cK$.
Then, the expected complexity of the Boltzmann sampler for $\cK'$---given
in Section~\ref{sec:sampKp}---satisfies
$$
\Lambda \cK'(z,w)=O\left((z_0-z)^{-1/2}\right)\ \mathrm{as}\ (z,w)\to(z_0,w_0).
$$
\end{lemma}
\begin{proof}
The sampler $\Gamma \cK'(z,w)$ has been obtained from $\Gamma \ul{\cK}(z,w)$
by applying  the procedure \UtoL  to the class $\cK$. It
is easily checked that the ratio number of black nodes/number of leaves in 
a bicolored binary tree
is bounded from above and from below (we have already used the ``below''
bound in Lemma~\ref{lem:comp_uK}). Precisely, $3|\tau|+3\geq ||\tau||$ and 
 $|\tau|\leq 2||\tau||/3$, from which it is easily checked that $\alphaLU=2/3$ and 
$\alphaUL=6$ (attained by the tree with 1 black and 3 white nodes).
Hence,  according to Corollary~\ref{lem:change_root}, 
$\Lambda \cK'(z,w)\leq 4\ \!\Lambda \ul{\cK}(z,w)$,
so $\Lambda \cK'(z,w)$ is $O\left((z_0-z)^{-1/2}\right)$. 
\end{proof}

\begin{lemma}[bicolored binary trees]\label{lem:comp_bin}
Let $(z_0,w_0)$ be a singular point of $\cK$.
Then,  the expected complexity of the Boltzmann sampler for $\cK$---given
in Section~\ref{sec:Ksamp}---satisfies
$$
\Lambda \cK(z,w)=O\left(1\right)\ \mathrm{as}\ (z,w)\to(z_0,w_0).
$$
\end{lemma}
\begin{proof}
At each attempt in the generator $\Gamma \cK(z,w)$, the first step is to
call $\Gamma\ul{\cK}(z,w)$ to generate a certain tree $\tau\in\ul{\cK}$
 (it is here convenient to assume that the object is ``chosen'' before the generation starts), with probability
$$
\frac{1}{\ul{K}(z,w)}\frac{z^{|\tau|}}{|\tau|!}w^{||\tau||};
$$
and the probability that the generation succeeds to finish is $2/(||\tau||+1)$.
Hence, the total probability of success at each attempt in $\Gamma \cK(z,w)$ 
satisfies
$$
\pacc=\sum_{\tau\in\ul{\cK}}\frac{1}{\ul{K}(z,w)}\frac{z^{|\tau|}}{|\tau|!}w^{||\tau||}\cdot\frac{2}{||\tau||+1}.
$$
As each object $\tau\in\cK$ gives rise to $||\tau||$ objects in $\ul{\cK}$ that all have
L-size $|\tau|$ and U-size 
$||\tau||-1$, we also have
$$
\pacc=\sum_{\tau\in\cK}\frac{2}{\ul{K}(z,w)}\frac{z^{|\tau|}}{|\tau|!}w^{||\tau||-1}=\frac{2K(z,w)}{w\ul{K}(z,w)}.
$$
As $\cK$ is $3/2$-singular and $\ul{\cK}$ is $1/2$-singular, $\pacc$ converges to the positive constant $c_0:=2K(z_0,w_0)/(w_0\ul{K}(z_0,w_0))$ 
as $(z,w)\to(z_0,w_0)$.

Now call $\mathfrak{A}(z,w)$ the random generator for $\cK$ delimited inside the repeat/until
loop of $\Gamma \cK(z,w)$, and let $\Lambda \frak{A}(z,w)$ be the expected complexity
of $\mathfrak{A}(z,w)$. According to Lemma~\ref{lem:target}, $\Lambda \cK(z,w)=\Lambda\frak{A}(z,w)/\pacc$. In addition, when $(z,w)\to (z_0,w_0)$, $\pacc$ converges to a positive constant, hence it 
remains to prove that $\Lambda \frak{A}(z,w)=O(1)$ 
in order to prove the lemma.

Let $\tau\in\ul{\cK}$, and let $m:=||\tau||$. During a call to $\frak{A}(z,w)$, and knowing (again, in advance) that  $\tau$ is under generation,
the probability  that  at least $k\geq 1$ nodes of $\tau$ are built  is $2/(k+1)$,
due to the Bernoulli
probabilities telescoping each other. Hence, for $k<m-1$, the probability $p_k$
that the generation aborts when exactly $k$ nodes are generated 
 satisfies $p_k=\frac{2}{k+1}-\frac{2}{k+2}=\frac{2}{(k+1)(k+2)}$. In addition, the probability that the whole tree is generated is $2/m$ (with a final rejection or not), in which case $(m-1)$ nodes are built. 
Measuring the complexity as the number of nodes that are built, 
we obtain the following expression for the expected complexity of $\frak{A}(z,w)$ knowing that $\tau$ is chosen:
$$
\Lambda\frak{A}^{(\tau)}(z,w)=\sum_{k=1}^{m-2}k\cdot p_k+(m-1)\frac{2}{m}\leq 2\ \!H_m,
$$
where $H_m:=\sum_{k=1}^m1/k$ is the $m$th Harmonic number. 
Define $a_m(z):=[w^m]\ul{K}(z,w)$. We have
$$
\Lambda\frak{A}(z,w)\leq \frac{2}{\ul{K}(z,w)}\sum_m H_ma_m(z)w^m\leq \frac{2}{\ul{K}(z,w)}\sum_m H_ma_m(z_0)w_0^m.
$$
Hence, writing $c_0:=3/\ul{K}(z_0,w_0)$, we have $\Lambda\frak{A}(z,w)\leq c_0\sum_m H_ma_m(z_0)w_0^m$ for $(z,w)$ close to $(z_0,w_0)$.
Using the Drmota-Lalley-Woods theorem (similarly as in Lemma~\ref{lem:sing_bin}),
it is easily shown that the function $w\mapsto \ul{K}(z_0,w)$ has a square-root singularity at $w=w_0$. Hence, the transfer theorems
of singularity analysis~\cite{fla,flaod} yield the asymptotic estimate $a_m(z_0)\sim c\ \! m^{-3/2}w_0^{-m}$ for some constant $c>0$, so that $a_m(z_0)\leq c'\ \! m^{-3/2}w_0^{-m}$ for some
constant $c'>0$. Hence $\Lambda\frak{A}(z,w)$ is bounded by the converging
series $c_0\ \!c'\sum_m H_m\ \!m^{-3/2}$ for $(z,w)$ close to $(z_0,w_0)$, which concludes the proof.
\end{proof}

\subsubsection{Complexity of the Boltzmann samplers for irreducible dissections}

\begin{lemma}[irreducible dissections]\label{lem:comp_irr}
Let $(z_0,w_0)$ be a singular point of $\cI$. Then, the expected complexities
of the Boltzmann samplers for $\cI$ and $\cI'$---described respectively in Section~\ref{sec:sampI} and~\ref{sec:sampIp}---satisfy, as $(z,w)\to (z_0,w_0)$:
\begin{eqnarray*}
\Lambda \cI(z,w)&=&O\ (1),\\
\Lambda \cI'(z,w)&=&O\left((z_0-z)^{-1/2}\right).
\end{eqnarray*}
\end{lemma}
\begin{proof}
As stated in Proposition~\ref{prop:bijbin3conn} and proved in~\cite{FuPoSc05}, the closure-mapping has linear
time complexity, i.e., there exists a constant $\lambda$ such that the cost of
closing any binary tree $\kappa$ is at most $\lambda\cdot||\kappa||$.
Recall that $\Gamma\cI(z,w)$ calls $\Gamma \cK(z,w)$ and closes the binary tree
generated. Hence
$$
\Lambda \cI(z,w)\leq \Lambda \cK(z,w)+\lambda\cdot ||\cK||_{(z,w)}\leq \Lambda \cK(z,w)+\lambda\cdot ||\cK||_{(z,w_0)},
$$
where the second inequality results from the monotonicity property of expected sizes (Lemma~\ref{lem:monotonicity}).
Again we use the fact that, for $\tau\in\cK$, $||\tau||\leq 3|\tau|+1$,
so $||\tau||\leq 4|\tau|$. Hence 
$$
\Lambda \cI(z,w)\leq \Lambda \cK(z,w)+4\lambda\cdot |\cK|_{(z,w_0)}.
$$
As the class $\cK$ is $3/2$-singular, the expected size $|\cK|_{(z,w_0))}$ is $O(1)$
when $z\to z_0$. In addition, according to Lemma~\ref{lem:comp_bin}, 
$\Lambda \cK(z,w)$ is $O(1)$ when $(z,w)\to (z_0,w_0)$.
Hence $\Lambda\cI(z,w)$ is $O(1)$. 

Similarly, for $\cI'$, we have 
$$
\Lambda \cI'(z,w)\leq \Lambda \cK'(z,w)+\lambda\cdot ||\cK'||_{(z,w))}\leq \Lambda \cK'(z,w)+4\lambda\cdot |\cZL\star\cK'|_{(z,w_0)}.
$$
As the class $\cK'$ is $1/2$-singular (and so is $\cZL\star\cK'$), 
the expected size  $|\cZL\star\cK'|_{(z,w_0)}$ is $O((z_0-z)^{-1/2})$ when $z\to z_0$.
In addition we have proved in Lemma~\ref{lem:comp_der_bin} that $\Lambda \cK'(z,w)$
is $O((z_0-z)^{-1/2})$. 
Therefore $\Lambda\cI'(z,w)$ is $O((z_0-z)^{-1/2})$.
\end{proof}

\begin{lemma}[rooted irreducible dissections]\label{lem:comp_root_irr}
Let $(z_0,w_0)$ be a singular point of $\cI$. Then,  the expected complexities
of the Boltzmann samplers for $\cJ$ and $\cJ'$---described respectively in Section~\ref{sec:sampI} and~\ref{sec:sampIp}---satisfy, as $(z,w)\to (z_0,w_0)$:
\begin{eqnarray*}
\Lambda \cJ(z,w)&=&O\ (1),\\
\Lambda \cJ'(z,w)&=&O\left((z_0-z)^{-1/2}\right).
\end{eqnarray*}
\end{lemma}
\begin{proof}
The sampler $\Gamma \cJ(z,w)$ is directly obtained from $\Gamma \cI(z,w)$, according to the identity $\cJ=3\star\cZL\star\cZU\star\cI$, so $
\Lambda\cJ(z,w)=\Lambda\cI(z,w)$, which is $O(1)$ as $(z,w)\to(z_0,w_0)$.

The sampler $\Gamma\cJ'(z,w)$ is obtained from $\Gamma \cI(z,w)$
and $\Gamma \cI'(z,w)$, according to the identity $\cJ'=3\star\cZL\star\cZU\star\cI'+3\star\cZU\star\cI$. Hence, 
$\Lambda\cJ'(z,w)\leq 1+\Lambda\cI(z,w)+\Lambda\cI'(z,w)$. 
According to Lemma~\ref{lem:comp_irr}, 
$\Lambda\cI(z,w)$ and $\Lambda\cI'(z,w)$ are respectively $O(1)$ and 
$O((z_0-z)^{-1/2})$ when $(z,w)\to (z_0,w_0)$. Hence $\Lambda\cJ'(z,w)$ is $O((z_0-z)^{-1/2})$.
\end{proof}

\begin{lemma}[admissible rooted irreducible dissections]\label{comp:Ia}
Let $(z_0,w_0)$ be a singular point of $\cI$. Then,  the expected complexities
of the Boltzmann samplers for $\cJa$ and $\cJa'$---described respectively in Section~\ref{sec:sampI} and~\ref{sec:sampIp}---satisfy, as $(z,w)\to (z_0,w_0)$:
\begin{eqnarray*}
\Lambda \cJa(z,w)&=&O\ (1),\\
\Lambda \cJa'(z,w)&=&O\left((z_0-z)^{-1/2}\right).
\end{eqnarray*}
\end{lemma}
\begin{proof}
Call $\ol{\Gamma}{\cJ}(z,w)$ the sampler that calls 
$\Gamma\cJ(z,w)$ and checks
if the dissection is admissible. 
By definition, $\Gamma\cJa(z,w)$ repeats calling $\ol{\Gamma}{\cJ}(z,w)$
until the dissection generated is  in $\cJa$. Hence the probability of acceptance $\pacc$
at each attempt is equal to $\Ja(z,w)/J(z,w)$, i.e., is equal to $\vec{M_3}(z,w)/J(z,w)$ (the isomorphism $\cJa\simeq\vec{\cM_3}$ yields $\Ja(z,w)=\vec{M_3}(z,w)$). Call
$\ol{\Lambda}\cJ(z,w)$ the expected complexity of $\ol{\Gamma}\cJ(z,w)$. 
By Lemma~\ref{lem:comp_root_irr},
$$
\Lambda \cJa(z,w)=\frac{1}{\pacc}\ol{\Lambda} \cJ(z,w)=\frac{J(z,w)}{\vec{M_3}(z,w)}\ol{\Lambda} \cJ(z,w).
$$
We recall from Section~\ref{sec:sing_beh} that the singular points are the same for rooted 3-connected planar graphs/maps, for bicolored binary trees, and for 
irreducible dissections. Hence $(z_0,w_0)$ is a singular point for $\vec{M_3}(z,w)$.
The classes $\cJ$ and $\vec{\cM_3}\simeq 2\star\vec{\cG_3}$ are $3/2$-singular
by Lemma~\ref{lem:sing_diss} and Lemma~\ref{lem:sing_3_conn}, respectively.
Hence, when $(z,w)\to(z_0,w_0)$, the series $J(z,w)$ and $\vec{M_3}(z,w)$
are $\Theta(1)$, even more they converge to positive constants (because
 these functions are rational in terms of 
bivariate series for binary trees). Hence $\pacc$ also 
converges to a positive constant, so it remains to prove that $\ol{\Lambda} \cJ(z,w)$
is $O(1)$. 
Testing admissibility (i.e., the existence of an internal path of length 3 
connecting the root-vertex to the opposite outer vertex) has clearly linear time complexity.
Hence, for some constant $\lambda$, 
$$\ol{\Lambda} \cJ(z,w)\leq\Lambda\cJ(z,w)+\lambda\cdot||\cJ||_{(z,w)}\leq \Lambda\cJ(z,w)+\lambda\cdot||\cJ||_{(z,w_0)},$$
where the second inequality results from the monotonicity of the expected sizes (Lemma~\ref{lem:monotonicity}).
Both $\Lambda\cJ(z,w)$ and $||\cJ||_{(z,w_0)}$ are $O(1)$ when $z\to z_0$ 
(by Lemma~\ref{lem:comp_root_irr} and because $\cJ$ is $3/2$-singular, respectively). Hence $\ol{\Lambda}\cJ(z,w)$ is also $O(1)$, so $\Lambda \cJa(z,w)$ is also $O(1)$.

The proof  for $\cJa'$ is similar. First, we have
$$
\Lambda \cJa'(z,w)=\frac{J'(z,w)}{\vec{M_3}'(z,w)}\cdot\ol{\Lambda} \cJ'(z,w),
$$
where $\ol{\Lambda} \cJ'(z,w)$ is the expected cost of a call to $\Gamma\cJ'(z,w)$
followed by an admissibility test.
Both series $J'(z,w)$ and $\vec{M_3}'(z,w)$ are $1/2$-singular, even more,
they  converge to positive constants as $(z,w)\to(z_0,w_0)$ (again, because these functions
are rational in terms of bivariate series of binary trees).  
Hence, when $(z,w)\to(z_0,w_0)$, the quantity $J'(z,w)/\vec{M_3}'(z,w)$ converges to a positive constant.
Moreover, according to the linear complexity of admissibility testing, we have
$\ol{\Lambda} \cJ'(z,w)\leq\Lambda \cJ'(z,w)+\lambda\cdot||\cJ'||_{(z,w_0)}$.
Both quantities $\Lambda \cJ'(z,w)$ and $||\cJ'||_{(z,w_0)}$ are $O((z_0-z)^{-1/2})$. Hence $\Lambda \cJa'(z,w)$
is also $O((z_0-z)^{-1/2})$.
\end{proof}

\subsubsection{Complexity of the Boltzmann samplers for 3-connected maps}

\begin{lemma}[rooted 3-connected maps]\label{lem:comp_M3}
Let $(z_0,w_0)$ be a singular point of $\cMt$. Then the expected complexities
of the Boltzmann samplers for $\vec{\cM_3}$ and $\vec{\cM_3}'$ satisfy respectively, as $(z,w)\to (z_0,w_0)$:
\begin{eqnarray*}
\Lambda \vec{\cM_3}(z,w)&=&O\ (1),\\
\Lambda \vec{\cM_3}'(z,w)&=&O\left((z_0-z)^{-1/2}\right).
\end{eqnarray*}
\end{lemma}
\begin{proof}
Recall that $\Gamma\vec{\cM_3}(z,w)$ ($\Gamma\vec{\cM_3}'(z,w)$, resp.) calls $\Gamma\cJa(z,w)$ ($\Gamma\cJa'(z,w)$, resp.) and
returns the primal map of the dissection. 
The primal-map construction is in fact just a reinterpretation of the combinatorial
encoding of rooted maps (in particular when dealing with the half-edge data structure). Hence $\Lambda \vec{\cM_3}(z,w)=\Lambda \cJa(z,w)$ and $\Lambda \vec{\cM_3}'(z,w)=\Lambda \cJa'(z,w)$.
This concludes the proof, according to the estimates for $\Lambda \cJa(z,w)$ and $\Lambda \cJa'(z,w)$ given in Lemma~\ref{comp:Ia}. (A proof following
the same lines as in Lemma~\ref{lem:comp_irr} would also be possible.)
\end{proof}

\subsubsection{Complexity of the Boltzmann samplers for 3-connected planar graphs}

\begin{lemma}[rooted 3-connected planar graphs]\label{lem:comp_samp_Gt}
Let $(z_0,w_0)$ be a singular point of $\cGt$. Then the expected complexities
of the Boltzmann samplers for $\vec{\cG_3}$, $\vec{\cG_3}'$ and 
$\ul{\vec{\cG_3}}$ satisfy respectively, as $(z,w)\to (z_0,w_0)$:
\begin{eqnarray*}
\Lambda \vec{\cG_3}(z,w)&=&O\ (1),\\
\Lambda \vec{\cG_3}'(z,w)&=&O\left((z_0-z)^{-1/2}\right),\\
\Lambda \ul{\vec{\cG_3}}(z,w)&=&O\left((z_0-z)^{-1/2}\right).
\end{eqnarray*}
\end{lemma}
\begin{proof}
The sampler $\Gamma \vec{\cG_3}(z,w)$ ($\Gamma \vec{\cG_3}'(z,w)$, resp.) is directly obtained from $\Gamma \vec{\cM_3}(z,w)$ ($\Gamma \vec{\cM_3}'(z,w)$, resp.) by forgetting the 
embedding. Hence $\Lambda \vec{\cG_3}(z,w)=\Lambda \vec{\cM_3}(z,w)$ and 
$\Lambda \vec{\cG_3}'(z,w)=\Lambda \vec{\cM_3}'(z,w)$, which are---by 
Lemma~\ref{lem:comp_M3}---respectively
$O(1)$ and $O((z_0-z)^{-1/2})$ as $(z,w)\to (z_0,w_0)$.

Finally, the sampler $\Gamma \ul{\vec{\cG_3}}(z,w)$ is obtained from $\Gamma \vec{\cG_3}'(z,w)$ by applying the procedure \LtoU to the class $\vec{\cG_3}$.
By the Euler relation, $\alphaUL=3$ (given asymptotically by triangulations)
and $\alphaLU=2/3$ (given asymptotically by cubic graphs).
Thus, by Corollary~\ref{lem:change_root}, $\Lambda\ul{\vec{\cG_3}}(z,w)\leq 2\cdot\Lambda\vec{\cG_3}'(z,w)$, which ensures that $\Lambda\ul{\vec{\cG_3}}(z,w)$ is 
$O((z_0-z)^{-1/2})$.
\end{proof}

\subsubsection{Complexity of the Boltzmann samplers for networks}
At first 
we need to introduce the following notations. Let $\cC$ be a class endowed with a 
Boltzmann sampler $\Gamma\cC(x,y)$ and let $\gamma\in\cC$. Then $\Lambda\cC^{(\gamma)}(x,y)$ 
denotes the expected complexity of $\Gamma\cC(x,y)$ conditioned on the fact that the 
object generated is $\gamma$. If $\Gamma\cC(x,y)$ uses rejection, i.e., repeats building objects and rejecting them until finally an object is accepted, then $\Lambda
\cC^{\mathrm{rej}}(x,y)$ denotes the expected complexity of $\Gamma\cC(x,y)$
without counting the last (successful) attempt.

\begin{lemma}[networks]\label{lem:comp_D}
Let $(z_0,y_0)$ be a singular point of $\cD$.
Then,  the expected complexity of the Boltzmann sampler
for $\cD$---described in Section~\ref{sec:2conn3conn}---satisfies 
$$
\Lambda \cD(z,y_0)=O\left(1\right)\ \mathrm{as}\ z\to z_0.
$$
\end{lemma}
\begin{proof}
Trakhtenbrot's decomposition ensures that a network
$\gamma\in\cD$ is a collection
of 3-connected components $\kappa_1,\ldots,\kappa_r$
(in $\cGtr$) that are assembled
together in a series-parallel backbone $\beta$ (due to
the auxiliary classes $\cS$ and $\cP$).
Moreover, if $\gamma$ is produced by the Boltzmann
sampler $\Gamma \cD(z,y_0)$,
then each of the 3-connected components $\kappa_i$
results from a call to $\Gamma \vec{\cG_3}(z,w)$,
where $w:=D(z,y_0)$.

An important point, which is proved in~\cite{BeGa}, 
is that the composition scheme to go from
rooted 3-connected planar graphs to networks is critical.
This means that  $w_0:=D(z,y_0)$ (change
of variable from 3-connected planar graphs to networks) 
is such that $(z_0,w_0)$ is a singular point of $\cGtr$.

As the series-parallel backbone is built
edge by edge, the cost of generating $\beta$
 is simply $||\beta||$ (the number of edges of $\beta$); and the
expected cost of generating $\kappa_i$,
 for $i\in [1..r]$, is $\Lambda
\vec{\cG_3}^{(\kappa_i)}(z,w)$. Hence
 \begin{equation}
 \Lambda
\cD^{(\gamma)}(z,y_0)=||\beta||+\sum_{i=1}^r\Lambda\vec{\cG_3}^{(\kappa_i)}(z,w).
 \end{equation}
 
\begin{claim}
 There exists a constant $c$ such that, for every
$\kappa\in\vec{\cG_3}$, 
 $$
 \Lambda\vec{\cG_3}^{(\kappa)}(z,w)\leq c||\kappa||\ \ \
\mathrm{as}\ \ (z,w)\to(z_0,w_0).
 $$
 \end{claim}
\smallskip

\ni{\it Proof of the claim.}
 The Boltzmann sampler $\Gamma \vec{\cG_3}(z,w)$ is
obtained by repeated attempts to 
 build binary trees until the tree is successfully
generated (no early interruption) and gives
 rise to a 3-connected planar graph (admissibility
condition). For $\kappa\in\cK$, call $c^{(\kappa)}$ the cost of building
$\kappa$ (i.e., generate the underlying 
 binary tree and perform the closure). Then
 $$
 \Lambda \vec{\cG_3}^{(\kappa)}(z,w)=\Lambda
\vec{\cG_3}^{\mathrm{rej}}(z,w)+c^{(\kappa)}.
 $$
  
 Notice that $\Lambda
\vec{\cG_3}^{\mathrm{rej}}(z,w)\leq \Lambda
\vec{\cG_3}(z,w)$,
 which is $O(1)$ as $(z,w)\to(z_0,w_0)$. Moreover, the
closure-mapping has linear time
 complexity. Hence there exists a constant $c$
independent from $\kappa$  and from $z$ such that
    $\Lambda \vec{\cG_3}^{(\kappa)}(z,w)\leq c\
\!||\kappa||$ as $z\to z_0$.
\qedclaim
 
 The claim ensures that, upon taking $c>1$, every
$\gamma\in\cD$ satisfies 
  $$
 \Lambda \cD^{(\gamma)}(z,y_0)\leq
c(||\beta||+\sum_{i=1}^r||\kappa_i||)\ \ \
\mathrm{as}\ \ z\to z_0.
 $$ 
Since each edge of $\gamma$ is represented at most
once in $\beta\cup\kappa_1\cup\ldots\cup\kappa_r$, we
also have $\Lambda D^{(\gamma)}(z,y_0)\leq
c||\gamma||$.
Hence, when $z\to z_0$, 
$\Lambda \cD^{(\gamma)}(z,y_0)\leq 3c\cdot(|\gamma|+1)$ (by the Euler relation), which yields
$$
\Lambda \cD(z,y_0)\leq 3c\cdot |\cZL\star\cD|_{(z,y_0)}.
$$
As the class $\cD$ is $3/2$-singular (clearly, so is $\cZL\star\cD$),
the expected size $|\cZL\star\cD|_{(z,y_0)}$ is $O(1)$ when $z\to z_0$. 
Hence $\Lambda \cD(z,y_0)$ is $O(1)$. 
\end{proof}

\begin{lemma}[derived networks]
Let $(z_0,y_0)$ be a singular point of $\cD$.
Then,  the expected complexity of the Boltzmann sampler
for $\cD'$---described in Section~\ref{sec:sampDp}---satisfies
$$
\Lambda \cD'(z,y_0)=O\left((z_0-z)^{-1/2}\right)\ \mathrm{as}\ z\to z_0.
$$
\end{lemma}
\begin{proof}
Let us fix $z\in(0,z_0)$. 
Define  $X:=(\Lambda \cD'(z,y_0),\Lambda \cS'(z,y_0),\Lambda \cP'(z,y_0),\Lambda \cH'(z,y_0))$. Our strategy here is to use the computation rules (Figure~\ref{fig:comp_rules})
to obtain a recursive equation specifying the vector $X$. 
By Remark~\ref{rk:finite}, we have to check that the components of $X$ are finite.
\begin{claim}\label{claim:Dp_finite}
For $z\in(0,z_0)$, the quantities $\Lambda \cD'(z,y_0)$, $\Lambda \cS'(z,y_0)$, $\Lambda \cP'(z,y_0)$, and $\Lambda \cH'(z,y_0)$ are finite.
\end{claim}

\noindent\emph{Proof of the claim.} Consider $\Lambda \cD'(z,y_0)$ (the verification
is similar for $\Lambda \cS'(z,y_0)$, $\Lambda \cP'(z,y_0)$, and $\Lambda \cH'(z,y_0)$).
Let $\gamma\in\cD'$, with $\beta$ the series-parallel backbone and $\kappa_1,\ldots,\kappa_r$ the 3-connected components of $\gamma$. Notice that each $\kappa_i$
is drawn either by $\Gamma\vec{\cG_3}(z,w)$ or $\Gamma\ul{\vec{\cG_3}}(z,w)$
or $\Gamma\vec{\cG_3}'(z,w)$, where $w=D(z,y_0)$. Hence the expected
cost of generating $\kappa_i$ is bounded by $M+c||\kappa_i||$, where $M:=\mathrm{Max}(\Lambda\vec{\cG_3}(z,w),\Lambda\ul{\vec{\cG_3}}(z,w),\Lambda\vec{\cG_3}'(z,w))$ and $c||\kappa_i||$ represents the cost of building $\kappa_i$ 
using the closure-mapping. As a consequence,
$$
\Lambda\cD'^{(\gamma)}(z,y_0)\leq ||\beta||+\sum_{i=1}^r M+c||\kappa_i||\leq C||\gamma||,\ \mathrm{with}\ C:=M+c+1.
$$ 
Hence
$$
\Lambda\cD'(z,y_0)\leq \frac{C}{D'(z,y_0)}\sum_{\gamma\in\cD'}||\gamma||\frac{z^{|\gamma|}}{|\gamma|!}y_0^{||\gamma||},
$$
which is $O(1)$ since it converges to the constant $Cy_0\partial_yD'(z,y_0)/D'(z,y_0)$.
\qedclaim

Using the computation rules given in Figure~(\ref{fig:comp_rules}), the decomposition grammar~(N') of derived networks---as given in Section~\ref{sec:sampDp}---is translated to a linear system
$$
X=AX+L,
$$ 
where $A$ is a $4\times 4$-matrix and $L$ is a 4-vector. Precisely, the components of $A$ are
rational or exponential expressions in terms of series of networks and their derivatives: all these quantities converge
as $z\to z_0$ because all the classes of networks are $3/2$-singular. Hence $A$ converges
to a matrix $A_0$ as $z\to z_0$. 
In addition, $A$ is a substochastic matrix, i.e., a matrix
with nonnegative coefficients and with sum at most 1 in each row. Indeed, the entries in each of 
the 4 rows of $A$
correspond to probabilities of a Bernoulli switch when calling $\Gamma D'(z,y)$, $\Gamma S'(z,y)$,
$\Gamma P'(z,y)$, and $\Gamma H'(z,y)$, respectively. Hence, the limit matrix  $A_0$ is also
substochastic. It is easily checked that $A_0$ is indeed strictly
substochastic, i.e., at least one row has sum $<1$ (here, the first and third row add up to 1, whereas the second and fourth row add up to $<1$). In addition, $A_0$ is irreducible, 
i.e., the dependency graph induced
by the nonzero coefficients of $A_0$ is strongly connected.  A well known result of Markov chain theory
ensures that  $(I-A_0)$ is invertible~\cite{Ke}. Hence, $(I-A)$ is invertible for $z$ close to $z_0$, 
and $(I-A)^{-1}$ converges to
the matrix $(I-A_0)^{-1}$.  Moreover, the components of $L$ are of the form 
$$L=\Big(a,b,c,d\cdot\Lambda \vec{\cGt}'(z,w)+e\cdot\Lambda \ul{\vec{\cG_3}}(z,w)\Big),$$ where $w=D(z,y_0)$ and  $\{a,b,c,d,e\}$ are expressions involving the series of networks, their derivatives, and the quantities $\{\Lambda D,\Lambda S, \Lambda P,\Lambda H\}$,
which have already been shown to be bounded as $z\to z_0$. As a consequence, $a,b,c,d,e$ are $O(1)$ as $z\to z_0$. 
Moreover, it has been shown in~\cite{BeGa} that 
the value $w_0:=D(z_0,y_0)$ is such that $(z_0,w_0)$ is singular for $\cG_3$,
and $w_0-w\sim \lambda\cdot(z_0-z)$, with $\lambda:=D'(z_0,y_0)$.
By Lemma~\ref{lem:comp_samp_Gt}, $\Lambda\vec{\cG_3}'(z,w)$ and $\Lambda\ul{\vec{\cG_3}}(z,w)$
are $O((z_0-z)^{-1/2})$ as $z\to z_0$; hence these quantities are also 
$O((z_0-z)^{-1/2})$.
We conclude that the components of $L$ are $O((z_0-z)^{-1/2})$,
as well as the components of $X=(I-A)^{-1}L$. In particular, $\Lambda\cD'(z,y_0)$
(the first component of $X$) is $O((z_0-z)^{-1/2})$.
\end{proof}

\vspace{2cm}

\subsubsection{Complexity of the Boltzmann samplers for 2-connected planar graphs}

\begin{lemma}[rooted 2-connected planar graphs]\label{lem:comp_vecG2}
Let $(z_0,y_0)$ be a singular point of $\cG_2$. Then the expected complexities
of the Boltzmann samplers for $\vec{\cG_2}$ and $\vec{\cG_2}'$ satisfy respectively, as $z\to z_0$:
\begin{eqnarray*}
\Lambda \vec{\cG_2}(z,y_0)&=&O\ (1),\\
\Lambda \vec{\cG_2}'(z,y_0)&=&O\left((z_0-z)^{-1/2}\right).
\end{eqnarray*}
\end{lemma}
\begin{proof}
Recall that the Boltzmann sampler $\Gamma \vec{\cG_2}(z,y_0)$ is directly
obtained from $\Gamma \cD(z,y_0)$, more precisely from $\Gamma (1+\cD)(z,y_0)$.
According to Lemma~\ref{lem:comp_D},  $\Lambda \cD(z,y_0)$ is $O(1)$ as $z\to z_0$, hence 
$\Lambda \vec{\cG_2}(z,y_0)$ is also $O(1)$.

Similarly $\Gamma\vec{\cG_2}'(z,y_0)$ is directly obtained from $\Gamma \cD'(z,y_0)$,
hence $\Lambda\vec{\cG_2}'(z,y_0)=\Lambda\cD'(z,y_0)$, which is $O((z_0-z)^{-1/2})$
as $z\to z_0$.
\end{proof}

\begin{lemma}[U-derived 2-connected planar graphs]\label{lem:comp_UG2}
Let $(z_0,y_0)$ be a singular point of $\cG_2$. Then,  the expected complexities
of the Boltzmann samplers for $\ul{\cG_2}$ and $\ul{\cG_2}'$---described in Section~\ref{sec:sampDp}---satisfy, as $z\to z_0$:
\begin{eqnarray*}
\Lambda \ul{\cG_2}(z,y_0)&=&O\ (1),\\
\Lambda \ul{\cG_2}'(z,y_0)&=&O\left((z_0-z)^{-1/2}\right).
\end{eqnarray*}
\end{lemma}
\begin{proof}
The Boltzmann sampler for $\ul{\cG_2}$ is directly obtained from the one for $\vec{\cG_2}$, according to the identity $2\star\ul{\cG_2}=\cZ_L\ \!\!\!^2\star\vec{\cG_2}$. 
Hence $\Lambda\ul{\cG_2}(z,y_0)=\Lambda\vec{\cG_2}(z,y_0)$,
which is $O(1)$ as $z\to z_0$, according to Lemma~\ref{lem:comp_vecG2}.
Similarly, the Boltzmann sampler for $\ul{\cG_2}'$ is directly obtained from the ones
for the classes $\vec{\cG_2}$ and $\vec{\cG_2}'$, according to the identity $2\star\ul{\cG_2}'=\cZ_L\ \!\!\!^2\star\vec{\cG_2}'+2\star\cZ_L\star\vec{\cG_2}$.
Hence $\Lambda\ul{\cG_2}(z,y_0)\leq 1+\Lambda\vec{\cG_2}'(z,y_0)+\Lambda\vec{\cG_2}(z,y_0)$.
 When $z\to z_0$, $\Lambda\vec{\cG_2}(z,y_0)$ is $O(1)$ 
and $\Lambda\vec{\cG_2}'(z,y_0)$
is $O((z_0-z)^{-1/2})$ according to Lemma~\ref{lem:comp_vecG2}. Hence, 
$\Lambda\ul{\cG_2}'(z,y_0)$ is $O((z_0-z)^{-1/2})$.
\end{proof}

\begin{lemma}[bi-derived 2-connected planar graphs]\label{lem:comp_LG2}
Let $(z_0,y_0)$ be a singular point of $\cG_2$. Then, the expected complexities
of the Boltzmann samplers for $\cGbp$ and $\cGbp'$---described in Section~\ref{sec:sampDp}---satisfy, as $z\to z_0$:
\begin{eqnarray*}
\Lambda \cGbp(z,y_0)&=&O\ (1),\\
\Lambda \cGbp'(z,y_0)&=&O\left((z_0-z)^{-1/2}\right).
\end{eqnarray*}
\end{lemma}
\begin{proof}
Recall that the Boltzmann sampler $\Gamma\cGbp(z,y_0)$ is obtained from $\Gamma\ul{\cG_2}(z,y_0)$
by applying the procedure \UtoL to the class $\cG_2$. In addition, according to 
the Euler relation, any simple connected planar graph $\gamma$ (with $|\gamma|$ the number of vertices and $||\gamma||$ the number of edges) satisfies $|\gamma|\leq
||\gamma||+1$ (trees) and $||\gamma||\leq 3|\gamma|-6$ (triangulations).
It is then easily checked that, for the class $\cG_2$, 
$\alphaUL=3$ (attained asymptotically by triangulations) and $\alphaLU=2$ (attained by the link-graph, which  has 2 vertices and 1 edge). Hence, by Corollary~\ref{lem:change_root},
 $\Lambda\cGbp(z,y_0)\leq 6\ \!\Lambda\ul{\cG_2}(z,y_0)$. 
Thus, by Lemma~\ref{lem:comp_UG2}, $\Lambda\cGbp(z,y_0)$ is $O(1)$ as $z\to z_0$.

The proof for $\Lambda \cGbp'(z,y_0)$ is similar, except that the procedure
\UtoL is now applied to the derived class $\cGbp$, meaning that the L-size is now
the number of vertices minus 1. We still have $\alphaUL=3$ (attained asymptotically by triangulations), and now $\alphaLU=1$
(attained by the link-graph). Corollary~\ref{lem:change_root} yields $\Lambda\cGbp'(z,y_0)\leq 3\ \!\Lambda\ul{\cG_2}'(z,y_0)$. 
Hence, from Lemma~\ref{lem:comp_UG2}, $\Lambda\cGbp'(z,y_0)$ is $O((z_0-z)^{-1/2})$ as $z\to z_0$.
\end{proof}

\subsubsection{Complexity of the Boltzmann samplers for connected planar graphs}

\begin{lemma}[derived connected planar graphs]\label{lem:comp_G1p}
Let $(x_0,y_0)$ be a singular point of $\cG_1$. Then,  the expected complexity
of the Boltzmann sampler for $\cGcp$---described in Section~\ref{sec:conn2conn}---satisfies
\begin{eqnarray*}
\Lambda \cGcp(x,y_0)&=&O\ \!(1)\ \ \ \mathrm{as}\ x\to x_0.\\
\end{eqnarray*}
\end{lemma}
\begin{proof}
Recall that the Boltzmann sampler for $\cGcp$ results from the identity (block
decomposition, Equation~\eqref{eq:2conn})
$$
\cGcp=\Set\left(\cGbp\circ_L(\cZ_L\star\cGcp) \right).
$$
We want to use the computation rules (Figure~\ref{fig:comp_rules})
to obtain a recursive equation for $\Lambda\cGcp(x,y_0)$. Again,
according to Remark~\ref{rk:finite}, we have to check that $\Lambda\cGcp(x,y_0)$
is finite. 
\begin{claim}\label{claim:Gcp_finite}
For $0<x<x_0$, the quantity $\Lambda\cGcp(x,y_0)$ is finite.
\end{claim}
\ni\textit{Proof of the claim.} Let $\gamma\in\cGcp$, with 
 $\kappa_1,\ldots,\kappa_r$ the 2-connected blocks 
of $\gamma$. We have
$$
\Lambda\cGcp^{(\gamma)}(x,y_0)=2||\gamma||+\sum_{i=1}^r\Lambda\cGbp^{(\kappa_i)}(z,y_0),\ \ \mathrm{where}\ z=xG_1\ \!\!\!'(x,y_0).
$$
(The first term stands for the cost of choosing the degrees using a generator for  a Poisson law; note that the sum of the degrees over all the vertices of $\gamma$ 
is $2||\gamma||$.)
It is easily shown that there exists a constant $M$ such that 
$\Lambda\cGbp^{(\kappa)}(z,y_0)\leq M||\kappa||$ for any $\kappa\in\cGbp$ 
(using the fact that such a bound holds for $\Lambda\cD^{(\kappa)}(z,y_0)$ 
and that $\Gamma\cGbp(z,y_0)$ is obtained from $\Gamma\cD(z,y_0)$ via a
simple rejection step). Therefore $\Lambda\cGcp^{(\gamma)}(x,y_0)\leq C||\gamma||$,
with $C=2+M$. We conclude that 
$$
\Lambda\cGcp(x,y_0)\leq \frac{C}{G_1\ \!\!\!'(x,y_0)}\sum_{\gamma\in\cGcp}||\gamma||\frac{x^{|\gamma|}}{|\gamma|!}y_0^{||\gamma||},
$$
which is $O(1)$ since it converges to the constant $Cy_0\partial_yG_1\ \!\!\!'(x,y_0)/G_1\ \!\!\!'(x,y_0)$.
\qedclaim

The computation rules (Figure~\ref{fig:comp_rules}) yield
$$
\Lambda\cGcp(x,y_0)=G_2\ \!\!\!'(z,y_0)\cdot\left(\Lambda\cGbp(z,y_0)+|\cGbp|_{(z,y_0)}\cdot\Lambda\cGcp(x,y_0) \right)\ \ \mathrm{where}\ z=xG_1\ \!\!\!'(x,y_0),
$$
so that
$$
\Lambda\cGcp(x,y_0)=\frac{G_2\ \!\!\!'(z,y_0)\Lambda\cGbp(z,y_0)}{1-G_2\ \!\!\!'(z,y_0)\cdot|\cGbp|_{(z,y_0)}}.
$$
Similarly as in the transition from 3-connected planar graphs to networks, 
we use the important point, proved in~\cite{gimeneznoy},  that the composition scheme
to go from 2-connected to connected planar graphs is critical.
This means that, when $x\to x_0$, the quantity $z=xG_1\ \!\!\!'(x,y_0)$
(which is the change of variable from 2-connected to connected) converges to 
a positive constant $z_0$ such that $(z_0,y_0)$ is a singular point of $\cG_2$.
Hence, according to Lemma~\ref{lem:comp_LG2}, $\Lambda\cGbp(z,y_0)$ is $O(1)$
as $x\to x_0$. Moreover, as the class $\cGbp$ is $3/2$-singular, 
the series $G_2\ \!\!\!'(z,y_0)$ 
and the expected size 
$|\cGbp|_{(z,y_0)}$ converge to positive
constants that are denoted respectively $G_2\ \!\!\!'(z_0,y_0)$ and $|\cGbp|_{(z_0,y_0)}$. 
We have shown that the numerator
of $\Lambda\cGcp(x,y_0)$ is $O(1)$ and that the denominator converges
as $x\to x_0$. To prove that  $\Lambda\cGcp(x,y_0)$ is $O(1)$, it remains to
check that the denominator does not converge to $0$, i.e., to prove that 
$G_2\ \!\!\!'(z_0,y_0)\cdot |\cGbp|_{(z_0,y_0)}\neq 1$. 

To show this, we use the simple
trick that the expected complexity and expected size of Boltzmann samplers
satisfy similar computation rules. Indeed, from Equation~\eqref{eq:2conn}, it is easy to derive
the equation
$$
|\cGcp|_{(x,y_0)}=G_2\ \!\!\!'(z,y_0)\cdot|\cGbp|_{(z,y_0)}\cdot\left(|\cGcp|_{(x,y_0)}+1\right)\ \ \mathrm{where}\ z=xG_1\ \!\!\!'(x,y_0),
$$
either using the formula $|\cC|_{(x,y)}=\partial_x C(x,y)/C(x,y)$,
or simply by interpreting what happens during a call to $\Gamma\cGcp(x,y)$
(an average of $G_2\ \!\!\!'(z,y_0)$ blocks are attached at the root-vertex, 
each block has 
average size $|\cGbp|_{(z,y_0)}$ and carries a connected component of average
size $(|\cGcp|_{(x,y_0)}+1)$  at each non-root vertex). Hence
$$
|\cGcp|_{(x,y_0)}=\frac{G_2\ \!\!\!'(z,y_0)\cdot |\cGbp|_{(z,y_0)}}{1-G_2\ \!\!\!'(z,y_0)\cdot |\cGbp|_{(z,y_0)}}.
$$
Notice that this is the same expression as $\Lambda\cGcp(x,y_0)$, except
for $|\cGbp|_{(z,y_0)}$ replacing $\Lambda\cGbp(z,y_0)$ in the numerator.
The important point is that we already know that  $|\cGcp|_{(x,y_0)}$ 
converges as $x\to x_0$, since the class $\cGcp$ is $3/2$-singular (see Lemma~\ref{lem:sing_planar}).
Hence $G_2\ \!\!\!'(z_0,y_0)\cdot |\cGbp|_{(z_0,y_0)}$ has to be different from $1$ (more precisely, it is strictly less than $1$), which concludes the proof.
\end{proof}

\begin{lemma}[bi-derived connected planar graphs]\label{lem:comp_G1pp}
Let $(x_0,y_0)$ be a singular point of $\cG_1$. Then,  the expected complexity
of the Boltzmann sampler for $\cGcp'$---described in Section~\ref{sec:sampCp}---satisfies
\begin{eqnarray*}
\Lambda \cGcp'(x,y_0)&=&O\ \left((x_0-x)^{-1/2}\right)\ \ \mathrm{as}\ x\to x_0.\\
\end{eqnarray*}
\end{lemma}
\begin{proof}
The proof for  $\Lambda \cGcp'(x,y_0)$ is easier than for $\Lambda \cGcp(x,y_0)$.
Recall that $\Gamma \cGcp'(x,y_0)$ is obtained from the identity
$$
\cGcp'=\left(\cGcp+\cZ_L\star\cGcp' \right)\star\cGbp'\circ_L(\cZ_L\star\cGcp)\star\cGcp.
$$
At first one easily checks (using similar arguments as in Claim~\ref{claim:Gcp_finite})
that $\Lambda \cGcp'(x,y_0)$ is finite. 
 Using the computation rules given in Figure~\ref{fig:comp_rules}, we obtain, writing as usual $z=xG_1\ \!\!\!'(x,y_0)$,
\begin{eqnarray*}
\Lambda\cGcp'(x,y_0)&\!\!\!=\!\!\!&1+\frac{G_1\ \!\!\!'(x,y_0)}{G_1\ \!\!\!'(x,y_0)\!+\!xG_1\ \!\!\!''(x,y_0)}\Lambda\cGcp(x,y_0)+\frac{xG_1\ \!\!\!''(x,y_0)}{G_1\ \!\!\!'(x,y_0)\!+\!xG_1\ \!\!\!''(x,y_0)}\Lambda\cGcp'(x,y_0)\\
&&+ \Lambda\cGbp'(z,y_0)+|\cGbp'|_{(z,y_0)}\cdot\Lambda\cGcp(x,y_0)+\Lambda\cGcp(x,y_0).
\end{eqnarray*}
Hence
$$\Lambda\cGcp'(x,y_0)=a(x,y_0)\cdot(1+b(x,y_0)\cdot \Lambda\cGcp(x,y_0)+\Lambda\cGbp'(z,y_0)+|\cGbp'|_{(z,y_0)}\cdot\Lambda\cGcp(x,y_0)),$$
where 

\vspace{-.3cm}

$$a(x,y_0)=\frac{G_1\ \!\!\!'(x,y_0)+xG_1\ \!\!\!''(x,y_0)}{G_1\ \!\!\!'(x,y_0)},\ \ \ b(x,y_0)=\frac{2G_1\ \!\!\!'(x,y_0)+xG_1\ \!\!\!''(x,y_0)}{G_1\ \!\!\!'(x,y_0)+xG_1\ \!\!\!''(x,y_0)}.$$ As the classes $\cGcp$ and $\cGcp'$ are respectively $3/2$-singular
and $1/2$-singular, the series $a(x,y_0)$ and $b(x,y_0)$ converge when $x\to x_0$.
As $\cGbp'$ is $1/2$-singular,   
$|\cGbp'|_{(z,y_0)}$ is $O((z_0-z)^{-1/2})$ when $z\to z_0$.
Moreover, according to Lemma~\ref{lem:comp_LG2}, $\Lambda\cGbp'(z,y_0)$ is $O((z_0-z)^{-1/2})$.
Next we use the fact that the change of variable from 2-connected to connected is critical. 
Precisely, as proved in~\cite{BeGa}, when $x\to x_0$ and when $z$ and $x$ are related by $z=xG_1\ \!\!\!'(x,y_0)$,
we have $z_0-z\sim \lambda\cdot (x_0-x)$, with $\lambda:=\lim \mathrm{d}z/\mathrm{d}x=x_0G_1\ \!\!\!''(x_0,y_0)+G_1\ \!\!\!'(x_0,y_0)$.
Hence, $|\cGbp'|_{(z,y_0))}$ and $\Lambda\cGbp'(z,y_0)$ are $O((x_0-x)^{-1/2})$.  
In addition, we have proved in Lemma~\ref{lem:comp_G1p}  that 
$\Lambda\cGcp(x,y_0)$ is $O(1)$.
We conclude that $\Lambda\cGcp'(x,y_0)$ is $O((x_0-x)^{-1/2})$.
\end{proof}

\begin{lemma}[connected planar graphs]\label{lem:comp_G1}
Let $(x_0,y_0)$ be a singular point of $\cG_1$. Then,  
the expected complexity
of the Boltzmann sampler for $\cG_1$---described in Section~\ref{sec:conn2conn}---satisfies
$$
\Lambda \cG_1(x,y_0)=O\ (1)\ \ \mathrm{as}\ x\to x_0.
$$
\end{lemma}
\begin{proof}
As described in Section~\ref{sec:conn2conn}, the sampler $\Gamma\cG_1(x,y)$ computes 
$\gamma\leftarrow\Gamma\cGcp(x,y)$ and keeps $\gamma$ with probability $1/(|\gamma|+1)$.
Hence the probability of success at each attempt is
$$
\pacc=\frac{1}{G_1\ \!\!\!'(x,y_0)}\sum_{\gamma\in\cGcp}\frac{1}{|\gamma|+1}\frac{x^{|\gamma|}}{|\gamma|!}y_0^{||\gamma||}=\frac{1}{G_1\ \!\!\!'(x,y_0)}\sum_{\gamma\in\cGcp}\frac{x^{|\gamma|}}{(|\gamma|+1)!}y_0^{||\gamma||}.
$$
Recall that for any class $\cC$, $\cC'_{n,m}$ identifies to $\cC_{n+1,m}$. Hence
$$
\pacc=\frac{1}{G_1\ \!\!\!'(x,y_0)}\sum_{\gamma\in\cG_1}\frac{x^{|\gamma|-1}}{|\gamma|!}y_0^{||\gamma||}=\frac{G_1(x,y_0)}{xG_1\ \!\!\!'(x,y_0)}.
$$
In addition, by Lemma~\ref{lem:target}, 
$\Lambda\cG_1(x,y_0)=\Lambda\cGcp(x,y_0)/\pacc$. As the classes $\cG_1$ and $\cGcp$ are
respectively $5/2$-singular and $3/2$-singular, both series $G_1(x,y_0)$ and $G_1\ \!\!\!'(x,y_0)$
converge to positive constants when $x\to x_0$. 
Hence $\pacc$ converges to a positive constant as well.
In addition, $\Lambda\cGcp(x,y_0)$ is $O(1)$ by Lemma~\ref{lem:comp_G1p}. 
Hence $\Lambda\cG_1(x,y_0)$ is  also $O(1)$.
\end{proof}

\subsubsection{Complexity of the Boltzmann samplers for planar graphs}\label{sec:comp_planar}

\begin{lemma}[planar graphs]
Let $(x_0,y_0)$ be a singular point of $\cG$. Then, the expected complexities
of the Boltzmann samplers for $\cG$,  $\cG'$ and $\cG''$---described in Section~\ref{sec:planconn} and~\ref{sec:sampGp}---satisfy, as $x\to x_0$: 
\begin{eqnarray*}
\Lambda \cG(x,y_0)&=&O\ (1),\\
\Lambda \cG'(x,y_0)&=&O\ (1),\\
\Lambda \cG''(x,y_0)&=&O\ ((x_0-x)^{-1/2}).
\end{eqnarray*}
\end{lemma}
\begin{proof}
Recall that $\Gamma\cG(x,y)$ is obtained from $\Gamma\cG_1(x,y)$
using the identity
$$
\cG=\Set(\cG_1),
$$
hence $\Lambda \cG(x,y_0)=G_1(x,y_0)\cdot\Lambda\cG_1(x,y_0)$. When 
$x\to x_0$, $G_1(x,y_0)$ converges (because $\cG_1$ is $5/2$-singular) 
and $\Lambda\cG_1(x,y_0)$
is $O(1)$ (by Lemma~\ref{lem:comp_G1}). Hence $\Lambda \cG(x,y_0)$ is $O(1)$.

Then, $\Gamma\cG'(x,y)$ is obtained from $\Gamma\cGcp(x,y)$ and $\Gamma\cG(x,y)$
using the identity
$$
\cG'=\cGcp\star\cG.
$$
Hence $\Lambda\cG'(x,y_0)=\Lambda\cGcp(x,y_0)+\Lambda\cG(x,y_0)$. When 
$x\to x_0$, $\Lambda\cGcp(x,y_0)$ is $O(1)$ (by Lemma~\ref{lem:comp_G1p}) and 
$\Lambda\cG(x,y_0)$
is $O(1)$, as proved above. Hence $\Lambda \cG'(x,y_0)$ is $O(1)$.

Finally, $\Gamma\cG''(x,y)$ is obtained from $\Gamma\cGcp'(x,y)$,
$\Gamma\cGcp(x,y)$, $\Gamma\cG'(x,y)$, and $\Gamma\cG(x,y)$
using the identity
$$
\cG''=\cGcp'\star\cG+\cGcp\star\cG'.
$$
Hence 
$$
\Lambda\cG''(x,y_0)=1+\frac{a}{a+b}\left(\Lambda\cGcp'(x,y_0)+\Lambda\cG(x,y_0)\right)+\frac{b}{a+b}\left(\Lambda\cGcp(x,y_0)+\Lambda\cG'(x,y_0)\right),
$$
where $a=G_1\ \!\!\!''(x,y_0)G(x,y_0)$ and $b=G_1\ \!\!\!'(x,y_0)G'(x,y_0)$. Thus
$$ 
\Lambda\cG''(x,y_0)\leq 1+\Lambda\cGcp'(x,y_0)+\Lambda\cG(x,y_0)+\Lambda\cGcp(x,y_0)+\Lambda\cG'(x,y_0).
$$
When $x\to x_0$, $\Lambda \cGcp'(x,y_0)$ is $O((x_0-x)^{-1/2})$ (by Lemma~\ref{lem:comp_G1pp}),
$\Lambda \cGcp(x,y_0)$ is $O(1)$ (by Lemma~\ref{lem:comp_G1p}), 
and $\Lambda \cG'(x,y_0)$ and  
$\Lambda \cG(x,y_0)$ are 
$O(1)$, as proved above. Hence 
$\Lambda \cG''(x,y_0)$ is $O((x_0-x)^{-1/2})$, which concludes the proof.
\end{proof}
This concludes the proof of the expected complexities of our random samplers.
(Recall that, thanks to Claim~\ref{claim:eq}, the proof has been reduced to proving
 the asymptotic estimate $\Lambda\cG''(x,y_0)=O((x_0-x)^{-1/2})$.) 
 
 \vspace{0.2cm}

\noindent\emph{Acknowledgements.}
I am very grateful to Philippe Flajolet for his encouragements and for several  
 corrections and suggestions that
led to a significant improvement of the presentation of the results.
I greatly thank the anonymous referee for an extremely 
detailed and insightful report, which led to a major revision of an earlier version of 
the article.
I have also enjoyed fruitful discussions with Gilles
Schaeffer,  Omer Gim\'enez and 
Marc Noy, in particular regarding the implementation of the algorithm.

   \bibliographystyle{plain}
\bibliography{Fusy_biblio}

\end{document}